\documentclass{amsart}

\usepackage{amscd}

\usepackage{graphicx}
\newtheorem{theorem}{Theorem}[section]
\newtheorem{lemma}[theorem]{Lemma}

\theoremstyle{definition}
\newtheorem{definition}[theorem]{Definition}
\newtheorem{example}[theorem]{Example}

\theoremstyle{proposition}
\newtheorem{proposition}[theorem]{Proposition}
\newtheorem*{remark}{Remark}
\theoremstyle{corollary}
\newtheorem{corollary}[theorem]{Corollary}

\numberwithin{subsection}{section}
\numberwithin{equation}{section}

\begin{document}

\title{Analytic study of singular curves}

\author{Yukitaka \textsc{Abe}}
\address{Department of Mathematics, University of Toyama, Toyama 930-8555, Japan}
\email{abe@sci.u-toyama.ac.jp}





\begin{abstract}
We study singular curves from analytic point of view.
We give completely analytic proofs for the Serre duality and a generalized Abel's theorem.
We also reconsider Picard varieties, Albanese varieties and generalized
Jacobi varieties of
singular curves analytically. 
We call an Albanese variety considered
as a complex Lie group an analytic Albanese variety.
We investigate them in detail.
For a non-singular curve (a compact Riemann surface) $X$,
there is the relation between the meromorphic function fields on
$X$ and on its Jacobi variety $J(X)$. We try to extend this
relation to the case of singular curves.
\end{abstract}

\maketitle

\section{Introduction}
Jacobi varieties and Picard varieties of compact Riemann surfaces
and their relation are a classical theme.
In the middle of the twenty century, a generalization of these
subjects to singular curves was studied by several authors. We think that studies
in this direction had been influenced by the monograph of Severi
\lq \lq Funzioni Quasi Abeliane". The treatment of the above studies
was completely algebraic. A generalized Jacobi variety of a singular curve
was first defined algebraically by Rosenlicht.
We analytically consider generalized Jacobi varieties.
A commutative algebraic group over ${\mathbb C}$ is
analytically a commutative complex Lie group.
We obtain some information about analytic properties of them.\par
We treat all of materials completely analytic.
We use some notations in Serre's book \lq \lq Groupes
alg\'ebriques et corps de classes". 
Let $X$ be a projective, irreducible, and non-singular algebraic
curve (i.e. a compact Riemann surface), and let $S$ be a finite
subset of $X$. We denote by ${\mathfrak m}$ a modulus with support
$S$. Endowing a suitable equivalence relation on $S$, we obtain
a singular curve $X_{\mathfrak m}$ by ${\mathfrak m}$.
Rational functions (i.e. meromorphic functions) $\varphi $ with
$\varphi \equiv 1\ {\rm mod}\ {\mathfrak m}$ were considered
in the above Serre's book. A generalized Abel's theorem was formulated 
in terms of these functions. The condition \lq \lq $\varphi \equiv 1\ 
{\rm mod}\ {\mathfrak m}$" means that $\varphi $ takes the common
value $1$ at each point of $S$.
If the set of singularities consists of the only one point, this
condition is appropriate. However it becomes considering
very special functions in the general case. 
We modify this part.
We consider meromorphic
functions which take each value $c_Q$ at each singular point $Q$
in $X_{{\mathfrak m}}$\,.
We develop the theory using these functions. 
We give a new formulation of a
generalized Abel's theorem in terms of such natural concept.
And we give a completely analytic proof of it.
We also reconsider Picard varieties and Albanese varieties of
singular curves.  We consider Albanese varieties analytically.
An Albanese variety considered as a complex Lie group is called
an analytic Albanese variety. 
We show that an analytic Albanese variety has the
universality property.
\par
The abelian function
field on the Jacobi variety $J(X)$ of $X$ is generated by
the fundamental abelian functions which are constructed by
generators of the meromorphic function field on $X$.
We consider a generalization of
this relation to the case of singular curves.\par
The paper is organized as follows. In Section 2 we state the
construction of singular curves from a compact Riemann surface.
Considering an equivalence relation $R$ on $S$, we set
$\overline{S} = S/R$. If ${\mathfrak m}$ is a modulus with
support $S$, then $(X \setminus S)\cup \overline{S}$ has the
structure of $1$-dimensional compact reduced complex space.
It is a singular curve $X_{{\mathfrak m}}$ defined by
${\mathfrak m}$.
For the convenience of readers,
we also give a proof of the Riemann-Roch theorem there.\par
Section 3 is devoted to a completely analytic proof of the
Serre duality. Here we introduce two special sheaves for
the analytic proof. 
\par
In Section 4, we give a new formulation
of a generalized Abel's theorem after defining a multiconstant
on the singularities. We also need a special sheaf for its proof.
We emphasize that our proof is completely analytic.\par
Divisor classes, Picard varieties and Albanese varieties of
singular curves are discussed in Section 5.
An Albanese variety is analytically considered as a
commutative complex Lie group, which we call an analytic
Albanese variety. We show that an analytic Albanese variety
is the product of copies of ${\mathbb C}$, copies of
${\mathbb C}^{*}$ and a quasi-abelian variety of kind 0 in
Section 5.6. Let ${\rm Alb}^{an}(X_{{\mathfrak m}})$ be
the analytic Albanese variety of a singular curve 
$X_{{\mathfrak m}}$ constructed from a compact Riemann
surface $X$ of genus $g$. The period map
$\varphi : X \setminus S \longrightarrow {\rm Alb}^{an}(X_
{{\mathfrak m}})$ admits ${\mathfrak m}$ for a modulus,
and is a holomorphic embedding if $g \geq 1$.
Furthermore we prove that the map $\varphi :
(X \setminus S)^{(\pi )} \rightarrow {\rm Alb}^{an}(X_{{\mathfrak m}})$
is surjective, where $\pi $ is the genus of $X_{{\mathfrak m}}$ and
$(X \setminus S)^{(\pi )}$ is the symmetric product of
degree $\pi $ (Theorem 5.17). As a corollary we obtain that
the divisor classes of degree 0 on $X_{{\mathfrak m}}$ and
${\rm Alb}^{an}(X_{{\mathfrak m}})$ are isomorphic.
In Section 5.8 the universality of analytic Albanese varieties
is proved.\par
In Section 6 further properties of analytic Albanese varieties
are studied. We determine period matrices of analytic Albanese
varieties of general singular curves in Section 6.1.
In Section 6.2 we fully study curves with nodes in the case
$\pi = 2$. Let $X_{{\mathfrak m}}$ and $X_{{\mathfrak m}'}$
be two singular curves of genus 2 with a node constructed
from a complex torus $X$. Then we determine the necessary and
sufficient condition for the biholomorphic equivalence of
$X_{{\mathfrak m}}$ and $X_{{\mathfrak m}'}$.
In Section 6.3 we show that for any singular curve $X_{{\mathfrak m}}$
whose singularity is the only cusp we have
${\rm Alb}^{an}(X_{{\mathfrak m}}) \cong J(X)
\times {\mathbb C}$.\par
The last section is devoted to the study of meromorphic
function fields. In the non-singular case there is the
connection between the two meromorphic function fields
${\rm Mer}(X)$ and ${\rm Mer}(J(X))$ on $X$ and on its Jacobi
variety $J(X)$ respectively.
The field ${\rm Mer}(J(X))$ is generated by the fundamental
abelian functions belonging to ${\rm Mer}(X)$. 
We try to generalize
this
relation to the case of singular curves.

\section{Singular Curves}

%

\subsection{Construction of singular curves}
Let $X$ be an irreducible non-singular complex projective algebraic
curve (i.e. a compact Riemann surface).
We denote by ${\mathcal O}_X$ the structure sheaf on $X$.
Let $S$ be a finite subset of $X$. We give an equivalence
relation $R$ on $S$. We define the quotient set
$\overline{S} := S/R$ of $S$ by $R$. We set
$$\overline{X} := (X \setminus S)\cup \overline{S}.$$
We induce to $\overline{X}$ the quotient topology
by the canonical projection $\rho : X \longrightarrow \overline{X}$.
Then $\overline{X}$ is a compact Hausdorff space.
We define a modulus with support $S$, according to Serre(\cite{ref16}).
\begin{definition}
A modulus ${\mathfrak m}$ with support $S$ is the data of an
integer ${\mathfrak m}(P)>0$ for each point $P \in S$.
\end{definition}
A modulus ${\mathfrak m}$ with support $S$ is also considered as
a positive divisor on $X$. We use the same notation
$${\mathfrak m} = \sum_{P \in S} {\mathfrak m}(P)P.$$
We may assume $\deg \ {\mathfrak m} 
\geq 
2$.\par
Let ${\rm Mer}(X)$ be the field of meromorphic functions on $X$.
For any $f \in {\rm Mer}(X)$ and any $P \in X$, we denote
by ${\rm ord }_P(f)$ the order of $f$ at $P$.
\begin{definition}
Let $f, g \in {\rm Mer}(X)$. We write
$$f \equiv g\quad {\rm mod}\ {\mathfrak m}$$
if ${\rm ord}_P(f-g) 
\geq 
{\mathfrak m}(P)$ for any $P \in S$.
\end{definition}
Let $\rho _* {\mathcal O}_X$ be the direct image of
${\mathcal O}_X$ by the projection
$\rho : X \longrightarrow \overline{X}$.
For any $Q \in \overline{S}$ we denote by ${\mathcal I}_Q$
the ideal of $(\rho _* {\mathcal O}_X)_Q$ formed by the
function $f$ with ${\rm ord}_P(f) \geq {\mathfrak m}(P)$ for
any $P \in \rho ^{-1}(Q)$.
We define a sheaf ${\mathcal O}_{\mathfrak m}$ on $\overline{X}$ by
$$
{\mathcal O}_{{\mathfrak m} , Q} :=
\begin{cases}
(\rho _*{\mathcal O}_X)_Q = {\mathcal O}_{X, Q}&
\mbox{if $Q \in X \setminus S$},\\
{\mathbb C} + {\mathcal I}_Q&
\mbox{if $Q \in \overline{S}$}.
\end{cases}$$
Then we obtain a 1-dimensional compact reduced complex
space $(\overline{X}, {\mathcal O}_{\mathfrak m})$, which
we denote by $X_{{\mathfrak m}}$.
\par
Conversely, any reduced and irreducible singular curve is obtained
as above.

\subsection{Genus of $X_{{\mathfrak m} }$}
For any $Q \in X_{{\mathfrak m} }$ we set
\begin{eqnarray*}
\delta _Q &:=& \dim ((\rho _*{\mathcal O}_X)_Q / {\mathcal O}_{{\mathfrak m}, Q})
\nonumber \\
& = & 
\begin{cases}
\dim ({\mathcal O}_{X, Q}/{\mathcal O}_{X, Q}) = 0&
\mbox{if $Q \in X \setminus S$,}\\
\dim ((\rho _*{\mathcal O}_X)_Q / ({\mathbb C} + {\mathcal I}_Q))
= \dim ((\rho _*{\mathcal O}_X)_Q /{\mathcal I}_Q) - 1&
\mbox{if $ Q \in \overline{S}$.}\\
\end{cases}\nonumber \\
\end{eqnarray*}
We set
$$\delta := \sum_{Q \in X_{{\mathfrak m} }} \delta _Q = \deg {\mathfrak m}
- \# \overline{S}.$$
Let $g$ be the genus of $X$. We define the genus $\pi $ of $X_{{\mathfrak m} }$ by
$$\pi := g + \delta.$$

\subsection{Divisors prime to $S$}
A divisor $D$ on $X$ is written as
$$D = \sum_{P \in X}D(P)P,\quad D(P) \in {\mathbb Z},$$
where $D(P) = 0$ except a finite number of $P \in X$.
\begin{definition}
A divisor $D$ on $X$ is said to be prime to $S$ if
$D(P) = 0$ for $P \in S$.
\end{definition}
We denote by ${\rm Div}(X_{\mathfrak m} )$ the group of divisors
prime to $S$. We have a homomorphism $\deg : {\rm Div}(X_{\mathfrak m})
\longrightarrow {\mathbb Z}$.
Let ${\rm Mer}(X_{\mathfrak m})$ be the field of meromorphic functions
on $X_{{\mathfrak m} }$.
We may consider ${\rm Mer}(X_{\mathfrak m})\subset {\rm Mer}(X)$ by
the map ${\rm Mer}(X_{\mathfrak m})\longrightarrow {\rm Mer}(X),$
$f \longmapsto f \circ \rho$.
Let $f \in {\rm Mer}(X_{\mathfrak m})$.
It defines the divisor
$$(f)= \sum_ {Q \in X_{{\mathfrak m} }}{\rm ord}_Q(f)Q,$$
where ${\rm ord}_Q(f) = \sum _{P \in \rho ^{-1}(Q)} {\rm ord}_P(f \circ \rho )$.
\begin{definition}
Let $D_1, D_2 \in {\rm Div}(X_{\mathfrak m})$.
We write $D_1 \sim D_2$ if there exists $f \in {\rm Mer}(X_{\mathfrak m})$
such that $D_1 - D_2 = (f)$.
\end{definition}
Since the relation \lq\lq $\sim $" is an equivalence relation on ${\rm Div}(X_{\mathfrak m})$,
we have the group of divisor classes $\overline{{\rm Div}(X_{\mathfrak m})}:= {\rm Div}(X_{\mathfrak m})/\sim $.
\begin{definition}
A divisor $D \in {\rm Div}(X_{\mathfrak m})$ is called a positive
divisor if $D(Q) \geq 0$ for any $Q \in X \setminus S$, and we write it
$D \geq 0$.
Furthermore, $D$ is said to be strictly positive if
$D \geq 0$ and $D \not= 0$, in this case we write $D > 0$.
\end{definition} 
For any $D_1, D_2 \in {\rm Div}(X_{\mathfrak m})$ we write
$D_1 \geq D_2$ (resp. $D_1 > D_2$) if
$D_1 - D_2 \geq 0$ (resp. $D_1 - D_2 >0$).

\subsection{A sheaf associated with a divisor}
Let $D \in {\rm Div}(X_{\mathfrak m}) \subset {\rm Div}(X)$.
We define
$$L(D) := \{ f \in {\rm Mer}(X) ; (f) \geq - D \}.$$
We denote by ${\mathcal L}(D)$ the sheafication of $L(D)$ on $X$.
Let ${\mathcal L}_{\mathfrak m}(D)$ be the sheaf on $X_{{\mathfrak m} }$ 
defined by
$${\mathcal L}_{\mathfrak m}(D)_Q := 
\begin{cases}
{\mathcal O}_{{\mathfrak m}, Q} &
\mbox{if $Q \in \overline{S}$,}\\
{\mathcal L}(D)_Q & \mbox{if $Q \in X \setminus S$}.\\
\end{cases}$$
Then ${\mathcal L}_{\mathfrak m}(D)$ is a coherent sheaf,
for it is locally isomorphic to ${\mathcal O}_{{\mathfrak m} }$\ .
\begin{proposition}
Let $D_1, D_2 \in {\rm Div}(X_{\mathfrak m}).$
If $D_2 \geq D_1$, then
$$H^{0}(X_{\mathfrak m}, {\mathcal L}_{\mathfrak m}(D_1))
\subset H^{0}(X_{\mathfrak m}, {\mathcal L}_{\mathfrak m}(D_2)).$$
\end{proposition}
\begin{proof}
There exists $D \in {\rm Div}(X_{\mathfrak m})$ with $D \geq 0$ such that
$D_2 = D_1 + D.$
For any $f \in H^{0}(X_{\mathfrak m}, {\mathcal L}_{\mathfrak m}(D_1))$,
we have
$$(f) \geq - D_1 = -D_2 + D \geq - D_2.$$
Then $f \in H^{0}(X_{\mathfrak m}, {\mathcal L}_{\mathfrak m}(D_2)).$
\end{proof}

Since ${\mathcal L}_{\mathfrak m}(0) = {\mathcal O}_{\mathfrak m},$
the following proposition is immediate.
\begin{proposition}
We have
$$H^{0}(X_{\mathfrak m}, {\mathcal L}_{\mathfrak m}(0)) = {\mathbb C}.$$
\end{proposition}
\begin{proposition}
Let $D \in {\rm Div}(X_{\mathfrak m}).$ If $\deg\ D < 0$, then
$$H^{0}(X_{\mathfrak m}, {\mathcal L}_{\mathfrak m}(D)) = 0.$$
\end{proposition}
\begin{proof}
We suppose that there exists $f \in H^{0}(X_{\mathfrak m}, {\mathcal L}_{\mathfrak m}(D))$ with $f \not= 0$.
Since $(f) \geq - D$, we have
$$0 = \deg (f) \geq - \deg \, D >0.$$
This is a contradiction.
\end{proof}
\begin{proposition}
We have
$$\dim H^1(X_{\mathfrak m}, {\mathcal O}_{\mathfrak m}) = \pi .$$
\end{proposition}
\begin{proof}
By an exact sequence
$$0 \longrightarrow {\mathcal O}_{\mathfrak m}\longrightarrow
\rho_{*}{\mathcal O}_X \longrightarrow
\rho_{*}{\mathcal O}_X/{\mathcal O}_{\mathfrak m}\longrightarrow 0$$
we obtain the following long cohomology sequence
\begin{eqnarray*}
0 &\longrightarrow& H^{0}(X_{\mathfrak m},{\mathcal O}_{\mathfrak m})
\longrightarrow H^{0}(X_{\mathfrak m}, \rho_{*}{\mathcal O}_X)
\longrightarrow H^{0}(X_{\mathfrak m}, \rho_{*}{\mathcal O}_X/{\mathcal O}_{\mathfrak m})\\
& \longrightarrow & 
H^{1}(X_{\mathfrak m},{\mathcal O}_{\mathfrak m})
\longrightarrow H^{1}(X_{\mathfrak m}, \rho_{*}{\mathcal O}_X)
\longrightarrow H^{1}(X_{\mathfrak m}, \rho_{*}{\mathcal O}_X/{\mathcal O}_{\mathfrak m}) \longrightarrow \cdots.
\end{eqnarray*}
We have
$$H^q(X_{\mathfrak m}, \rho_{*}{\mathcal O}_X)\cong
H^q(X, {\mathcal O}_X)\quad
\mbox{for all $q \geq 0$},$$
especially
$$H^0(X_{\mathfrak m}, {\mathcal O}_{\mathfrak m}) = {\mathbb C}
= H^0(X_{\mathfrak m}, \rho_{*}{\mathcal O}_{X})
\cong H^0(X, {\mathcal O}_X).$$
And we have $H^1(X_{\mathfrak m}, \rho_{*}{\mathcal O}_X/{\mathcal O}_{\mathfrak m})
= 0$ for ${\rm supp}(\rho_{*}{\mathcal O}_X/{\mathcal O}_{\mathfrak m})
= \overline{S}.$
From
$$\rho_{*}{\mathcal O}_{X}/{\mathcal O}_{\mathfrak m}\cong
\bigoplus_{Q \in \overline{S}}{\mathbb C}_{Q}^{\delta _Q}$$
it follows that
$$\dim H^0(X_{\mathfrak m}, \rho_{*}{\mathcal O}_X/
{\mathcal O}_{\mathfrak m}) = \sum_{Q \in \overline{S}}\delta _Q = \delta .$$
Then we obtain
$$\dim H^1(X_{\mathfrak m}, {\mathcal O}_{\mathfrak m}) =
\dim H^1(X, {\mathcal O}_{X}) + \delta
= g + \delta = \pi$$
by the exact sequence
$$0 \longrightarrow {\mathbb C}^{\delta } \longrightarrow
H^1(X_{\mathfrak m}, {\mathcal O}_{\mathfrak m}) \longrightarrow
H^1(X, {\mathcal O}_X) \longrightarrow 0.$$
\end{proof}

\subsection{Riemann-Roch Theorem (first version)}
\begin{theorem}
Let $X, S, {\mathfrak m}, X_{\mathfrak m}$ be as above.
Let $D \in {\rm Div}(X_{\mathfrak m})$. Then,
$H^0(X_{\mathfrak m}, {\mathcal L}_{\mathfrak m}(D))$ and
$H^1(X_{\mathfrak m}, {\mathcal L}_{\mathfrak m}(D))$ are
finite dimensional, and we have
$$\dim H^0(X_{\mathfrak m}, {\mathcal L}_{\mathfrak m}(D))
- \dim H^1(X_{\mathfrak m}, {\mathcal L}_{\mathfrak m}(D))
= \deg\, D + 1 - \pi .$$
\end{theorem}
\begin{proof}
(I) The case $D = 0$.\\
By Proposition 2.7 we have   $H^0(X_{\mathfrak m}, {\mathcal L}_{\mathfrak m}(0))
= {\mathbb C}.$
Also we have $\dim H^1(X_{\mathfrak m}, {\mathcal O}_{\mathfrak m}) = \pi $
by Proposition 2.9.
Since ${\mathcal L}_{\mathfrak m}(0) = {\mathcal O}_{\mathfrak m}$\;,
we obtain the assertion.\\
(II) The case $D \geq 0$.\\
We prove the assertion by induction on $d := \deg\, D.$
When $d=0$, it is the case (I).\par
Let $d \geq 1.$ We assume that it is true for $0 \leq \deg\, D \leq d-1.$
Let $D = P_1 + \cdots + P_{d-1} +P$. We set
$D^{-} := P_1 + \cdots + P_{d-1}.$ Then we have an exact sequence
$$0 \longrightarrow {\mathcal L}_{\mathfrak m}(D^{-})
\longrightarrow {\mathcal L}_{\mathfrak m}(D) \longrightarrow
{\mathbb C}_{P} \longrightarrow 0.$$
It gives the following exact sequence of cohomology groups
\begin{eqnarray*}
0 & \longrightarrow & H^0(X_{\mathfrak m}, {\mathcal L}_{\mathfrak m}(D^{-}))
\longrightarrow H^0(X_{\mathfrak m}, {\mathcal L}_{\mathfrak m}(D))
\longrightarrow H^0(X_{\mathfrak m}, {\mathbb C}_P)\\
& \longrightarrow & H^1(X_{\mathfrak m}, {\mathcal L}_{\mathfrak m}(D^{-}))
\longrightarrow H^1(X_{\mathfrak m}, {\mathcal L}_{\mathfrak m}(D))
\longrightarrow 0.
\end{eqnarray*}
Since $H^1(X_{\mathfrak m}, {\mathcal L}_{\mathfrak m}(D^{-}))$ is
of finite dimension, so is $H^1(X_{\mathfrak m}, {\mathcal L}_{\mathfrak m}(D)).$ 
Furthermore we have
$$ \dim H^0(X_{\mathfrak m}, {\mathcal L}_{\mathfrak m}(D)) \leq
\dim H^0(X_{\mathfrak m}, {\mathcal L}_{\mathfrak m}(D^{-}))
+ 1 <  + \infty .$$
Therefore, all vector spaces in the above exact sequence are finite dimensional.
Hence we have
\begin{eqnarray*}
& & - \dim H^0(X_{\mathfrak m}, {\mathcal L}_{\mathfrak m}(D^{-}))
+ \dim H^0(X_{\mathfrak m}, {\mathcal L}_{\mathfrak m}(D))
- \dim H^0(X_{\mathfrak m}, {\mathbb C}_P)\\
& & + \dim H^1(X_{\mathfrak m}, {\mathcal L}_{\mathfrak m}(D^{-}))
- \dim H^1(X_{\mathfrak m}, {\mathcal L}_{\mathfrak m}(D))= 0.
\end{eqnarray*}
Thus we obtain by the assumption of induction
\begin{eqnarray*}
\lefteqn{\dim H^0(X_{\mathfrak m}, {\mathcal L}_{\mathfrak m}(D))
- \dim H^1(X_{\mathfrak m}, {\mathcal L}_{\mathfrak m}(D))}\\
& & = \dim H^0(X_{\mathfrak m}, {\mathcal L}_{\mathfrak m}(D^{-})) 
- \dim H^1(X_{\mathfrak m}, {\mathcal L}_{\mathfrak m}(D^{-})) + 1\\
& & = \deg D^{-} + 1 - \pi + 1 = \deg D + 1 - \pi .
\end{eqnarray*}
(III) The general case.\\
Let
$$D = P_1 + \cdots + P_n - Q_1 - \cdots - Q_m\quad
(P_i \not= Q_j).$$
We show the assertion by induction on $m$.
When $m =0$, it is the case (II).\par
Let $m \geq 1.$ We assume that it is true till $m-1.$
We set
$$D^+ := P_1 + \cdots + P_n - Q_1 - \cdots - Q_{m-1}$$
and $Q := Q_m$.
Since $D = D^+ - Q$, we have a short exact sequence
$$0 \longrightarrow {\mathcal L}_{\mathfrak m}(D) \longrightarrow
{\mathcal L}_{\mathfrak m}(D^+) \longrightarrow {\mathbb C}_Q
\longrightarrow 0.$$
Then we obtain the following long exact sequence of cohomology groups
\begin{eqnarray*}
0 &\longrightarrow & H^0(X_{\mathfrak m}, {\mathcal L}_{\mathfrak m}(D))
\longrightarrow  H^0(X_{\mathfrak m}, {\mathcal L}_{\mathfrak m}(D^+))
\longrightarrow  H^0(X_{\mathfrak m}, {\mathbb C}_Q)\\
&\longrightarrow &  H^1(X_{\mathfrak m}, {\mathcal L}_{\mathfrak m}(D))
\longrightarrow  H^1(X_{\mathfrak m}, {\mathcal L}_{\mathfrak m}(D^+))
\longrightarrow  H^1(X_{\mathfrak m}, {\mathbb C}_Q) = 0.
\end{eqnarray*}
Hence we have
$$\dim H^1(X_{\mathfrak m}, {\mathcal L}_{\mathfrak m}(D)) \leq
\dim H^1(X_{\mathfrak m}, {\mathcal L}_{\mathfrak m}(D^+)) +1
< + \infty$$
and
$$\dim H^0(X_{\mathfrak m}, {\mathcal L}_{\mathfrak m}(D)) \leq
\dim H^0(X_{\mathfrak m}, {\mathcal L}_{\mathfrak m}(D^+)) 
< + \infty .$$
As in the case (II) we obtain
\begin{eqnarray*}
\lefteqn{\dim H^0(X_{\mathfrak m}, {\mathcal L}_{\mathfrak m}(D))
- \dim H^1(X_{\mathfrak m}, {\mathcal L}_{\mathfrak m}(D))}\\
& & = \dim H^0(X_{\mathfrak m}, {\mathcal L}_{\mathfrak m}(D^{+})) 
- \dim H^1(X_{\mathfrak m}, {\mathcal L}_{\mathfrak m}(D^{+})) - 1\\
& & = \deg D^{+} + 1 - \pi - 1 = \deg D + 1 - \pi .
\end{eqnarray*}
\end{proof}

\section{Analytic Proof of the Serre Duality}

\subsection{Sheaf $\Omega _{\mathfrak m}$}
Let $U \subset X_{\mathfrak m}$ be an open set.
We define
$$\Omega _{\mathfrak m}(U) := \{
\text{a meromorphic $1$-form $\omega $ on $\rho ^{-1}(U)$
satisfying the following condition $(*)$}\}.$$
The condition $(*)$:\\
For any $Q \in U$ and any $f \in {\mathcal O}_{{\mathfrak m} , Q}$, we have
$$\sum_{P \in \rho ^{-1}(Q)}{\rm Res}_P(\rho ^{*}f\omega ) = 0.$$
Then a presheaf $\{\Omega _{\mathfrak m} (U), r^{U}_{V}\}$ defines
a sheaf $\Omega _{\mathfrak m}$ on $X_{\mathfrak m}$\,,
where $r^{U}_{V}$ is the restriction map.
This sheaf $\Omega _{\mathfrak m}$ is called the duality sheaf on $X_{\mathfrak m}$
in general.
\par
Let $\Omega $ be the sheaf of germs of holomorphic $1$-forms on $X$.
The direct image $\rho _{*}\Omega $ of $\Omega $ is considered as
an ${\mathcal O}_{\mathfrak m}$\,-submodule of $\Omega _{\mathfrak m}$\,.
Let $D \in {\rm Div}(X_{\mathfrak m}) \subset {\rm Div}(X)$.
For any open subset $W$ of $X$ we define
$$\Omega (D)(W) := \{ \text{
a meromorphic $1$-form $\eta $ on $W$ with $(\eta ) \geq - D$ on $W$}\}.$$
Then a presheaf $\{ \Omega (D)(W), r^{W}_{W'}\}$ gives a sheaf $\Omega (D)$
on $X$. We define a sheaf $\Omega _{\mathfrak m}(D)$ on $X_{\mathfrak m}$ by
$$\Omega _{\mathfrak m}(D)_Q :=
\left\{ 
\begin{array}{ll}
\Omega _{{\mathfrak m} ,Q} & \text{if\  $Q \in \overline{S}$,}\\
\Omega (D)_{Q}&\text{if\  $Q \in X \setminus S$}.\\
\end{array}
\right. $$
It is obvious that $\Omega _{\mathfrak m}(0) = \Omega _{\mathfrak m}$\,.
\par
The Serre duality is represented as follows.
\begin{theorem}[Serre duality]
For any $D \in {\rm Div}(X_{\mathfrak m})$ we have
$$H^0(X_{\mathfrak m}, \Omega _{\mathfrak m}(-D))\cong
H^1(X_{\mathfrak m}, {\mathcal L}_{\mathfrak m}(D))^{*},$$
where $H^1(X_{\mathfrak m}, {\mathcal L}_{\mathfrak m}(D))^{*}$
is the dual space of $H^1(X_{\mathfrak m}, {\mathcal L}_{\mathfrak m}(D))$.
\end{theorem}
The purpose of this section is to give a completely analytic
proof of it.
\subsection{Sheaves ${\mathcal E}^{(1,0)}_{\mathfrak m}$ and
${\mathcal E}^{(2)}_{\mathfrak m}$}
Let $U \subset X_{\mathfrak m}$ be an open set. We define
$${\mathcal E}^{(1,0)}_{\mathfrak m}(U) := \{
\text{a $C^{\infty}$ $(1,0)$-form $\omega $ on
$U \setminus (U \cap \overline{S})$ satisfying the following condition $(\star )$}\}.$$
The condition $(\star )$:\\
For any $Q \in U \cap \overline{S}$ we let
$\rho ^{-1}(Q) = \{ P_1, \dots , P_k\}$.
Take an open neighbourhood $V$ of $Q$ in $U$ such that
$$\rho^{-1}(V) = \bigsqcup_{i=1}^{k}V_i\:
\text{(disjoint union)}, \quad P_i \in V_i,$$
$(V_i, z_i)$ is a coordinate neighbourhood of $P_i$ with
$z_i(P_i) = 0$ and there exists a $C^{\infty }$ function
$\varphi _i$ on $V_i \setminus \{ P_i \}$ with
$\rho ^{*}\omega = \varphi _i dz_i$.
Then the limit $\lim _{P \to P_i}\varphi _i(P) z_i(P)^{{\mathfrak m}(P_i)}$
exists and there exists $\varepsilon _0 > 0$ such that for any
$C_i := \{ P \in V_i ; |z_i(P)| = r_i \}$ with
$0 < r_i < \varepsilon _0$ $(i = 1, \dots , k)$ we have
$$\sum _{i=1}^{k} \int _{C_i}\rho ^{*}\omega = 0,$$
where $C_i$ has the anticlockwise direction.\\
We denote by ${\mathcal E}^{(1,0)}_{{\mathfrak m} }$ the sheaf given by
a presheaf $\{ {\mathcal E}^{(1,0)}_{\mathfrak m}(U), r^{U}_{V}\}$
on $X_{\mathfrak m}$.
A sheaf ${\mathcal E}^{(2)}_{\mathfrak m}$ is the image of
${\mathcal E}^{(1,0)}_{\mathfrak m}$ by the operator $d$.
It is obvious from the definitions that the sequence
$$0 \longrightarrow \Omega _{\mathfrak m} \longrightarrow
{\mathcal E}^{(1,0)}_{\mathfrak m} \stackrel{d}\longrightarrow
{\mathcal E}^{(2)}_{\mathfrak m} \longrightarrow 0$$
is exact. The following proposition is immediate.
\begin{proposition}
The sheaves ${\mathcal E}^{(1,0)}_{\mathfrak m}$ and ${\mathcal E}^{(2)}_{\mathfrak m}$
are fine.
\end{proposition}
Then the above exact sequence is a fine resolution of $\Omega _{{\mathfrak m}}$.
Hence we obtain an isomorphism
$$H^1(X_{\mathfrak m}, \Omega _{\mathfrak m}) \cong
H^0(X_{{\mathfrak m}},{\mathcal E}^{(2)}_{\mathfrak m})/
d H^0(X_{{\mathfrak m}}, {\mathcal E}^{(1,0)}_{\mathfrak m}).$$
\subsection{Residues}
As shown in the preceding section, any $\xi \in H^1(X_{\mathfrak m}, \Omega _{\mathfrak m})$ has a representative $\omega \in H^0(X_{{\mathfrak m}}, {\mathcal E}^{(2)}_{\mathfrak m})$.
\begin{definition}
For any $\xi \in H^1(X_{\mathfrak m}, \Omega _{\mathfrak m})$ we define
its residue ${\rm Res}(\xi )$ by
$${\rm Res}(\xi ) := \frac{1}{2\pi \sqrt{-1}}
\iint_{X \setminus S}\rho ^{*}\omega ,$$
where $\omega \in H^0(X_{{\mathfrak m}}, {\mathcal E}^{(2)}_{\mathfrak m})$ is a
representative of $\xi $.
\end{definition}
\begin{proposition}
The residue ${\rm Res}(\xi )$ of $\xi $ is well-defined.
\end{proposition}
\begin{proof}
Let $\overline{S} = \{ Q_1, \dots , Q_N\}$. We set
$$\rho ^{-1}(Q_i) = \{ P_{i1}, \dots , P_{i,k(i)}\},\quad
i = 1, \dots , N.$$
Then we have
$$S = \{ P_{11}, \dots , P_{1,k(1)}, P_{21}, \dots , P_{N, k(N)} \}.$$
For any $i=1, \dots , N$, there exists an open neighbourhood $W_i$ of
$Q_i$ such that if we denote by $U_{ij}$ the connected component of
$\rho^{-1}(W_i)$ containing $P_{ij}$, then
$$\rho^{-1}(W_i) = \bigsqcup _{j=1}^{k(i)}U_{ij}.$$
Let $(V_{ij}, z_{ij})$ be a coordinate neighbourhood of $P_{ij}$
such that $z_{ij}(P_{ij}) = 0$ and $V_{ij} \subset U_{ij}$.
For a sufficiently small $\varepsilon > 0$ we set
$$\Delta _{ij}(\varepsilon ) := \{ P \in V_{ij} ;
|z_{ij}(P)| < \varepsilon \}.$$
We denote by $C_{ij}(\varepsilon )$ the boundary of $\Delta _{ij}(\varepsilon )$
with positive direction.
\par
Let $\omega \in H^0(X_{{\mathfrak m}}, {\mathcal E}^{(2)}_{\mathfrak m})$ be
any representaitive of $\xi $. We can take $W_i$ so small that there exists
$\eta _i \in H^0(\overline{W_i}, {\mathcal E}^{(1,0)}_{\mathfrak m})$ with
$\omega = d\eta _i$ on $\overline{W_i} \setminus \{ Q_i \}$ by the definition
of ${\mathcal E}^{(2)}_{{\mathfrak m}}$.
Since $\rho ^{*}\omega = d(\rho ^{*}\eta _i)$, we have
\begin{eqnarray*}
\lefteqn{\iint_{X \setminus \left(\cup_{ij}\Delta _{ij}(\varepsilon )\right)}
\rho ^{*}\omega }\\
&& =
\iint_{X \setminus \left(\cup _{i}\rho ^{-1}(W_i)\right)}\rho ^{*}\omega +
\iint_{\cup _{i}\left(\rho ^{-1}(W_i) \setminus \cup _{j}
\Delta _{ij}(\varepsilon )\right)}\rho ^{*}\omega \\
&& = \iint_{X \setminus \left(\cup _{i}\rho ^{-1}(W_i)\right)}\rho ^{*}\omega +
\int_{\cup _{i}\partial (\rho ^{-1}(W_i))} \rho ^{*}\eta _i -
\int _{\cup _{ij} C_{ij}(\varepsilon )}\rho ^{*}\eta _i \\
&& = \iint_{X \setminus \left(\cup _{i}\rho ^{-1}(W_i)\right)}\rho ^{*}\omega +
\int _{\cup _{i}\partial (\rho ^{-1}(W_i))}\rho ^{*}\eta _i -
\sum _{i=1}^{N}\left(\sum_{j=1}^{k(i)}\int _{C_{ij}(\varepsilon )}\rho ^{*}\eta _i
\right)\\
&& = \iint_{X \setminus \left(\cup _{i}\rho ^{-1}(W_i)\right)}\rho ^{*}\omega +
\sum_{i=1}^{N}\int _{\partial (\rho ^{-1}(W_i))}\rho ^{*}\eta _i.\\
\end{eqnarray*}
Since the last expression in the above equalities is independent
on $\varepsilon $, 
we have the limit
$$\lim _{\varepsilon \to 0}
\iint_{X \setminus \left(\cup _{ij}\Delta _{ij}(\varepsilon )\right)}
\rho ^{*}\omega = \iint _{X \setminus S}\rho ^{*}\omega .$$
\par
Furthermore, we see by the above argument that ${\rm Res}(\xi )$ is
determined independently of the choice of a representative $\omega $
of $\xi $.
\end{proof}
Let ${\mathcal M}^{(1)}$ be the sheaf of germs of meromorphic 1-forms on $X$.
Consider an open covering ${\mathfrak U} = \{ U_i\}_{i \in I}$ of
$X_{\mathfrak m}$.
\begin{definition}
A cochain $\mu = \{ \omega _i\} \in C^0({\mathfrak U}, \rho _{*}{\mathcal M}^{(1)})$
is an $\Omega _{\mathfrak m}$-distribution if
$$\omega _j - \omega _i \in H^0(U_i \cap U_j, \Omega _{{\mathfrak m}})\quad
\text{for any $i, j \in I$}.$$
\end{definition}
If $\mu = \{ \omega _i \} \in C^0({\mathfrak U}, \rho _{*}{\mathcal M}^{(1)})$
is an $\Omega _{\mathfrak m}$-distribution, then
$\delta \mu \in Z^1({\mathfrak U}, \Omega _{\mathfrak m})$, where $\delta $
is the coboundary operator. We denote by
$[\delta \mu ] \in H^1(X_{\mathfrak m}, \Omega _{\mathfrak m})$
the cohomology class given by $\delta \mu $.
\begin{definition}
Let $\mu = \{ \omega _i\} \in C^0({\mathfrak U}, \rho _{*}{\mathcal M}^{(1)})$
be an $\Omega _{\mathfrak m}$-distribution. For any $Q \in X_{\mathfrak m}$
we define the residue ${\rm Res}_Q(\mu )$ of $\mu $ at $Q$ by
$${\rm Res}_Q(\mu ) := {\rm Res}_Q(\omega _i) =
\sum_{P \in \rho^{-1}(Q)}{\rm Res}_P(\rho ^{*}\omega _i)\quad
\text{for some $i$ with $Q \in U_i$}.$$
\end{definition}
\begin{remark}
We see that the definition of ${\rm Res}_Q(\mu )$ is independent of the choice of
$i$ with $Q \in U_i$, by the definition of $\Omega _{\mathfrak m}$.
\end{remark}
We define
$${\rm Res}(\mu ) := \sum _{Q \in X_{\mathfrak m}}
{\rm Res}_Q(\mu ).$$
\begin{remark}
The residue ${\rm Res}(\mu )$ does not change if we take a refinement
of ${\mathfrak U}$.
\end{remark}
\begin{proposition}
Let $\mu = \{ \omega _i\} \in C^0({\mathfrak U}, \rho _{*}{\mathcal M}^{(1)})$
be an $\Omega _{\mathfrak m}$-distribution. Then we have the cohomology class
$[\delta \mu ]\in H^1(X_{\mathfrak m}, \Omega _{\mathfrak m})$.
In this case, we have
$${\rm Res}(\mu ) = {\rm Res}([\delta \mu ]).$$
\end{proposition}
\begin{proof}
We first note that $\delta \mu = \{ \omega _j - \omega _i\} \in Z^1({\mathfrak U},
\Omega _{\mathfrak m}) \subset Z^1({\mathfrak U}, {\mathcal E}^{(1,0)}_{\mathfrak m})$.
Since ${\mathcal E}^{(1,0)}_{{\mathfrak m}}$ is a fine sheaf,
there exists a cochain $\{ \sigma _i \} \in C^0({\mathfrak U}, {\mathcal E}^{(1,0)}_
{\mathfrak m})$ such that
$$\omega _j - \omega _i = \sigma _j - \sigma _i\quad
\text{on $U_i \cap U_j$}.$$
The operator $d$ has the decomposition $d = \partial + \overline{\partial}$.
Since
$$d(\sigma _j - \sigma _i) = d(\omega _j - \omega _i) =
\overline{\partial}(\omega _j - \omega _i) = 0,$$
$d\sigma _j = d\sigma _i$ on $U_i \cap U_j$. Then there exists a 2-form 
$\tau \in H^0(X_{{\mathfrak m}}, {\mathcal E}^{(2)}_{{\mathfrak m}})$
such that
$\tau |{U_i} = d\sigma _i$ on $U_i$. The form $\tau $ is a representative
of $[\delta \mu ]$.
\par
We assume that $Q_1, \dots , Q_n \in X_{\mathfrak m} \setminus \overline{S}
= X \setminus S$ are the all poles of $\mu $ on $X_{\mathfrak m} \setminus
\overline{S}$. We set
$X'_{\mathfrak m} := X_{\mathfrak m} \setminus
\{Q_1, \dots , Q_n\}$. Since
$$\sigma _i - \omega _i = \sigma _j - \omega _j\quad
\text{on $X'_{\mathfrak m} \cap U_i \cap U_j$},$$
there exists $\sigma \in H^0(X'_{\mathfrak m}, {\mathcal E}^{(1,0)}_{{\mathfrak m}})$
such that $\sigma = \sigma _i - \omega _i$ on $X'_{\mathfrak m} \cap U_i$.
Then we have
$$d\sigma = d\sigma _i - d\omega _i = d\sigma _i\quad
\text{on $X'_{\mathfrak m} \cap U_i$}.$$
Thus we obtain that $\tau = d\sigma $ on $X'_{\mathfrak m}$.
\par
For any $k = 1, \dots , n$, there exists $i = i(k) \in I$ such that
$Q_k \in U_{i(k)}$. Let $(V_k, z_k)$ be a coordinate neighbourhood of
$Q_k$ such that $V_k \subset U_{i(k)}$ and $z_k(Q_k) = 0$.
We may assume that $V_1, \dots , V_n$ are pairwise disjoint and 
$z_1(V_1), \dots , z_n(V_n)$ are discs in ${\mathbb C}$.
There exist $C^{\infty }$ functions $f_k$ $(k = 1, \dots , n)$ on
$X_{\mathfrak m}$ with compact support such that
$0 \leq f_k \leq 1$, ${\rm supp}(f_k) \subset V_k$ and
$f_k|V'_k = 1$ for an open neighbourhood $V'_k$ of $Q_k$ with
$V'_k \subset \subset V_k$. Define $g := 1 - (f_1 + \cdots + f_n)$.
Since $g$ is a $C^{\infty}$ function on $X_{\mathfrak m}$ with
$g|V'_k = 0$, $g\sigma $ is a $C^{\infty }$ $(1,0)$-form on
$X_{\mathfrak m} \setminus \overline{S}$. We can take $V_k$ so small
that it does not contain a point in $\overline{S}$. 
Then we have $g\sigma = \sigma $ in a neighbourhood of any point
in $\overline{S}$.
Therefore
$g\sigma \in H^0(X_{{\mathfrak m}}, {\mathcal E}^{(1,0)}_{\mathfrak m})$.
\par
Take any $P \in S$. Let $(W, z)$ be a coordinate neighbourhood of $P$
with $z(P) = 0$. For a small $\varepsilon > 0$ we set
$$\Delta (P, \varepsilon ) := \{ P' \in W ;
|z(P')| < \varepsilon \}.$$
Then we obtain
\begin{eqnarray*}
\iint_{X \setminus S}\rho ^{*}(d(g\sigma )) & = &
\iint_{X \setminus S}d(\rho ^{*}(g\sigma ))\\
& = & \lim _{\varepsilon \to 0}\iint _
{X \setminus \left(\cup _{P \in S}\Delta (P, \varepsilon )\right)}
d(\rho ^{*}(g\sigma ))\\
& = & \lim _{\varepsilon \to 0}\left( - \sum _{P \in S}
\int _{\partial \Delta (P, \varepsilon )}\rho ^{*}\sigma \right)\\
& = & 0.
\end{eqnarray*}
Since $1 = g + \sum _{k=1}^{n}f_k$, we have
$$\tau = d\sigma = d(g\sigma ) + \sum _{k=1}^{n}
d(f_k\sigma )\quad
\text{on $X'_{\mathfrak m}$}.$$
The above equality holds on $X_{\mathfrak m}$ for the both sides are
$C^{\infty }$ on $X_{\mathfrak m} \setminus \overline{S}$. Then we obtain
\begin{eqnarray*}
\iint_{X \setminus S}\rho ^{*}\tau & = &
\iint_{X \setminus S}\rho ^{*}(d(g\sigma )) +
\sum_{k=1}^{n}\iint_{X \setminus S}\rho ^{*}(d(f_k\sigma ))\\
& = & \sum _{k=1}^{n}\iint_{X \setminus S}
\rho ^{*}(d(f_k\sigma ))\\
& = & \sum _{k=1}^{n}\left(
\iint_{\rho^{-1}(V_k)}\rho ^{*}(d(f_k\sigma _{i(k)})) -
\iint_{\rho ^{-1}(V_k)}\rho ^{*}(d(f_k\omega _{i(k)}))\right).\\
\end{eqnarray*}
Here we note
$$\iint_{\rho ^{-1}(V_k)}\rho ^{*}(d(f_k\sigma _{i(k)})) =
\iint_{\rho ^{-1}(V_k)}d(\rho ^{*}(f_k\sigma _{i(k)})) =
\int_{\partial(\rho ^{-1}(V_k))}\rho ^{*}(f_k\sigma _{i(k)}) = 0.$$
We set
$$\tilde{\Delta }_k(\varepsilon ) :=
\rho ^{-1}(\{Q \in V_k ; |z_k(Q)| < \varepsilon \})$$
for a sufficiently small $\varepsilon > 0$. Since $\omega _{i(k)}$ is
a meromorphic 1-form with $Q_k$ as a pole, we obtain
\begin{eqnarray*}
\iint_{\rho^{-1}(V_k)}\rho ^{*}(d(f_k\omega _{i(k)})) & = &
\lim _{\varepsilon \to 0}\iint_{\rho^{-1}(V_k) \setminus \tilde{\Delta }_k
(\varepsilon )}\rho ^{*}(d(f_k\omega _{i(k)}))\\
& = & \lim _{\varepsilon \to 0}\iint_{\rho^{-1}(V_k)\setminus
\tilde{\Delta }_k(\varepsilon )}d(\rho^{*}(f_k\omega _{i(k)}))\\
& = & \lim _{\varepsilon \to 0}\left(
\int_{\partial(\rho ^{-1}(V_k))}\rho ^{*}(f_k\omega _{i(k)}) -
\int_{\partial \tilde{\Delta }_k(\varepsilon )}
\rho ^{*}\omega _{i(k)}\right)\\
& = & - \lim _{\varepsilon \to 0}\int_{\partial \tilde{\Delta }_k(\varepsilon )}
\rho ^{*}\omega _{i(k)}\\
& = & - 2\pi \sqrt{-1} {\rm Res}_{Q_k}(\omega _{i(k)}).
\end{eqnarray*}
Thus we obtain
$$\frac{1}{2\pi \sqrt{-1}}\iint_{X \setminus S}\rho^{*}\tau =
\sum_{k=1}^{n} {\rm Res}_{Q_k}(\omega _{i(k)}) = {\rm Res}(\mu ).$$
\par
On the other hand, we have
$${\rm Res}([\delta \mu ]) = 
\frac{1}{2\pi \sqrt{-1}} \iint_{X \setminus S}\rho ^{*}\tau ,$$
because $\tau $ is a representative of $[\delta \mu ]$.
\end{proof}
Let $K$ be a canonical divisor on $X$ given by a meromorphic 1-form $\omega $.
We may assume that $K$ is prime to $S$ (i.e. $K \in {\rm Div}(X_{\mathfrak m})$).
Otherwise we take a meromorphic function $f$ on $X$ with
$(f) + K = 0$ at any $P \in S$ (for example, see Lemma 1.15 in
Chapter VI in \cite{ref13}) and set $\omega ' = f\omega $.
Then $(\omega ')$ is a divisor prime to $S$.
For any $D \in {\rm Div}(X_{\mathfrak m})$ and for any $f \in H^0(X_{\mathfrak m}, {\mathcal L}_{\mathfrak m} (D+K))$,
we have $(\rho ^{*}f)\omega \in H^0(X, \Omega (D))$.
Then we obtain an injection
$$\begin{array}{ccccc}
H^0(X_{\mathfrak m}, {\mathcal L}_{\mathfrak m}(D+K))& \longrightarrow & H^0(X, \Omega (D)) &\hookrightarrow &
\rho ^{*}H^0(X_{\mathfrak m}, \Omega _{\mathfrak m}(D))\\
f & \longmapsto & (\rho ^{*}f)\omega . & & \\
\end{array}$$
\begin{lemma}
There exists $k_0 \in {\mathbb Z}$ such that
$$\dim H^0(X_{\mathfrak m}, \Omega _{\mathfrak m}(D)) \geq \deg D + k_0$$
for any $D \in {\rm Div}(X_{\mathfrak m})$.
\end{lemma}
\begin{proof}
By the Riemann-Roch Theorem (Theorem 2.10) we have
$$\dim H^0(X_{\mathfrak m}, {\mathcal L}_{\mathfrak m}(D+K)) -
\dim H^1(X_{\mathfrak m}, {\mathcal L}_{\mathfrak m}(D+K)) = \deg( D+K) + 1 - \pi .$$
Then we obtain
\begin{eqnarray*}
\dim H^0(X_{\mathfrak m}, \Omega _{\mathfrak m}(D)) & \geq &
\dim H^0(X, \Omega (D))\\
& \geq & \dim H^0(X_{\mathfrak m}, {\mathcal L}_{\mathfrak m}(D + K))\\
& = & \dim H^1(X_{\mathfrak m}, {\mathcal L}_{\mathfrak m}(D + K)) + \deg (D + K) + 1 - \pi \\
& \geq & \deg D + \deg K + 1 - \pi .
\end{eqnarray*}
If we set $k_0 := \deg K + 1 - \pi $, then the assertion holds.
\end{proof}

\subsection{Dual pairing}
Take any $D \in {\rm Div}(X_{\mathfrak m})$. The product
$$\Omega _{\mathfrak m}(-D) \times {\mathcal L}_{\mathfrak m}(D) \longrightarrow \Omega _{\mathfrak m},\quad
(\omega , f) \longmapsto \omega f$$
induces a map
$$H^0(X_{\mathfrak m}, \Omega _{\mathfrak m}(-D)) \times H^1(X_{\mathfrak m}, {\mathcal L}_{\mathfrak m}(D))
\longrightarrow H^1(X_{\mathfrak m}, \Omega _{\mathfrak m}).$$
The composition of this map with ${\rm Res} : H^1(X_{\mathfrak m}, \Omega _{\mathfrak m}) \longrightarrow {\mathbb C}$
produces a bilinear map
$$\begin{array}{c}
\langle \ ,\ \rangle : H^0(X_{\mathfrak m}, \Omega _{\mathfrak m}(-D)) \times
H^1(X_{\mathfrak m}, {\mathcal L}_{\mathfrak m}(D)) \longrightarrow {\mathbb C}\\
(\omega , \xi ) \longmapsto \langle \omega , \xi \rangle := {\rm Res}(\omega \xi ).\\
\end{array}$$
This map induces a linear map
$$\iota _{D} : H^0(X_{\mathfrak m}, \Omega _{\mathfrak m}(-D)) \longrightarrow
H^1(X_{\mathfrak m}, {\mathcal L}_{\mathfrak m}(D))^{*}.$$
\begin{proposition}
The map $\iota _{D}$ is injective.
\end{proposition}
\begin{proof}
It suffices to show that for any non-zero $\omega \in H^0(X_{\mathfrak m}, \Omega _{\mathfrak m}(-D))$
there exists $\xi \in H^1(X_{\mathfrak m}, {\mathcal L}_{\mathfrak m}(D))$ such that
$\langle \omega , \xi \rangle \not= 0$.
Let $Q \in X_{{\mathfrak m}} \setminus \overline{S}$ be a point with $D(Q) = 0$. We take a coordinate neighbourhood
$(U_0, z)$ of $Q$ such that $z(Q) = 0$ and $D|U_0 = 0$.
We may assume that there exists $f \in H^0(U_0, {\mathcal O}_{\mathfrak m})$ such that
$\omega = fdz$ on $U_0$. Furthermore, we can take $U_0$ so small that $f$ has no zeros in
$U_0 \setminus \{ Q \}$ and $U_0 \subset X_{\mathfrak m} \setminus \overline{S}$. We set
$U_1 := X_{\mathfrak m} \setminus \{ Q \}$ and ${\mathfrak U} := \{ U_0, U_1 \}$.
We define $\eta = \{ f_0, f_1\} \in C^0({\mathfrak U}, {\mathcal M}_{\mathfrak m})$ by
$$f_0 := \frac{1}{zf}\quad \text{and\quad $f_1 := 0$},$$
where ${\mathcal M}_{\mathfrak m}$ is the sheaf of germs of meromorphic functions on $X_{\mathfrak m}$.
Then $\omega \eta = \{ dz/z, 0\} \in C^0({\mathfrak U}, \rho _{*}{\mathcal M}^{(1)})$
is an $\Omega _{\mathfrak m}$-distribution with
${\rm Res}(\omega \eta ) = {\rm Res}_Q(dz/z) = 1$.
We have $\delta \eta \in Z^1({\mathfrak U}, {\mathcal L}_{\mathfrak m}(D))$.
Let $\xi := [\delta \eta ]$ be the cohomology class of $\delta \eta $ in
$H^1(X_{\mathfrak m}, {\mathcal L}_{\mathfrak m}(D))$.
Since $\omega \xi = \omega [\delta \eta ] = [\delta (\omega \eta )]$, we obtain
$$\langle \omega , \xi \rangle = {\rm Res}(\omega \xi ) = {\rm Res}([\delta (\omega \eta )])
= {\rm Res}(\omega \eta ) = 1$$
by Proposition 3.7.
\end{proof}
Let $D, D' \in {\rm Div}(X_{\mathfrak m})$ with $D'\leq D$.
The inclusion $0 \longrightarrow {\mathcal L}_{\mathfrak m}(D') \longrightarrow {\mathcal L}_{\mathfrak m}(D)$
induces an epimorphism
$$H^1(X_{\mathfrak m}, {\mathcal L}_{\mathfrak m}(D')) \longrightarrow
H^1(X_{\mathfrak m}, {\mathcal L}_{\mathfrak m}(D)) \longrightarrow 0.$$
We obtain the following monomorphism of dual spaces from this epimorphism
$$0 \longrightarrow H^1(X_{\mathfrak m}, {\mathcal L}_{\mathfrak m}(D))^{*}
\stackrel{i^{D}_{D'}}{\longrightarrow }
H^1(X_{\mathfrak m}, {\mathcal L}_{\mathfrak m}(D'))^{*}.$$
The diagram
$$
\begin{CD}
0 @>>> H^1(X_{\mathfrak m}, {\mathcal L}_{\mathfrak m}(D))^{*}
@>i^{D}_{D'}>> H^1(X_{\mathfrak m}, {\mathcal L}_{\mathfrak m}(D'))^{*}\\
@.   @AA\iota _D A   @AA\iota _{D'}A\\
0 @>>> H^0(X_{\mathfrak m}, \Omega _{\mathfrak m}(-D))
@>>> H^0(X_{\mathfrak m}, \Omega _{\mathfrak m}(-D'))
\end{CD}$$
commutes.
\begin{lemma}
Let $D, D' \in {\rm Div}(X_{\mathfrak m})$ such that $D' \leq D$.
Suppose that
$\lambda \in H^1(X_{\mathfrak m}, {\mathcal L}_{\mathfrak m}(D))^{*}$ and
$\omega \in H^0(X_{\mathfrak m}, \Omega _{\mathfrak m}(-D'))$ satisfy
$i^{D}_{D'}(\lambda ) = \iota _{D'}(\omega )$.
Then $\omega $
belongs to $H^0(X_{\mathfrak m}, \Omega _{\mathfrak m}(-D))$ and
$\lambda = \iota _D(\omega )$.
\end{lemma}
\begin{proof}
We suppose that $\omega \notin H^0(X_{\mathfrak m}, \Omega _{\mathfrak m}(-D))$.
Then there exists $Q \in X_{{\mathfrak m}} \setminus \overline{S}$ such that 
${\rm ord}_Q(\omega ) < D(Q)$.
Take a coordinate neighbourhood $(U_0, z)$ of $Q$ with $z(Q) = 0$.
We may assume that $U_0$ satisfies the following conditions;\\
(i) there exists $f \in H^0(U_0, {\mathcal M}_{\mathfrak m})$ such that
$\omega = fdz$ on $U_0$ and $f$ has no zeros and poles in
$U_0 \setminus \{ Q \}$,\\
(ii) $D|(U_0 \setminus \{ Q \}) = 0$ and
$D'|(U_0 \setminus \{ Q \}) = 0$.\\
We set $U_1 := X_{\mathfrak m} \setminus \{ Q \}$ and
${\mathfrak U} := \{ U_0, U_1 \}$.
We define $\eta = \{ f_0, f_1 \} \in C^0({\mathfrak U}, {\mathcal M}_{\mathfrak m})$
by $f_0 := 1/(zf)$ and $f_1: = 0$. Then 
$\eta \in C^0({\mathfrak U}, {\mathcal L}_{\mathfrak m}(D))$ and
$\delta \eta \in Z^1({\mathfrak U}, {\mathcal L}_{\mathfrak m}(D)) =
Z^1({\mathfrak U}, {\mathcal L}_{\mathfrak m}(D')) =
Z^1({\mathfrak U}, {\mathcal O}_{\mathfrak m}).$
We denote by $\xi $ and by $\xi '$ the cohomology classes of
$\delta \eta $ in $H^1(X_{\mathfrak m}, {\mathcal L}_{\mathfrak m}(D))$
and in $H^1(X_{\mathfrak m}, {\mathcal L}_{\mathfrak m}(D'))$ respectively.
Since $\eta \in C^0({\mathfrak U}, {\mathcal L}_{\mathfrak m}(D))$, we have
$\xi = 0$. Therefore we obtain
$$\langle \omega , \xi '\rangle = \iota _{D'}(\omega ) (\xi ') =
i^{D}_{D'}(\lambda )(\xi ') = \lambda (\xi ) = 0.$$
On the other hand, we have
$$\langle \omega , \xi '\rangle = {\rm Res}([\delta (\omega \eta )]) =
{\rm Res}(\omega \eta ) = 1$$
by Proposition 3.7. This is a contradiction. Then we conclude
$\omega \in H^0(X_{\mathfrak m}, \Omega _{\mathfrak m}(-D))$.
Furthermore we have
$$i^{D}_{D'}(\lambda ) = \iota _{D'}(\omega ) = i^{D}_{D'}(\iota _D(\omega )).$$
Since $i^{D}_{D'}$ is one-to-one, we obtain
$\lambda = \iota _D(\omega ).$
\end{proof}
Let $D, B \in {\rm Div}(X_{\mathfrak m})$. Any $\psi \in
H^0(X_{\mathfrak m}, {\mathcal L}_{\mathfrak m}(B))$ gives the following
sheaf morphism
$${\mathcal L}_{\mathfrak m}(D-B) \overset \psi \longrightarrow
{\mathcal L}_{\mathfrak m}(D),\quad
f \longmapsto \psi f.$$
This sheaf morphism induces a linear map
$$H^1(X_{\mathfrak m}, {\mathcal L}_{\mathfrak m}(D-B))
\longrightarrow H^1(X_{\mathfrak m}, {\mathcal L}_{\mathfrak m}(D)).$$
Then we obtain a linear map of dual spaces from the above linear map.
Using the same notation, we write it as follows
$$\psi : H^1(X_{\mathfrak m}, {\mathcal L}_{\mathfrak m}(D))^{*}
\longrightarrow H^1(X_{\mathfrak m}, {\mathcal L}_{\mathfrak m}(D-B))^{*}.$$
We note that the diagram
$$
\begin{CD}
 H^1(X_{\mathfrak m}, {\mathcal L}_{\mathfrak m}(D))^{*}
 @>\psi >>
 H^1(X_{\mathfrak m}, {\mathcal L}_{\mathfrak m}(D-B))^{*}\\
 @AA\iota _D A   @AA\iota _{D-B}A\\
 H^0(X_{\mathfrak m}, \Omega _{\mathfrak m}(-D))
@>\psi >> H^0(X_{\mathfrak m}, \Omega _{\mathfrak m}(-D+B))
\end{CD}$$
commutes. 
\begin{lemma}
Let $D, B \in {\rm Div}(X_{\mathfrak m})$.
If $\psi \in H^0(X_{\mathfrak m}, {\mathcal L}_{\mathfrak m}(B))$
is not identically zero and has no zeros on $\overline{S},$ then the map
$$\psi : H^1(X_{\mathfrak m}, {\mathcal L}_{\mathfrak m}(D))^{*}
\longrightarrow H^1(X_{\mathfrak m}, {\mathcal L}_{\mathfrak m}(D-B))^{*}$$
is injective.
\end{lemma}
\begin{proof}
Set $A:= (\psi )$. Then $A \in {\rm Div}(X_{\mathfrak m})$ and
$A \geq -B$ by the assumption. The map
$\psi : {\mathcal L}_{\mathfrak m}(D-B) \longrightarrow {\mathcal L}_{\mathfrak m}(D)$
has the decomposition as follows
$${\mathcal L}_{\mathfrak m}(D-B) \hookrightarrow
{\mathcal L}_{\mathfrak m}(D+A) \overset \psi \longrightarrow
{\mathcal L}_{\mathfrak m}(D).$$
Here the last map is an isomorphism for $\psi $ has no zeros on $\overline{S}$.
Then the map $\psi : {\mathcal L}_{\mathfrak m}(D-B) \longrightarrow
{\mathcal L}_{\mathfrak m}(D)$ is an injection. Therefore we obtain the assertion.
\end{proof}
\begin{lemma}
Let
$f : V \longrightarrow W$ be a ${\mathbb C}$-linear map between 
${\mathbb C}$-linear spaces $V$ and $W$. We suppose that $V$ is finite dimensional, and
consider it as a topological space by an isomorphism
$V \cong {\mathbb C}^N$ for some $N \in {\mathbb N}$.
If
there exists
an open dense subset $U$ of $V$ such that $f|U$ is one-to-one,
then $f$ is one-to-one on $V$.
\end{lemma}
\begin{proof}
Assume that there exist $v_1, v_2 \in V$ with $v_1 \not= v_2$ such that
$f(v_1) = f(v_2)$. Since $U + v_1$ is also an open dense subset of $V$,
$V \setminus (U + v_1)$ has no interior point. Then
$(U + v_1) \cap (U + v_2) \not= \phi $ for $U + v_2$ is open.
Therefore we have $u_1, u_2 \in U$ such that
$u_1 + v_1 = u_2 + v_2$. By the assumption we obtain $f(u_1) = f(u_2)$.
This gives $u_1 = u_2$. Hence we have $v_1 = v_2$, a contradiction.
\end{proof}
For any $Q \in \overline{S}$ we define
$$N_Q(B) := \{ \psi \in H^0(X_{\mathfrak m}, {\mathcal L}_{\mathfrak m}(B)) ;
\psi = 0 \; \text{at $Q$}\}.$$
By Theorem 2.10 $H^0(X_{\mathfrak m}, {\mathcal L}_{\mathfrak m}(B))$ is of
finite dimension. We note that
$$\dim N_Q(B) \leq \dim H^0(X_{\mathfrak m}, {\mathcal L}_{\mathfrak m}(B)) - 1.$$
Then a set
$$H^0(X_{\mathfrak m}, {\mathcal L}_{\mathfrak m}(B)) \setminus
\left( \bigcup _{Q\in \overline{S}}N_Q(B)\right)$$
is an open dense subset of $H^0(X_{\mathfrak m}, {\mathcal L}_{\mathfrak m}(B))$.

\subsection{Proof of the Serre duality}
We prove Theorem 3.1 analytically. We have already shown that
$$\iota _D : H^0(X_{\mathfrak m}, \Omega _{\mathfrak m}(-D))
\longrightarrow H^1(X_{\mathfrak m}, {\mathcal L}_{\mathfrak m}(D))^{*}$$
is injective (Proposition 3.9). Then it suffices to show that
$\iota _D$ is also surjective.
\par
Take any non-zero $\lambda \in H^1(X_{\mathfrak m}, {\mathcal L}_{\mathfrak m}(D))^{*}$.
Fix $P \in X \setminus S$. We set
$$D_n := D - n P \in {\rm Div}(X_{\mathfrak m})\quad (n \in {\mathbb N}).$$
Let
$$\Lambda := \{ \psi \lambda ; \psi \in
H^0(X_{\mathfrak m}, {\mathcal L}_{\mathfrak m}(n P))\} \subset
H^1(X_{\mathfrak m}, {\mathcal L}_{\mathfrak m}(D_n))^{*}.$$ 
By Lemma 3.11 a linear map
$$H^0(X_{\mathfrak m}, {\mathcal L}_{\mathfrak m}(n P)) \longrightarrow \Lambda ,
\quad \psi \longmapsto \psi \lambda $$
is one-to-one on an open dense subset
$H^0(X_{\mathfrak m}, {\mathcal L}_{\mathfrak m}(n P))\setminus
\left( \cup_{Q\in \overline{S}} N_Q(n P)\right)$ of
$H^0(X_{{\mathfrak m}}, {\mathcal L}_{{\mathfrak m}}(n P))$.
Then it is one-to-one on $H^0(X_{\mathfrak m}, {\mathcal L}_{\mathfrak m}(n P))$ by Lemma 3.12.
Hence we have $\Lambda \cong H^0(X_{\mathfrak m}, {\mathcal L}_{\mathfrak m}(n P)).$
It follows from the Riemann-Roch Theorem (Theorem 2.10) that
\begin{eqnarray*}
\dim \Lambda & = & \dim H^0(X_{\mathfrak m}, {\mathcal L}_{\mathfrak m}(n P))\\
& = & \dim H^1(X_{\mathfrak m}, {\mathcal L}_{\mathfrak m}(n P)) + 1 - \pi + \deg (n P)\\
&\geq & 1 - \pi +n.
\end{eqnarray*}
Since a map
$$\iota _{D_n} : H^0(X_{\mathfrak m}, \Omega _{\mathfrak m}(-D_n))
\longrightarrow H^1(X_{\mathfrak m}, {\mathcal L}_{\mathfrak m}(D_n))^{*}$$
is injective, we have
\begin{eqnarray*}
\dim {\rm Im}(\iota _{D_n}) & = &
\dim H^0(X_{\mathfrak m}, \Omega _{\mathfrak m}(-D_n))\\
& \geq & n - \deg D + k_0
\end{eqnarray*}
by Lemma 3.8, where ${\rm Im}(\iota _{D_n})$ is the image of $\iota _{D_n}$.
By Proposition 2.8 we have $H^0(X_{\mathfrak m}, {\mathcal L}_{\mathfrak m}(D_n))
= 0$ if $n > \deg D$.
Using the Riemann-Roch Theorem again, we obtain
\begin{eqnarray*}
\dim H^1(X_{\mathfrak m}, {\mathcal L}_{\mathfrak m}(D_n))^{*}
& = & \dim H^1(X_{\mathfrak m}, {\mathcal L}_{\mathfrak m}(D_n))\\
& = & \dim H^0(X_{\mathfrak m}, {\mathcal L}_{\mathfrak m}(D_n))
- 1 + \pi - \deg D_n \\
& = & \pi - 1 - \deg D + n
\end{eqnarray*}
if $n > \deg D$.
Then it holds for a sufficiently large $n$ that
$$\dim \Lambda + \dim ({\rm Im}(\iota _{D_n})) >
\dim H^1(X_{\mathfrak m}, {\mathcal L}_{\mathfrak m}(D_n))^{*}.$$
Thus $\Lambda \cap {\rm Im}(\iota _{D_n}) \not= \phi $ for a 
sufficiently large $n$.
Since $H^0(X_{{\mathfrak m}},{\mathcal L}_{{\mathfrak m}}(nP))
\setminus \left(\cup _{Q \in \overline{S}}N_Q(nP)\right)$ is an
open dense subset of $H^0(X_{{\mathfrak m}},{\mathcal L}_{{\mathfrak m}}(nP))$,
there exist
$\psi \in H^0(X_{\mathfrak m}, {\mathcal L}_
{\mathfrak m}(n P))$ and $\omega \in H^0(X_{\mathfrak m}, \Omega _{\mathfrak m}
(-D_n))$ such that 
$\psi \lambda = \iota _{D_n}(\omega )$ and $\psi (Q) \not= 0$ at any
$Q \in \overline{S}$.

\par
Let $A := (\psi )$. Then $A \in {\rm Div}(X_{{\mathfrak m}})$ and
$1/\psi \in H^0(X_{{\mathfrak m}}, {\mathcal L}_{{\mathfrak m}}(A))$. If we set
$$D' := D_n - A \in {\rm Div}(X_{{\mathfrak m}}),$$
then we have
$$D - D' = n P + A \geq 0.$$
Therefore we obtain a monomorphism
$$i^{D}_{D'} :
H^1(X_{{\mathfrak m}}, {\mathcal L}_{{\mathfrak m}}(D))^{*} \longrightarrow
H^1(X_{{\mathfrak m}}, {\mathcal L}_{{\mathfrak m}}(D'))^{*}.$$
We have
\begin{eqnarray*}
i^{D}_{D'}(\lambda )
& = & \frac{1}{\psi } \psi \lambda \\
& = & \frac{1}{\psi }\iota _{D_n}(\omega)\\
& = & \iota _{D_n - A}\left( \frac{\omega }{\psi }\right)\\
& = & i_{D'}\left( \frac {\omega }{\psi }\right).
\end{eqnarray*}
Then it
follows from Lemma 3.10 that
$\omega / \psi \in H^0(X_{\mathfrak m}, \Omega _{\mathfrak m}(-D))$
and
$\lambda = \iota _D(\omega /\psi ).$
This finishes the proof.

\subsection{Riemann-Roch Theorem (second version)}
Combining Theorem 2.10 with Theorem 3.1, we obtain the following second version of
the Riemann-Roch Theorem.
\begin{theorem}
For any $D \in {\rm Div}(X_{\mathfrak m})$ we have
$$\dim H^0(X_{\mathfrak m},{\mathcal L}_{\mathfrak m}(D)) -
\dim H^0(X_{\mathfrak m}, \Omega _{\mathfrak m}(-D))
= \deg D + 1 - \pi .$$
\end{theorem}

\section{Generalized Abel's Theorem}

\subsection{Introduction}
A generalized Abel's theorem was first formulated and proved algebraically in \cite{ref15}.
After that, Jambois \cite{ref12} tried to treat it analytically. However, it seems to us that his argument is not correct. 
Furthermore, we think that the condition $f \equiv 1 \ {\rm mod}\ {\mathfrak m}$ in the
statement of the above generalized Abel's theorem is unusual.
We should consider the principal divisors which are defined by meromorphic functions on
$X_{\mathfrak m}$.
\par
We use the same notations as in the preceding sections. We assign a non-zero
constant $c_Q$ to each point $Q$ in $\overline{S}$. We call
$$c(\overline{S}) := (c_Q)_{Q \in \overline{S}}$$
a multiconstant on $\overline{S}$.
\begin{definition}
Let $f \in {\rm Mer}(X)$, and let $c(\overline{S})$ be a multiconstant on $\overline{S}$.
We write
$$f \equiv c(\overline{S})\ {\rm mod}\ {\mathfrak m}$$
if ${\rm ord}_P(f - c_Q) \geq {\mathfrak m}(P)$ for any $P \in S$ with $\rho (P) = Q$
at any $Q \in \overline{S}$.
\end{definition}
Our formulation of a generalized Abel's theorem is the following.
\begin{theorem}
Let $D \in {\rm Div}(X_{\mathfrak m})$ with $\deg D = 0$.
Then there exists a meromorphic function $f$ on $X$ with
$f \equiv c(\overline{S})\ {\rm mod}\ {\mathfrak m}$ for some multiconstant
$c(\overline{S})$ such that $D = (f)$ if and only if there exists a $1$-chain
$c \in C_1(X \setminus S)$ with $\partial c = D$ such that
$$\int _{c}\rho ^{*}\omega = 0$$
for any $\omega \in H^0(X_{\mathfrak m}, \Omega _{\mathfrak m})$.
\end{theorem}

\subsection{Weak solutions}
Let $D \in {\rm Div}(X_{\mathfrak m})$. We set
$$X_D := \{ P \in X ; D(P) \geq 0 \}.$$
\begin{definition}
A $C^{\infty }$ function $f$ on $X_D$ is called a weak solution of $D$
if it satisfies the following condition:\\
For any $P \in X$ there exist a coordinate neighbourhood $(U, z)$ of $P$
with $z(P) = 0$ and a $C^{\infty }$ function $\psi $ on $U$ with
$\psi (P) \not= 0$ such that
$$f = \psi z^{D(P)}\quad \text{on\quad $U \cap X_D$}.$$
\end{definition}
The following properties are immediate from the definition.\\
(i) If $f$ and $g$ are weak solutions of $D$, then there exists a
$C^{\infty }$ function $\varphi $ on $X$ such that
$$f = \varphi g$$
and $\varphi $ has no zeros.\\
(ii) If $f_1$ and $f_2$ are weak solutions of $D_1$ and $D_2$ respectively,
then $f := f_1f_2$ and $g := f_1/f_2$ are weak solutions of
$D_1 + D_2$ and $D_1 - D_2$ respectively.\\
(iii) If $f$ is a weak solution of $D$, then $df/f$ is smooth on the
complement of ${\rm Supp}(D)$ and $\overline{\partial}f/f$ is
smooth on $X$.

\subsection{Sheaf ${\mathcal E}^{(1)}_{\mathfrak m}$}
Let $U \subset X_{\mathfrak m}$ be an open set.
We define
$${\mathcal E}^{(1)}_{\mathfrak m}(U) := \{ \text{a $C^{\infty }$ $1$-form
$\omega $ on $U \setminus (U \cap \overline{S})$ satisfying the following
condition $(\star \star )$} \}.$$
The condition $(\star \star )$:\\
Let $Q \in U \cap \overline{S}$. We set
$\rho ^{-1}(Q) = \{ P_1, \dots , P_k\}.$
Let $V \subset U$ be an open neighbourhood of $Q$ such that
$$\rho ^{-1}(V) = \bigsqcup_{i=1}^{k}V_i\quad (P_i \in V_i),$$
$(V_i, z_i)$ is a coordinate neighbourhood of $P_i$ with $z_i(P_i )= 0$
and there exist $C^{\infty }$ functions $\varphi _i$ and $\psi _i$ on
$V_i \setminus \{ P_i\}$ with
$$\rho ^{*} \omega = \varphi _i dz_i + \psi _i d\overline{z}_i \quad
\text{on \quad $V_i \setminus \{ P_i \} $}.$$
Then limits
$$\lim_{P\to P_i}\varphi _i(P) z_i(P)^{{\mathfrak m}(P_i)}\quad
\text{and \quad $\lim_{P\to P_i}\psi _i(P) \overline{z_i(P)}^{{\mathfrak m}(P_i)}$}$$
exist.\\
Then a presheaf $\{ {\mathcal E}^{(1)}_{\mathfrak m}(U), r^{U}_{V}\} $
defines a sheaf ${\mathcal E}^{(1)}_{\mathfrak m}$ on $X_{\mathfrak m}$.

\subsection{Lemmas}

\begin{lemma}
Suppose that $c : [0,1] \longrightarrow X \setminus S$ is a curve and $U$
is a relatively compact open neighbourhood of $c([0,1])$ in $X \setminus S$.
Then there exists a weak solution $f$ of $\partial c$ with
$f|(X \setminus U) = 1$ such that for every $1$-form
$\omega \in H^0(X_{\mathfrak m}, {\mathcal E}^{(1)}_{\mathfrak m})$
with $d\omega = 0$ we have
$$\frac{1}{2\pi \sqrt{-1}} \iint _X \frac{d f}{f} \wedge \rho^{*} \omega
= \int _c \rho ^{*}\omega .$$ 
\end{lemma}

\begin{proof}
(a) We first consider the case that $c([0,1])$ is contained in a coordinate
neighbourhood $(U, z)$ such that $z(U) \subset {\mathbb C}$ is the unit disk.
\par
Let $a := c(0)$ and $b := c(1)$. 
We identify $U$ with $z(U)$.
There exists a positive number $r < 1$
such that $c([0,1]) \subset \{ |z| < r \}$. The function
$\log \left( (z-b)/(z-a)\right)$ has a well-defined branch in
$\{ r < |z| < 1\}.$
Take a $C^{\infty }$ function $\psi $ on $U$ such that
$\psi = 1 $ on $\{ |z| \leq r\}$ and $\psi = 0$ on
$\{ r' \leq |z| \},$ where $r < r' < 1.$
We define a $C^{\infty }$ function $f_0$ on $U \setminus \{ a \}$ by
$$f_0 := 
\begin{cases}
\exp \left(\psi \log \left( \frac{z-b}{z-a}\right)\right)&
\text{if $r < |z| < 1$,}\\
\frac{z-b}{z-a}& \text{if $|z| \leq r$}.\\
\end{cases}
$$
Since $f_0|\{ r' \leq |z| < 1 \} = 1$, we can extend $f_0$ to a
$C^{\infty }$ function $f$ on $X \setminus \{ a\}.$
Then $f$ is a weak solution of $\partial c$. Let
$\omega \in H^0(X_{\mathfrak m}, {\mathcal E}^{(1)}_{\mathfrak m})$ be
$d$-closed.
We have a $C^{\infty }$ function $g$ with compact support such that
$\rho ^{*} \omega = d g$ on $\{ |z| \leq r' \}.$
\par
There exist a coordinate neighbourhood $(V_1, t_1)$ of $a$ with
$t_1(a) = 0$ and a $C^{\infty }$ function $\psi _1 $ on $V_1$ such that
$V_1 \subset U$, $\psi _1(a) \not= 0$ and
\begin{equation*}
f = \psi _1 \frac{1}{t_1}\quad \text{on\quad$V_1 \setminus \{ a \}$}.
\end{equation*}
We can also take a coordinate neighbourhood $(V_2, t_2)$ of $b$ with
$t_2(b) = 0$ and a $C^{\infty }$ function $\psi _2 $ on $V_2$ such that
$V_2 \subset U$, $\psi _2(b) \not= 0$ and
\begin{equation*}
f = \psi _2 t_2\quad \text{on\quad$V_2 \setminus \{ b \}$}.
\end{equation*}
We may assume that $V_1 \cap V_2 = \phi $.
Let $0 < s < s' < 1$. We take $C^{\infty }$ functions
$\sigma _1$ and $\sigma _2 $ on $X$ such that
${\rm Supp}(\sigma _j) \subset \{ |t_j| < s'\}$ and
$\sigma _j|\{|t_j| \leq s\} = 1$ for $j = 1, 2$.
We set $g_j := \sigma _j g$ $(j = 1, 2)$ and
$$g_0 := g - ( g_1 + g_2).$$
Since $g_0$ has a compact support ${\rm Supp}(g_0)$ with
${\rm Supp}(g_0) \subset X \setminus \{ a, b \}$, we have
$$\iint _X \frac{df}{f} \wedge dg_0 = -
\iint _{X \setminus \{ a, b\}}d\left( g_0 \frac{df}{f}\right) = 0$$
by Stokes' theorem.
Considering the support of $df/f$ and the above equality, we obtain
\begin{align*}
\frac{1}{2\pi \sqrt{-1}}\iint _X\frac{df}{f}\wedge \rho ^{*}\omega & =
\frac{1}{2\pi \sqrt{-1}}\iint _X\frac{df}{f}\wedge dg \\
& = \frac{1}{2 \pi \sqrt{-1}}\left(
\iint _X \frac{df}{f}\wedge dg_1 + \iint _X \frac{df}{f}\wedge dg_2
\right).
\end{align*}
Now we have
$$\iint _X \frac{df}{f} \wedge dg_1 = \iint _{V_1}
\frac{df}{f}\wedge dg_1 = \lim _{\varepsilon \to 0}
\int _{|t_1| = \varepsilon }g_1 \frac{df}{f} =
- 2 \pi \sqrt{-1}g(a).$$
Similarly we have
$$\iint _X \frac{df}{f}\wedge dg_2 = 2 \pi \sqrt{-1}g(b).$$
Thus we obtain
$$\frac{1}{2\pi \sqrt{-1}}\iint _X \frac{d f}{f}\wedge \rho ^{*} \omega =
g(b) - g(a) = \int _c\rho^{*}\omega .$$
\par
(b) Next we consider the general case.
\par
There exist a partition
$$0 = t_0 < t_1 < \cdots < t_n = 1$$
of the closed interval $[0,1]$ and coordinate neighbourhoods $(U_j, z_j)$,
$j=1, \dots ,n$ such that\\
(i) $c([t_{j-1}, t_j]) \subset U_j \subset U,$\\
(ii) $z_j(U_j) \subset {\mathbb C}$ is the unit disc.\\
Let $c_j := c|[t_{j-1}, t_j]$.
By (a) we have a weak solution $f_j$ of $\partial c_j$ such that $f_j|(X\setminus U_j) = 1$ and
$$\int_{c_j}\rho ^{*} \omega = \frac{1}{2\pi \sqrt{-1}} \iint_X
\frac{d f_j}{f_j} \wedge \rho ^{*} \omega$$
for any $\omega \in H^0(X_{\mathfrak m}, {\mathcal E}^{(1)}_{\mathfrak m})$ with
$d\omega = 0$.
Then $f := f_1 f_2 \cdots f_n$ is a function with the desired
property.
\end{proof}

\begin{lemma}
For any $D \in {\rm Div}(X_{\mathfrak m})$ the following two conditions are equivalent.\\
(1) There exists a meromorphic function $g$ on $X$ such that $D = (g)$ and we have a branch
$f$ of $\log g$ defined in a neighbourhood of $S$ with the property
$$\sum_{P\in \rho ^{-1}(Q)} {\rm Res}_P(f\omega ) = 0$$
for any point $Q \in \overline{S}$ and for any
$\omega \in H^0(X, \rho ^{*}\Omega _{\mathfrak m})$.\\
(2) There exist a meromorphic function $g$ on $X$ and a multiconstant $c(\overline{S})$ such that
$$D = (g)\quad \text{and\quad $g \equiv c(\overline{S})$\ ${\rm mod}\ {\mathfrak m}$}.$$
\end{lemma}

\begin{proof}
Let $g$ be a meromorphic function on $X$ satisfying the condition (1). Take any $Q \in \overline{S}$.
Let $\rho ^{-1}(Q) = \{ P_1, \dots ,P_N\}$. If $N = 1$, then we set
$c_Q := g(P_1) \not= 0$. 
Next we consider the case $N \geq 2$. For any $j = 2, \dots , N$, there exists a
meromorphic $1$-form $\omega $ such that it is holomorphic on $X \setminus \{ P_1, P_j\}$ and
has the simple poles $P_1$ and $P_j$.
We may assume that ${\rm Res}_{P_1}(\omega ) = 1$. Then it naturally follows from the definition that
${\rm Res}_{P_j}(\omega ) = -1$. We note that $\omega \in H^0(X, \rho ^{*}\Omega _{\mathfrak m})$.
By the assumption we have
$$0 = \sum_{P \in \rho ^{-1}(Q)} {\rm Res}_P(f\omega ) =
{\rm Res}_{P_1}(f\omega ) + {\rm Res}_{P_j}(f\omega ) = f(P_1) - f(P_j),$$
where $f$ is a branch of $\log g$ in a neighbourhood of $S$. Then we obtain
$$g(P_1) = \exp (f(P_1)) = \exp (f(P_j)) = g(P_j).$$
Since $j$ is arbitrary, we can set
$$c_Q := g(P_1) = g(P_2) = \cdots = g(P_N).$$
We show that ${\rm ord}_P(g - c_Q) \geq {\mathfrak m}(P)$ for any
$P \in \rho ^{-1}(Q)$. If ${\mathfrak m}(P) =1$, it is just $g(P) = c_Q$.
Suppose that ${\mathfrak m}(P) \geq 2$. For any $k$ with $2 \leq k \leq {\mathfrak m}(P)$,
there exists a meromorphic $1$-form $\omega _k$ such that it is holomorphic on
$X \setminus \{ P\}$ and the singularity at $P$ is $1/z^k$, where $z$ is a local coordinate with
$z(P) = 0$. We note that $\omega _k \in H^0(X, \rho ^{*}\Omega _{\mathfrak m})$.
Since $f$ is a branch of $\log g$, we have the following representation of $f$ in a
neighbourhood of $P$
$$f(z) = f(P) + \frac{g'(P)}{g(P)}z + \frac{1}{2!}\left(\frac{g'}{g}\right)'(P) z^2 +
\cdots + \frac{1}{(k-1)!}\left(\frac{g'}{g}\right)^{(k-2)}(P) z^{k-1} + \cdots .$$
We show that $g^{(\ell )}(P) = 0$ for $\ell = 1, \dots , {\mathfrak m}(P) -1$, by the
induction on $\ell $. Considering the above representation of $f$ for $k = 2$, we obtain
$$0 = {\rm Res}_P(f\omega _2) = \frac{g'(P)}{g(P)},$$
hence $g'(P) = 0$. Let $2 \leq \ell \leq {\mathfrak m}(P) - 1$ and assume that
$g'(P) = \cdots = g^{(\ell - 1)}(P) = 0$. Then we have
$$f(z) = f(P) + \frac{1}{\ell !}\left(\frac{g'}{g}\right)^{(\ell -1)}(P) z^{\ell } + \cdots .$$
Since
$$0 = {\rm Res}_P(f\omega _{\ell + 1}) = \frac{1}{\ell !}\left(\frac{g'}{g}\right)^{(\ell -1)}(P),$$
we have $g^{(\ell )}(P) = 0$. Therefore we conclude that
${\rm ord}_P(g - c_Q) \geq {\mathfrak m}(P)$. This means that
$g \equiv c(\overline{S})\ {\rm mod}\ {\mathfrak m}$.
\par
Conversely, we assume that there exists a meromorphic function $g$ on $X$ with $D = (g)$ and
$g \equiv c(\overline{S})\ {\rm mod}\ {\mathfrak m}$ for some multiconstant $c(\overline{S})$.
Take any $Q \in \overline{S}$ and any $P \in \rho ^{-1}(Q)$. Let $z$ be a local coordinate at $P$
with $z(P) = 0$. We can write
$$g(z) = c_Q + a_{{\mathfrak m}(P)}z^{{\mathfrak m}(P)} + \cdots .$$
Then we have a branch $f$ of $\log g$ in a neighbourhood of $P$ such that
$$f(z) = \log c_Q + \frac{a_{{\mathfrak m}(P)}}{c_Q}z^{{\mathfrak m}(P)} + \cdots .$$
For any $\omega \in H^0(X, \rho ^{*}\Omega _{\mathfrak m})$ and for any $Q \in \overline{S}$,
we have
$$\sum_{P \in \rho ^{-1}(Q)}{\rm Res}_P(f\omega ) = \log c_Q
\left(\sum_{P\in \rho ^{-1}(Q)}{\rm Res}_P(\omega )\right) = 0.$$
\end{proof}

\subsection{Proof of Theorem 4.2}
We first show the necessity. Assume that there exists a $1$-chain $c \in C_1(X \setminus S)$ with
$\partial c = D$ such that
$$\int _c \rho ^{*}\omega = 0$$
for any $\omega \in H^0(X_{\mathfrak m}, \Omega _{\mathfrak m})$. By Lemma 4.4 we have a weak
solution $f$ of $D = \partial c$ such that $f|(X\setminus U) = 1$ and
$$\frac{1}{2\pi \sqrt{-1} }\iint _X \frac{d f}{f}\wedge \rho ^{*} \omega =
\int _c \rho ^{*}\omega $$
for every $\omega \in H^0(X_{\mathfrak m}, {\mathcal E}^{(1)}_{\mathfrak m})$ with $d\omega = 0$,
where $U$ is an open neighbourhood of the support of $c$
with $U \subset \subset X \setminus S$.
Since $H^0(X_{\mathfrak m}, \Omega _{\mathfrak m}) \subset H^0(X_{\mathfrak m}, {\mathcal E}^{(1)}_{\mathfrak m})$,
we obtain for every $\omega \in H^0(X_{\mathfrak m}, \Omega _{\mathfrak m})$
$$0 = \int _c \rho ^{*}\omega = \frac{1}{2\pi \sqrt{-1}} \iint _X \frac{d f}{f} \wedge \rho ^{*}\omega
= \frac{1}{2\pi \sqrt{-1}} \iint _X \frac{\overline{\partial } f}{f} \wedge \rho ^{*}\omega $$
by the assumption. Note that $\sigma := \overline{\partial }f /f$ is a $C^{\infty }$ $(0,1)$-form on $X$.
It follows from the above equality that
$$\frac{1}{2\pi \sqrt{-1}} \iint _X \sigma \wedge \eta = 0$$
for every $\eta \in H^0(X, \Omega )$, for we have $H^0(X, \Omega ) \subset
\rho ^{*} H^0(X_{\mathfrak m}, \Omega _{\mathfrak m})$.
Then there exists a $C^{\infty }$ function $g$ on $X$ with 
$\overline{\partial }g = \sigma = \overline{\partial }f/f$ (Corollary (19.10) in \cite{ref10}).
If we set $F := e^{-g}f$, then it is also a weak solution of $D$. Furthermore, $F$ is meromorphic
for $\overline{\partial }F = 0$. Since $f = 1$ on a neighbourhood of $S$, we have $F = e^{-g}$ there.
Then, $-g$ is a branch of $\log F$ defined in the neighbourhood of $S$. For any 
$\omega \in H^0(X_{\mathfrak m}, \Omega _{\mathfrak m})$ we have
$$\frac{1}{2\pi \sqrt{-1}}\iint _X \overline{\partial }g \wedge \rho ^{*}\omega =
\frac{1}{2\pi \sqrt{-1}}\iint _X \frac{\overline{\partial }f}{f} \wedge \rho ^{*}\omega
= \int _c\rho ^{*}\omega = 0.$$
Let $Q \in \overline{S}$. We set $\rho ^{-1}(Q) = \{ P_1, \dots , P_N \}$.
Take a sufficiently small $\varepsilon >0$. We denote by $B_j(\varepsilon )$ the disc centered at
$P_j$ with radius $\varepsilon $.
Since
\begin{eqnarray*}
\frac{1}{2\pi \sqrt{-1}}\iint _X \overline{\partial }g \wedge \rho ^{*}\omega & = &
\lim _{\varepsilon \to 0}\frac{1}{2\pi \sqrt{-1}} \iint _{X \setminus \left( \cup _{j=1}^{N}B_j(\varepsilon )\right)}
\overline{\partial }g \wedge \rho ^{*}\omega \\
& = & \lim _{\varepsilon \to 0}\left(\sum_{j=1}^{N}\frac{1}{2\pi \sqrt{-1}}
\int _{\partial B_j(\varepsilon )}(-g)\rho ^{*}\omega \right)\\
& = & \sum _{P \in \rho ^{-1}(Q)}{\rm Res}_P\left((-g)\rho ^{*}\omega \right),
\end{eqnarray*}
we obtain
$$\sum _{P \in \rho ^{-1}(Q)}{\rm Res}_P\left((-g)\rho ^{*}\omega \right) = 0.$$
Then the condition (1) in Lemma 4.5 is fulfilled for
$$\rho ^{*}H^0(X_{\mathfrak m}, \Omega _{\mathfrak m}) = H^0(X, \rho ^{*}\Omega _{\mathfrak m}).$$
Therefore, there exists a meromorphic function $h$
on $X$ and a multiconstant $c(\overline{S})$ such that
$$D = (h)\quad \text{and\quad $h \equiv c(\overline{S})\ {\rm mod}\ {\mathfrak m}$}.$$
\par
Next we prove the sufficiency. Ler $f$ be a meromorphic function on $X$ with
$f \equiv c(\overline{S})\ {\rm mod}\ {\mathfrak m}$ for some multiconstant
$c(\overline{S})$ such that $D = (f)$. Take any $\omega \in H^0(X_{\mathfrak m}, \Omega _{\mathfrak m})$.
Let $F : X \longrightarrow {\mathbb P}^1$ be the holomorphic map defined by $f$. We note that
$d := \deg F \geq \sum _{P \in \rho ^{-1}(Q)}{\mathfrak m}(P) \geq 2$ for any
$Q \in \overline{S}$. The trace
${\rm Trace}(\rho ^{*}\omega )$ of $\rho ^{*}\omega $ by $F$ is a meromorphic $1$-form on ${\mathbb P}^1$.
\smallskip
\\ 
{\bf Assertion.}
${\rm Trace}(\rho ^{*}\omega ) = 0.$
\begin{proof}
Note that ${\rm Trace}(\rho ^{*}\omega )$ is holomorphic except
$F(S) = \{ c_Q ; Q \in \overline{S} \}$. Let $Q \in \overline{S}$.
Take any $P \in \rho ^{-1}(Q)$. Let $m (\geq {\mathfrak m}(P))$ be the multiplicity
of $F$ at $P$. We can take local coordinates $t$ and $w$ centered at $c_Q$ and
at $P$ respectively by which $F$ is represented as $t = w^m$. We may assume that
there exists a meromorphic function $h(w)$ in a neighbourhood of $P$ such that
$\rho ^{*}\omega = h(w)dw$ and $h(w)$ has the following Laurent expansion
$$\sum _{n \geq -{\mathfrak m}(P)}c_n w^n.$$
The preimages of $t = w^m (\not= 0)$ are $\zeta ^{i}w$
$(i = 0, 1, \dots , m-1)$, where $\zeta = \exp \left(\sqrt{-1}\frac{2\pi }{m}\right)$.
Since $n \geq - {\mathfrak m}(P)$ and $m \geq {\mathfrak m}(P)$,
we have
\begin{eqnarray*}
\sum _{i=0}^{m-1}
\frac{h(\zeta ^i w)}{m w^{m-1}} & = &
\frac{1}{m}\sum _{n \geq - {\mathfrak m}(P)}c_n
\left( \sum _{i=0}^{m-1}\zeta ^{i(n-m+1)}\right)w^{n-m+1}\\
& = & \sum _{k \geq 0}c_{km-1}t^{k-1}.
\end{eqnarray*}
Noting $c_{-1} = {\rm Res}_P(\rho ^{*}\omega )$, we obtain the expression
of ${\rm Trace}(\rho ^{*}\omega ) $ at $c_Q$ as follows:
$${\rm Trace}(\rho ^{*}\omega ) = \left( \left(
\sum _{P \in \rho ^{-1}(Q)}{\rm Res}_P(\rho ^{*}\omega )\right)
\frac{1}{t} + \text{holomorphic part} \right) dt.$$
Then it is holomorphic at $c_Q$ for $\omega \in H^0(X_{\mathfrak m}, \Omega _{\mathfrak m})$.
Therefore it must be zero.
\end{proof}
\par
Let $\{ a_1, \dots , a_r \}$ be the branch points of $F$. We set
$$Y := {\mathbb P}^1 \setminus \{ F(a_1), \dots , F(a_r) \}.$$
We take a curve $\gamma $ from $0$ to $\infty $ contained in
$Y \setminus F(S)$. Let $c_1, \dots , c_d$ be the curves in $X$ consisting of
$F^{-1}(\gamma )$. If we define
$$c := c_1 + \cdots + c_d,$$
then $c \in C_1(X \setminus S)$, $\partial c = D$ and
$$\int _c \rho ^{*}\omega = \int _{\gamma }{\rm Trace}(\rho ^{*}\omega )
= 0$$
for any $\omega \in H^0(X_{\mathfrak m}, \Omega _{\mathfrak m})$.
This completes the proof.

\section{Varieties Associated to Singular Curves}

\subsection{Non-singular case}
Let $X$ be a compact Riemann surface of genus $g$. We denote by ${\rm Div}(X)$ the group
of divisors on $X$. It has subgroups ${\rm Div}^0(X)$ and ${\rm Div}_P(X)$ consisting of
divisors of degree 0 and principal divisors, respectively. We define
$$\overline{{\rm Div}(X)} := {\rm Div}(X)/{\rm Div}_P(X)\quad
\text{and\quad $\overline{{\rm Div}^0(X)} := {\rm Div}^0(X)/{\rm Div}_P(X)$}.$$
Let $\{ \omega _1, \dots , \omega _g \}$ be a basis of $H^0(X, \Omega )$.
We take a canonical
homology basis $\{ \alpha _1, \beta _1, \dots , \alpha _g, \beta _g \}$ of $X$.
We denote by $\Lambda $
the lattice determined by these bases. We define a period map
$\varphi : X \longrightarrow A := {\mathbb C}^g/\Lambda $ with base point $P_0 \in X$ by
$$\varphi (P) := \left( \int _{P_0}^{P}\omega _1 , \dots , \int _{P_0}^{P}\omega _g\right)\quad
{\rm mod}\ \ \Lambda .$$
The map $\varphi $ is extended onto ${\rm Div}(X)$, and we obtain an isomorphism
$$\overline{\varphi } : \overline{{\rm Div}^0(X)} \longrightarrow A$$
by Abel's Theorem and the solution of Jacobi Inversion Problem. In this case, $A$ is called the
Albanese variety of $X$. We write ${\rm Alb}(X) = A$.
\par
Let ${\rm Pic}(X) = H^1(X, {\mathcal O}_X^{*})$ be the group of isomorphic classes of holomorphic line
bundles on $X$. We denote by ${\rm Pic}^0(X)$ the subgroup consisting of classes of topologically
trivial ones. It is called the Picard variety of $X$. Since every holomorphic line bundle on $X$ has
a non-trivial meromorphic section, we have
$${\rm Pic}(X) \cong \overline{{\rm Div}(X)},$$
hence
$${\rm Pic}^0(X) \cong \overline{{\rm Div}^0(X)}.$$
Therefore, we obtain isomorphisms
$${\rm Pic}^0(X) \cong \overline{{\rm Div}^0(X)} \cong {\rm Alb}(X) = A.$$
Furthermore, we can give a direct isomorphism ${\rm Alb}(X) \cong {\rm Pic}^0(X)$ as complex
Lie groups.
The Jacobi variety $J(X)$ of $X$ is $\overline{{\rm Div}^0(X)}$ with the complex structure induced
from the above isomorphisms.

\subsection{Divisor classes and Picard varieties of singular curves}
Results in this section are stated in \cite{ref2} without details.
Let $X$ be a compact Riemann surface. Consider a singular curve
$X_{\mathfrak m}$ constructed by a modulus ${\mathfrak m}$ with support $S$ from $X$ as
in Section 2. 
Since $H^2(X_{{\mathfrak m}},{\mathcal O}_{{\mathfrak m}}) = 0$, we obtain the following
exact sequence by a standard argument
\begin{equation}
0 \longrightarrow H^1(X_{\mathfrak m},{\mathbb Z}) \longrightarrow H^1(X_{\mathfrak m}, {\mathcal O}_{\mathfrak m})
\longrightarrow H^1(X_{\mathfrak m}, {\mathcal O}_{\mathfrak m}^{*}) \stackrel{c}{\longrightarrow}
H^2(X_{\mathfrak m}, {\mathbb Z}) \longrightarrow 0.
\end{equation}
We define
$${\rm Pic}(X_{\mathfrak m}) := H^1(X_{\mathfrak m}, {\mathcal O}_{\mathfrak m}^{*})$$
and
$${\rm Pic}^0(X_{\mathfrak m}) := \{ L \in H^1(X_{\mathfrak m}, {\mathcal O}_{\mathfrak m}^{*}) ;
c(L) = 0 \}.$$
Let ${\mathcal M}_{\mathfrak m}$ be the quotient sheaf of ${\mathcal O}_{\mathfrak m}$. We denote
the divisor sheaf on $X_{\mathfrak m}$ by
$${\mathcal D}_{\mathfrak m} := {\mathcal M}_{\mathfrak m}^{*} / {\mathcal O}_{\mathfrak m}^{*}.$$
From an exact sequence
$$\{ 1 \} \longrightarrow {\mathcal O}_{\mathfrak m}^{*} \longrightarrow {\mathcal M}_{\mathfrak m}^{*}
\longrightarrow {\mathcal D}_{\mathfrak m} \longrightarrow \{ 1 \}$$
we obtain the following exact sequence
\begin{equation}
\begin{split}
\cdots \longrightarrow & H^0(X_{\mathfrak m}, {\mathcal M}_{\mathfrak m}^{*}) \longrightarrow
H^0(X_{\mathfrak m}, {\mathcal D}_{\mathfrak m}) \longrightarrow \\
 & H^1(X_{\mathfrak m}, {\mathcal O}_{\mathfrak m}^{*})
\longrightarrow H^1(X_{\mathfrak m}, {\mathcal M}_{\mathfrak m}^{*}) \longrightarrow \cdots .
\end{split}
\end{equation}

\begin{lemma}
For every holomorphic vector bundle $E \longrightarrow X_{\mathfrak m}$ we have
$$\dim H^1(X_{\mathfrak m}, {\mathcal O}_{\mathfrak m}(E)) < \infty ,$$
where ${\mathcal O}_{\mathfrak m}(E)$ is the sheaf of germs of holomorphic sections of $E$.
\end{lemma}

\begin{proof}
Since $\widetilde{E} := \rho ^{*}E$ is a holomorphic vector bundle over $X$, we have
$$\dim H^1(X, {\mathcal O}(\widetilde{E})) < \infty .$$
Let $\overline{S} = \{ Q_1, \dots , Q_N \}$. If we set
$${\mathcal F}_E := \rho _{*}{\mathcal O}(\widetilde{E})/{\mathcal O}_{\mathfrak m}(E),$$
then
$$H^0(X_{\mathfrak m}, {\mathcal F}_E) \cong \bigoplus _{j=1}^{N}
{\mathbb C}^{r\ell _j},$$
where $r = {\rm rank}\ E$ and 
$\ell _j = \sum _{P \in \rho ^{-1}(Q_j)} {\mathfrak m}(P) -1$. 
From an exact sequence
$$0 \longrightarrow {\mathcal O}_{{\mathfrak m}}(E)
\longrightarrow \rho _{*}{\mathcal O}(\widetilde{E})
\longrightarrow {\mathcal F}_E \longrightarrow 0$$
we obtain the following exact sequence
\begin{equation*}
\begin{split}
\cdots \longrightarrow & H^0(X_{\mathfrak m}, \rho _{*}{\mathcal O}(\widetilde{E}))
\longrightarrow H^0(X_{\mathfrak m}, {\mathcal F}_{\mathfrak m}) \longrightarrow \\
& H^1(X_{\mathfrak m}, {\mathcal O}_{\mathfrak m}(E)) \longrightarrow H^1(X_{\mathfrak m}, \rho _{*}
{\mathcal O}(\widetilde{E})) \longrightarrow 0.
\end{split}
\end{equation*}
Since
$$H^q(X_{\mathfrak m}, \rho _{*}{\mathcal O}(\widetilde{E})) \cong
H^q(X, {\mathcal O}(\widetilde{E})) \ \ 
\text{for\ \ $q \geq 0$},$$
we have
\begin{eqnarray*}
\lefteqn{\dim H^1(X_{\mathfrak m}, {\mathcal O}_{\mathfrak m}(E))}\\
&=& \dim H^1(X, {\mathcal O}(\widetilde{E})) + \dim H^0(X_{\mathfrak m}, {\mathcal F}_E)
- \dim H^0(X, {\mathcal O}(\widetilde{E}))\\
& < & \infty .
\end{eqnarray*}
\end{proof}

\begin{proposition}
Let $\pi : E \longrightarrow X_{\mathfrak m}$ be a holomorphic vector bundle over $X_{\mathfrak m}$.
For any $Q \in X_{\mathfrak m} \setminus \overline{S}$ there exists a meromorphic section
$s$ of $E$ over $X_{\mathfrak m}$ such that $Q$ is the only pole of $s$.
\end{proposition}

\begin{proof}
By Lemma 5.1 we have
$$k := \dim H^1(X_{\mathfrak m}, {\mathcal O}_{\mathfrak m}(E)) < \infty .$$
Let $(U_1, z)$ be a coordinate neighbourhood of $Q$ with $z(Q) = 0$ such that
$z(U_1)$ is a disc in ${\mathbb C}$ and $E$ has a trivialization on $U_1$.
We set $U_2 := X_{\mathfrak m} \setminus \{ Q \}$. Then $U_2$ is a 1-dimensional 
Stein space. Consider a Stein covering ${\mathfrak U} = \{ U_1, U_2\}$ of $X_{\mathfrak m}$.
Let $\pi ^{-1}(U_1) \cong U_1 \times {\mathbb C}^r$ be the trivialization of $E$ on $U_1$,
where $r = {\rm rank}\ E$. We define a holomorphic section $s_j(z) := (z^{-j}, \dots ,
z^{-j})$ of $E$ on $U_1 \cap U_2 = U_1 \setminus \{ Q \}$ for $j \in {\mathbb N}$.
Since ${\mathfrak U}$ is a Stein covering, we have
$$H^1(X_{\mathfrak m}, {\mathcal O}_{\mathfrak m}(E)) \cong H^1({\mathfrak U}, {\mathcal O}_{\mathfrak m}(E)).$$
We denote by $\zeta _j \in Z^1({\mathfrak U}, {\mathcal O}_{\mathfrak m}(E))$ the cocycle
given by $s_j$. We see that $\zeta _1, \dots , \zeta _{k+1}$ are linearly dependent modulo
the coboundaries for $k = \dim H^1(X_{\mathfrak m}, {\mathcal O}_{\mathfrak m}(E))$.
Then there exist $(c_1, \dots c_{k+1}) \in {\mathbb C}^{k+1} \setminus \{ 0 \}$ and
$\eta = (f_1, f_2) \in C^0({\mathfrak U}, {\mathcal O}_{\mathfrak m}(E))$ such that
$$\sum _{j=1}^{k+1}c_j\zeta _j = \delta \eta .$$
If we set 
\begin{equation*}
f := \begin{cases}
    f_1 + \sum _{j=1}^{k+1}c_j s_j & \text{on $U_1$,}\\
    f_2  & \text{on $U_2$},
    \end{cases}
\end{equation*}
then $f$ is a meromorphic section of $E$ and has the only pole at $Q$.
\end{proof}

We note that a holomorphic line bundle over $X_{{\mathfrak m}}$ is given
by an element of $H^0(X_{{\mathfrak m}},{\mathcal D}_{{\mathfrak m}})$
if and only if it has a non-trivial meromorphic section.
Applying the above proposition to holomorphic line bundles over $X_{\mathfrak m}$,
we see that the map $H^0(X_{\mathfrak m}, {\mathcal D}_{\mathfrak m})
\longrightarrow H^1(X_{{\mathfrak m}}, {\mathcal O}_{{\mathfrak m}}^{*})$
in (5.2) is surjective and
then
$$H^1(X_{\mathfrak m}, {\mathcal O}_{\mathfrak m}^{*}) \cong
H^0(X_{\mathfrak m}, {\mathcal D}_{\mathfrak m})/
H^0(X_{\mathfrak m}, {\mathcal M}_{\mathfrak m}^{*}).$$
\par
We investigate the divisor sheaf ${\mathcal D}_{{\mathfrak m}}$ more closely.
Let $L \in H^1(X_{\mathfrak m}, {\mathcal O}_{\mathfrak m}^{*})$.
We denote by ${\mathcal O}_{\mathfrak m}(L)$ the sheaf of germs of holomorphic
sections of $L$ over $X_{\mathfrak m}$. Similarly, 
${\mathcal E}_{\mathfrak m}^{p,q}(L)$ is the sheaf of germs of $L$-valued
$C^{\infty}$ $(p,q)$-forms over $X_{\mathfrak m}$. The following exact
sequence
$$0 \longrightarrow {\mathcal O}_{\mathfrak m}(L) \longrightarrow
{\mathcal E}_{\mathfrak m}^{0,0}(L) \stackrel{\overline{\partial}}{\longrightarrow}
{\mathcal E}_{\mathfrak m}^{0,1}(L) \longrightarrow 0$$
is a fine resolution of ${\mathcal O}_{\mathfrak m}(L)$.
Then we have
$$H^1(X_{\mathfrak m}, {\mathcal O}_{\mathfrak m}(L)) \cong
H^0(X_{\mathfrak m},{\mathcal E}_{\mathfrak m}^{0,1}(L))/
\overline{\partial}H^0(X_{\mathfrak m},{\mathcal E}_{\mathfrak m}^{0,0}(L)),$$
$$H^q(X_{\mathfrak m},{\mathcal O}_{\mathfrak m}(L)) = 0\quad
\text{for \ \ $q \geq 2$}.$$
Let ${\mathcal E}_{\mathfrak m}^{r}$ be the sheaf of germs of $C^{\infty }$
$r$-forms on $X_{\mathfrak m}$.
Then we obtain
$$H^2(X_{\mathfrak m},{\mathbb C}) \cong H^0(X_{\mathfrak m},{\mathcal E}_{\mathfrak m}^2)
/d H^0(X_{\mathfrak m},{\mathcal E}_{\mathfrak m}^1)$$
by the following exact sequence
$$0 \longrightarrow {\mathbb C} \longrightarrow {\mathcal E}_{\mathfrak m}^0
\stackrel{d}{\longrightarrow} {\mathcal E}_{\mathfrak m}^1 \stackrel{d}{\longrightarrow}
{\mathcal E}_{\mathfrak m}^2 \longrightarrow 0,$$
because sheaves ${\mathcal E}_{{\mathfrak m}}^{r}$
$(r = 0, 1, 2)$ are also fine.
It follows from (5.1) that a sequence
$$0 \longrightarrow H^1(X_{\mathfrak m},{\mathcal O}_{\mathfrak m})/
H^1(X_{\mathfrak m},{\mathbb Z}) \longrightarrow H^1(X_{\mathfrak m},{\mathcal O}_
{\mathfrak m}^{*}) \stackrel{c}{\longrightarrow} H^2(X_{\mathfrak m},{\mathbb Z})
\cong {\mathbb Z} \longrightarrow 0$$
is exact. The Chern class $c(L)$ of $L$ is considered as a class in
$H^2(X_{\mathfrak m},{\mathbb C}) \cong {\mathbb C}$. The pull-back
$\rho ^{*}\varphi $ of any $\varphi \in H^0(X_{\mathfrak m},{\mathcal E}_{\mathfrak m}^2)$
is a $C^{\infty}$ 2-form on $X$. We define an epimorphism
$\Phi : H^0(X_{\mathfrak m},{\mathcal E}_{\mathfrak m}^2) \longrightarrow {\mathbb C}$
by
$$\Phi (\varphi ) := \iint _X \rho ^{*}\varphi \quad
\text{for any $\varphi \in H^0(X_{\mathfrak m}, {\mathcal E}_{\mathfrak m}^2)$}.$$
Since $d$ and $\rho ^{*}$ are commutative, we have
$$d H^0(X_{\mathfrak m},{\mathcal E}_{\mathfrak m}^1) \subset
{\rm Ker}\ \Phi .$$
Then we obtain an epimorphism
$$\sigma : H^0(X_{\mathfrak m}, {\mathcal E}_{\mathfrak m}^2)/
d H^0(X_{\mathfrak m},{\mathcal E}_{\mathfrak m}^1) \longrightarrow
H^0(X_{\mathfrak m},{\mathcal E}_{\mathfrak m}^2)/{\rm Ker}\ \Phi .$$
However, we have $H^0(X_{\mathfrak m},{\mathcal E}_{\mathfrak m}^2)/{\rm Ker}\ \Phi
\cong {\mathbb C}$ and
$$H^0(X_{\mathfrak m},{\mathcal E}_{\mathfrak m}^2)/
d H^0(X_{\mathfrak m},{\mathcal E}_{\mathfrak m}^1) \cong
H^2(X_{\mathfrak m},{\mathbb C}) \cong {\mathbb C}.$$
Then, $\sigma $ is an isomorphism and
${\rm Ker}\ \Phi  = d  H^0(X_{\mathfrak m},{\mathcal E}_{\mathfrak m}^1)$.
\par
We suppose that $L \in H^1(X_{\mathfrak m},{\mathcal O}_{\mathfrak m}^{*})$
has a representative $(\xi _{\alpha \beta }) \in Z^1({\mathfrak U},{\mathcal O}_
{\mathfrak m}^{*})$ for some sufficiently fine covering
${\mathfrak U} = \{ U_{\alpha }\}$ of $X_{\mathfrak m}$. We set
$$\sigma _{\alpha \beta } := \frac{1}{2\pi \sqrt{-1}} \log \xi _{\alpha \beta }.$$
Then $(\sigma _{\alpha \beta } )\in C^1({\mathfrak U},{\mathcal O}_{\mathfrak m})$.
The Chern class $c(L)$ of $L$ is the class given by the coboundary $(c_{\alpha \beta \gamma})$ of $(\sigma _{\alpha \beta })$, where
$$c_{\alpha \beta \gamma } = \sigma _{\beta \gamma} - \sigma _{\alpha \gamma} + \sigma _{\alpha \beta }
= \sigma _{\alpha \beta } + \sigma _{\beta \gamma}
+ \sigma _{\gamma \alpha}.$$
We consider $(c_{\alpha \beta \gamma})\in Z^2({\mathfrak U},{\mathbb Z})$
as a 2-cocycle in $Z^2({\mathfrak U},{\mathcal E}_{\mathfrak m}^0)$. Then there
exists $(\sigma '_{\alpha \beta})\in C^1({\mathfrak U},{\mathcal E}_{\mathfrak m}^0)$
such that
$$c_{\alpha \beta \gamma} = \sigma '_{\alpha \beta} + \sigma '_{\beta \gamma}
+ \sigma '_{\gamma \alpha},$$
for $H^2(X_{{\mathfrak m}}, {\mathcal E}_{{\mathfrak m}}^0) = 0$.
Since $(c_{\alpha \beta \gamma })$ is $d$-closed, we have
$(d\sigma '_{\alpha \beta}) \in Z^1({\mathfrak U},{\mathcal E}_{\mathfrak m}^1)$.
Therefore, we can take $(\tau _{\alpha}) \in C^0({\mathfrak U},{\mathcal E}_{\mathfrak m}^1)$ with $d\sigma '_{\alpha \beta} = \tau _{\beta} - \tau _{\alpha}$.
Hence we obtain a global 2-form $\varphi$ defined by $(d\tau _{\alpha})$.
This is a 2-form which is a representative of $c(L)$ in
$H^0(X_{{\mathfrak m}}, {\mathcal E}_{{\mathfrak m}}^2)$. Then we write
$\varphi(L) = \varphi$. In this case $\rho ^{*}\varphi(L)$ is a representative of
the Chern class $c_X(\rho ^{*}L)$ of $\rho ^{*}L$ on $X$. We know
$$\iint _X \rho^{*}\varphi(L) = c_X(\rho^{*}L) \in {\mathbb Z}.$$
We sum up the above facts as follows. The Chern class $c(L)$ of any
$L \in H^1(X_{\mathfrak m},{\mathcal O}_{\mathfrak m}^{*})$ is given by
$$c(L) = \iint _X\rho ^{*}\varphi(L),$$
where $\varphi(L) \in H^0(X_{\mathfrak m},{\mathcal E}_{\mathfrak m}^2)$ is
a representative of $c(L)$.
\par
We obtain the following lemma by a slight modification of the proof in
the non-singular case.

\begin{lemma}
Let $L \in H^1(X_{\mathfrak m},{\mathcal O}_{\mathfrak m}^{*})$.
Assume that $L$ is represented by $(\zeta _{\alpha \beta}) \in
Z^1({\mathfrak U},{\mathcal O}_{\mathfrak m}^{*})$ for some open covering
${\mathfrak U} = \{ U_{\alpha }\}$ of $X_{\mathfrak m}$. For any $\alpha$
there exists a nowhere-vanishing $C^{\infty}$ function $r_{\alpha }$ on
$U_{\alpha }$ such that $\log r_{\alpha}$ has a well-defined branch on
$U_{\alpha}$ and
$$r_{\alpha} = r_{\beta}|\zeta _{\beta \alpha}|^2\quad
\text{on\quad$U_{\alpha}\cap U_{\beta}$}.$$
Then, $\varphi := \frac{1}{2\pi \sqrt{-1}}\partial \overline{\partial}
\log r_{\alpha}$ is a well-defined 2-form on $X_{\mathfrak m}$ and we have
$$c(L) = \iint _X \rho ^{*}\varphi = \frac{1}{2\pi \sqrt{-1}}
\iint _X \partial \overline{\partial}\log (r_{\alpha}\circ \rho ).$$
\end{lemma}

\begin{remark}
Functions $\{ r_{\alpha}\}$ satisfying the condition in the above lemma
always exist. We can take such functions taking strictly positive values.
\end{remark}
\par
Let $L \in H^1(X_{\mathfrak m},{\mathcal O}_{\mathfrak m}^{*})$.
We denote by ${\mathcal M}_{\mathfrak m}^{*}(L)$ the sheaf of germs of non-zero
meromorphic sections of $L$. For any $f \in H^0(X_{\mathfrak m},{\mathcal M}_
{\mathfrak m}^{*}(L))$ we set
$${\rm ord}_Q(f) = \sum _{P \in \rho ^{-1}(Q)} {\rm ord}_P(f\circ \rho )\quad
\text{for $Q \in X_{\mathfrak m}$}.$$

\begin{proposition}
Let $L \in H^1(X_{\mathfrak m},{\mathcal O}_{\mathfrak m}^{*})$.
Take any $f \in H^0(X_{\mathfrak m},{\mathcal M}_{\mathfrak m}^{*}(L))$.
Then we have
$$c(L) = \sum _{Q \in X_{\mathfrak m}}{\rm ord}_Q(f).$$
\end{proposition}

\begin{proof}
Let $\{ Q_1, \dots , Q_N\}$ be the set of all $Q \in X_{\mathfrak m}$ such that
${\rm ord}_Q(f) \not= 0$. We show that
$$c(L) = \sum _{i=1}^{N} {\rm ord}_{Q_i}(f).$$
\par
Suppose that $L$ is given by $(\zeta _{\alpha \beta})\in Z^1({\mathfrak U},
{\mathcal O}_{\mathfrak m}^{*})$.
We may assume that the open covering ${\mathfrak U} = \{ U_{\alpha}\}$ of
$X_{\mathfrak m}$ satisfies the following condition:\\
For any $i=1, \dots , N$ there exists an open neighbourhood $V_i$ of $Q_i$ such that
$V_i \subset U_{\alpha _i}$ for some $\alpha _i$ and
$V_i \cap U_{\alpha} = \phi $ if $\alpha \not= \alpha _i$.\\
Let $(f_{\alpha})$ be meromorphic functions which represent
$f \in H^0(X_{\mathfrak m},{\mathcal M}_{\mathfrak m}^{*}(L))$.
Then $|f_{\alpha}|^2$ is a nowhere-vanishing $C^{\infty}$ function on
$U_{\alpha} \setminus \left(\{ Q_1, \dots , Q_N\} \cap U_{\alpha}\right)$,
and satisfies
$$|f_{\alpha}|^2 = |\zeta _{\alpha \beta}|^2 |f_{\beta}|^2.$$
Modifying $f_{\alpha}$ on $V_i$ without changing the above functional equation,
we obtain positive-valued $C^{\infty}$ functions $(g_{\alpha})$ such that
$$g_{\alpha} = |\zeta _{\alpha \beta}|^2 g_{\beta}\ \ 
\text{on\ \ $U_{\alpha} \cap U_{\beta}$},$$
$$g_{\alpha} = |f_{\alpha}|^2 \ \ 
\text{on\ \ $U_{\alpha} \setminus \left( \left( \cup _{i=1}^{N} V_i \right)
\cap U_{\alpha} \right) $ }.$$
By Lemma 5.3 we have
\begin{equation*}
\begin{split}
c(L) & = \frac{1}{2\pi \sqrt{-1}} \iint _X \partial \overline{\partial}
\log (g_{\alpha}\circ \rho )^{-1}\\
 & = \frac{1}{2\pi \sqrt{-1}} \iint _X  \overline{\partial} \partial
\log (g_{\alpha} \circ \rho ).
\end{split}
\end{equation*}
Since
$$\overline{\partial }\partial \log (g_{\alpha }\circ \rho )
= \overline{\partial }\partial \log |f_{\alpha} \circ \rho |^2 = 0$$
on $\rho ^{-1}\left(U_{\alpha } \setminus \left(\left(
\cup _{i = 1}^{N}V_i\right) \cap U_{\alpha }\right) \right)$,
we have
\begin{equation*}
\begin{split}
c(L) & = \sum _{i=1}^{N}\frac{1}{2\pi \sqrt{-1}}\iint _{\rho ^{-1}(V_i)}
\overline{\partial }\partial \log (g_{\alpha } \circ \rho )\\
& = \sum_{i=1}^{N}\frac{1}{2\pi \sqrt{-1}}
\int _{\partial (\rho ^{-1}(V_i))}
\partial \log (g_{\alpha } \circ \rho ).
\end{split}
\end{equation*}
Noting 
$g_{\alpha } \circ  \rho = |f_{\alpha } \circ \rho |^2$ on
$\partial \left(\rho ^{-1}(V_i)\right)$, we obtain
\begin{equation*}
\begin{split}
c(L) & = 
\sum _{i=1}^{N} \frac{1}{2\pi \sqrt{-1}}
\int _{\partial (\rho ^{-1}(V_i))}d \log (f_{\alpha } \circ \rho )\\
& =  \sum _{i=1}^{N}  \sum _{P \in \rho ^{-1}(Q_i)}{\rm ord}_P(f_{\alpha } \circ \rho )\\
& = \sum _{i=1}^{N} {\rm ord}_{Q_i}(f).
\end{split}
\end{equation*}
\end{proof}

\par
An element in $H^0(X_{\mathfrak m},{\mathcal D}_{\mathfrak m})$ is
identified with a divisor
$$D = \sum_{Q \in X_{\mathfrak m}}D(Q)Q,$$
where
$D(Q) = \sum _{P \in \rho ^{-1}(Q)}n_P,\ n_P \in {\mathbb Z}$ with
$|n_P| \geq {\mathfrak m}(P)$ and $n_P n_{P'} >0$ for any
$P, P' \in \rho ^{-1}(Q)$ if $Q \in \overline{S}$ and $D(Q) \not= 0$,
$D(Q) \in {\mathbb Z}$ if $Q \in X_{{\mathfrak m}} \setminus \overline{S}$,
and the number of points $Q$ with $D(Q)\not= 0$ is
finite. We denote by $\widetilde{{\rm Div}}_{\mathfrak m}(X_{\mathfrak m})$
the set of all such divisors. Note that the divisor group ${\rm Div}(X_{\mathfrak m})$
is considered as a subgroup of $\widetilde{{\rm Div}}_{\mathfrak m}(X_{\mathfrak m})$.
Any $f \in {\rm Mer}^{*}(X_{\mathfrak m}) := {\rm Mer}(X_{\mathfrak m}) \setminus
\{ 0 \}$ gives a divisor
$$(f) := \sum _{Q \in X_{\mathfrak m}}{\rm ord}_Q(f)Q$$
in $\widetilde{{\rm Div}}_{\mathfrak m}(X_{\mathfrak m})$.
\begin{remark}
The set
$$\{ f \in {\rm Mer}(X) ; f \equiv c(\overline{S})\ {\rm mod}\ {\mathfrak m}\; \,  
\text{for some multiconstant $c(\overline{S})$}\}$$
is identified with a subset of ${\rm Mer}(X_{\mathfrak m})$.
\end{remark}
\begin{definition}
The divisors $D_1, D_2 \in \widetilde{{\rm Div}}_{\mathfrak m}(X_{\mathfrak m})$
are said to be equivalent if there exists $f \in {\rm Mer}(X_{\mathfrak m})$
such that $D_1 - D_2 = (f)$. In this case we write 
$D_1 \sim _{\mathfrak m} D_2$.
\end{definition}
The relation \lq\lq $\sim _{\mathfrak m}$" is an equivalence relation on
$\widetilde{{\rm Div}}_{\mathfrak m}(X_{\mathfrak m})$.
We denote $[\widetilde{{\rm Div}}_{\mathfrak m}(X_{\mathfrak m})] :=
\widetilde{{\rm Div}}_{\mathfrak m}(X_{\mathfrak m})/\sim _{\mathfrak m}$.
\par
Let $P \in X$. Take a local coordinate $z$ with $z(P) = 0$.
For any $n, m \in {\mathbb Z}$ with $n\leq m$, we define a Laurent polynomial
$$r(z) = \sum _{j=n}^{m}c_j z^j.$$
A function $f \in {\rm Mer}(X)$ is said to have $r(z)$ as a Laurent tail at $P$
if ${\rm ord}_P(f-r) > m$. The following lemma is well-known 
(cf. Chapter VI, $\S $1, Lemma 1.15 in \cite{ref13}).
 
\begin{lemma}
Let $X$ be a compact Riemann surface, and let $P_1, \dots , P_N$ be a finite
number of points in $X$. Take a local coordinate $z_i$ at $P_i$ with $z_i(P_i) = 0$
for $i = 1, \dots , N$. Suppose that a Laurent polynomial $r_i(z_i)$ is
given at each point $P_i$. Then there exists $f \in {\rm Mer}(X)$
such that $f$ has $r_i(z_i)$ as a Laurent tail at $P_i$ for all $i = 1, \dots , N$.
\end{lemma}

\begin{lemma}
For any $\widetilde{D} \in \widetilde{{\rm Div}}_{\mathfrak m}(X_{\mathfrak m})$
there exists $f \in {\rm Mer}(X_{\mathfrak m})$ such that
$\widetilde{D} - (f) \in {\rm Div}(X_{\mathfrak m})$.
\end{lemma}

\begin{proof}
Let $Q \in \overline{S}$ be a point with $\widetilde{D}(Q) \not= 0$.
Put $M := \widetilde{D}(Q)$. Consider the case $M > 0$.
We set $\rho ^{-1}(Q) = \{ P_1, \dots , P_N\}$.
For any $P_i$ there exists
$n_i \in {\mathbb N}$ with $n_i \geq {\mathfrak m}(P_i)$
such that $M = \sum _{i=1}^{N}n_i$.
Let $z_i$ be a local coordinate at $P_i$ with
$z_i(P_i) = 0$. Define $r_i(z_i) := z_i^{n_i}$. For any
$P \in S \setminus \{ P_1, \dots , P_N \}$ we set a Laurent polynomial
$$r_P(z_P) := 1 + z_P^{{\mathfrak m}(P)},$$
where $z_P$ is a local coordinate at $P$ with $z_P(P) = 0$.
By Lemma 5.6 we obtain a meromorphic function $f \in {\rm Mer}(X)$ such that
$f$ has $r_i(z_i)$ as a Laurent tail at each $P_i$ and $r_P(z_P)$ as a
Laurent tail at each $P \in S \setminus \{ P_1, \dots , P_N\}$. 
Then $f$ is considered as the pull-back of some
$g \in {\rm Mer}(X_{\mathfrak m})$. It is obvious that
$\widetilde{D} - (g)$ is zero at $Q$.\par
When $M < 0$, we apply the above argument to $- \widetilde{D}$.
Since the number of $Q \in \overline{S}$ with 
$\widetilde{D}(Q) \not= 0$ is finite, 
we obtain the lemma.
\end{proof}

We set
$${\rm Div}^{0}(X_{\mathfrak m}) := \{ D \in {\rm Div}(X_{\mathfrak m}) ;
\deg  D = 0 \}$$
and
$$\widetilde{{\rm Div}}_{\mathfrak m}^{0}(X_{\mathfrak m}) := \{
\widetilde{D} \in \widetilde{{\rm Div}}_{\mathfrak m}(X_{\mathfrak m}) ; 
\deg  \widetilde{D} = 0 \}.$$
We define $\overline{{\rm Div}^{0}(X_{\mathfrak m})} := 
{\rm Div}^{0}(X_{\mathfrak m})/\sim $ and $[\widetilde{{\rm Div}}_{\mathfrak m}^{0}
(X_{\mathfrak m})] := \widetilde{{\rm Div}}_{\mathfrak m}^{0}(X_{\mathfrak m})/
\sim _{\mathfrak m}$.

\begin{proposition}
We have
$$[\widetilde{{\rm Div}}_{\mathfrak m}(X_{\mathfrak m})] \cong
\overline{{\rm Div}(X_{\mathfrak m})}$$
and
$$[\widetilde{{\rm Div}}_{\mathfrak m}^{0}(X_{\mathfrak m})] \cong
\overline{{\rm Div}^{0}(X_{\mathfrak m})}.$$
\end{proposition}

\begin{proof}
We define an epimorphism
$$\Phi : \widetilde{{\rm Div}}_{\mathfrak m}(X_{\mathfrak m}) \longrightarrow
\overline{{\rm Div}(X_{\mathfrak m})}$$
as follows. For any $\widetilde{D} \in \widetilde{{\rm Div}}_{\mathfrak m}(X_
{\mathfrak m})$ there exists $D \in {\rm Div}(X_{\mathfrak m})$ such that
$\widetilde{D} \sim _{\mathfrak m} D$ (Lemma 5.7).
We set $\Phi (\widetilde{D}) := [D]$,
where $[D]$ is the equivalence class of $D$ in
$\overline{{\rm Div}(X_{\mathfrak m})}$.
It is easy to verify that $\Phi $ is well-defined and
$${\rm Ker}\; \Phi = \{ (f) ; f \in {\rm Mer}(X_{\mathfrak m})\}.$$
Then we obtain that 
$[\widetilde{{\rm Div}}_{\mathfrak m}(X_{\mathfrak m})] \cong
\overline{{\rm Div}(X_{\mathfrak m})}$.
\par
The second isomorphism is obvious.
\end{proof}

By Proposition 5.8 we have
\begin{equation*}
\begin{split}
{\rm Pic}(X_{\mathfrak m}) & \cong
H^0(X_{\mathfrak m}, {\mathcal D}_{{\mathfrak m}})/H^0(X_{\mathfrak m}, {\mathcal M}^{*}_{\mathfrak m})\\
& \cong [\widetilde{{\rm Div}}_{\mathfrak m}(X_{\mathfrak m})]\\
& \cong \overline{{\rm Div}(X_{\mathfrak m})}.
\end{split}
\end{equation*}
It follows from Propositions 5.4 and 5.8 that
$${\rm Pic}^{0}(X_{\mathfrak m}) \cong
[\widetilde{{\rm Div}}_{\mathfrak m}^{0}(X_{\mathfrak m})] \cong
\overline{{\rm Div}^{0}(X_{\mathfrak m})}$$
(see Section 3 in \cite{ref2}). The group ${\rm Pic}^{0}(X_{\mathfrak m})$ is
called the Picard variety of $X_{\mathfrak m}$.

\subsection{The structure of Picard varieties}
We study the relation between the Picard variety of $X_{\mathfrak m}$ and that of $X$.
Results stated in this section have been already gotten in \cite{ref2}.
We restate them in detail because of the convenience of readers.
\par
By the projection $\rho : X \longrightarrow X_{\mathfrak m}$ we have the inclusion
${\mathcal O}_{\mathfrak m}^{*} \hookrightarrow \rho _{*}{\mathcal O}_X^{*}$.
Let ${\mathcal S}:= \rho _{*}{\mathcal O}_X^{*} / {\mathcal O}_{\mathfrak m}^{*}$.
It is obvious by the definition of ${\mathcal O}_{\mathfrak m}^{*}$ that
$H^0(X_{\mathfrak m}, {\mathcal S}) \cong {\mathbb C}^m \times
({\mathbb C}^{*})^n$ for some $m,n \in {\mathbb N} \cup \{ 0 \}$.
From an exact sequence
$$\{ 1 \} \longrightarrow {\mathcal O}_{\mathfrak m}^{*} \longrightarrow
\rho _{*} {\mathcal O}_X^{*} \longrightarrow {\mathcal S} \longrightarrow \{ 1 \}$$
the following exact sequence follows:
\begin{equation*}
\begin{split}
0 &\longrightarrow H^0(X_{\mathfrak m},{\mathcal O}_{\mathfrak m}^{*})
\longrightarrow H^0(X_{\mathfrak m}, \rho _{*}{\mathcal O}_X^{*})
\longrightarrow H^0(X_{\mathfrak m},{\mathcal S})\\
&\longrightarrow H^1(X_{\mathfrak m},{\mathcal O}_{\mathfrak m}^{*})
\longrightarrow H^1(X_{\mathfrak m}, \rho _{*} {\mathcal O}_X^{*})
\longrightarrow H^1(X_{\mathfrak m},{\mathcal S}) \longrightarrow \cdots .
\end{split}
\end{equation*}
Since $H^0(X_{\mathfrak m},{\mathcal O}_{\mathfrak m}^{*}) \cong
H^0(X_{\mathfrak m}, \rho _{*}{\mathcal O}_X^{*}) \cong {\mathbb C}^{*}$
and $H^1(X_{\mathfrak m},{\mathcal S}) = 0$, we have an exact sequence
$$0 \longrightarrow H^0(X_{\mathfrak m},{\mathcal S})
\longrightarrow H^1(X_{\mathfrak m},{\mathcal O}_{\mathfrak m}^{*})
\longrightarrow H^1(X_{\mathfrak m}, \rho _{*}{\mathcal O}_X^{*})
\longrightarrow 0.$$
We set ${\mathcal T}:= \rho _{*}{\mathcal O}_X/{\mathcal O}_{{\mathfrak m}}$.
Applying the above argument to ${\mathcal T}$, we obtain the following
exact sequence
$$0 \longrightarrow H^0(X_{\mathfrak m},{\mathcal T})
\longrightarrow H^1(X_{\mathfrak m},{\mathcal O}_{\mathfrak m})
\longrightarrow H^1(X_{\mathfrak m}, \rho _{*}{\mathcal O}_X)
\longrightarrow 0.$$
Let ${\mathcal Z} := \rho _{*}{\mathbb Z}_X/{\mathbb Z}$, where
${\mathbb Z}_X$ is the constant sheaf ${\mathbb Z}$ on $X$.
Then we also obtain an exact sequence
$$0 \longrightarrow H^0(X_{\mathfrak m},{\mathcal Z})
\longrightarrow H^1(X_{\mathfrak m},{\mathbb Z})
\longrightarrow H^1(X_{\mathfrak m}, \rho _{*}{\mathbb Z}_X)
\longrightarrow 0$$
in the same manner. Combining the above diagram with (5.1),
we obtain the following commutative diagram with exact rows and columns
$$\begin{CD}
@. @. 0 @. 0 @. \\
@. @. @VVV @VVV  \\
0 @>>> H^0(X_{{\mathfrak m}}, {\mathcal Z}) @>>>
H^1(X_{{\mathfrak m}},{\mathbb Z}) @>>>
H^1(X_{{\mathfrak m}}, \rho _{*}{\mathbb Z}_X) @>>>0 \\
@. @. @VVV @VVV  \\
0 @>>> H^0(X_{{\mathfrak m}}, {\mathcal T}) @>>>
H^1(X_{{\mathfrak m}},{\mathcal O}_{{\mathfrak m}}) @>>>
H^1(X_{{\mathfrak m}}, \rho _{*}{\mathcal O}_X) @>>>0 \\
@. @. @VVV @VVV  \\
0 @>>> H^0(X_{{\mathfrak m}}, {\mathcal S}) @>>>
H^1(X_{{\mathfrak m}},{\mathcal O}_{{\mathfrak m}}^{*}) @>>>
H^1(X_{{\mathfrak m}}, \rho _{*}{\mathcal O}_X^{*}) @>>>0 \\
@. @. @VVV @VVV   \\
@.@. H^2(X_{{\mathfrak m}}, {\mathbb Z}) @=
H^2(X_{{\mathfrak m}}, \rho _{*}{\mathbb Z}_X)
@= {\mathbb Z}\\
@. @. @VVV @VVV  \\
@. @. 0 @. 0 @. \\
\end{CD}
$$
Therefore we have the following exact sequence
\begin{equation}
0 \longrightarrow H^0(X_{\mathfrak m},{\mathcal S}) \longrightarrow
{\rm Pic}^0(X_{\mathfrak m}) \xrightarrow{\ \sigma \ }
{\rm Pic}^0(X) \longrightarrow 0.
\end{equation}

\begin{theorem}[Theorem 1 in \cite{ref2}]
The above exact sequence (5.3) splits. Therefore we have
$${\rm Pic}^{0}(X_{\mathfrak m}) \cong {\rm Pic}^0(X) \oplus
H^0(X_{\mathfrak m},{\mathcal S}).$$  
\end{theorem}

\begin{proof}
It suffices to show that there exists a homomorphism
$\tau : {\rm Pic}^0(X) \longrightarrow {\rm Pic}^0(X_{\mathfrak m})$
with $\sigma \circ \tau = id$. We can take a finite open covering
${\mathfrak U} = \{ U_i\}$ of $X_{\mathfrak m}$ consisting of Stein
open sets with the following properties.\\
(i) Each $U_i$ has at most one point in $\overline{S}$.\\
(ii) If $i \not= j$, then $(U_i \cap U_j) \cap \overline{S} = \phi$
and $\rho ^{-1}(U_i)\cap \rho ^{-1}(U_j)$ has the only one component.\\
Let $\widetilde{{\mathfrak U}} := \{ \rho ^{-1}(U_i)\}$.
Since any $V := U_{i_0} \cap \cdots \cap U_{i_N}$ is a pure
1-dimensional Stein space, we have
$$H^2(V, {\mathbb Z}) = 0\quad \text{and \quad
$H^1(V, {\mathcal O}_{{\mathfrak m}}) = 0$.}$$
Then $H^1(V, {\mathcal O}_{{\mathfrak m}}^{*}) = 0$.
This shows that ${\mathfrak U}$ is a Leray covering of
$X_{{\mathfrak m}}$ for ${\mathcal O}_{{\mathfrak m}}^{*}$.
Similarly $\widetilde{{\mathfrak U}}$ is a Leray covering
of $X$ for ${\mathcal O}_X^{*}$.
Then we have
$$H^1(X_{\mathfrak m},{\mathcal O}_{\mathfrak m}^{*}) \cong
H^1({\mathfrak U},{\mathcal O}_{\mathfrak m}^{*})\quad
\text{and\quad $H^1(X,{\mathcal O}_X^{*}) \cong
H^1(\widetilde{{\mathfrak U}},{\mathcal O}_X^{*})$}.$$
\par
Take any $x \in {\rm Pic}^0(X)$. We assume that $x$ is the isomorphic class
of topologically trivial holomorphic line bundle $\widetilde{L}$ on $X$.
Let $(\tilde{g}_{ij}) \in Z^1(\widetilde{{\mathfrak U}},{\mathcal O}_X^{*})$
be a 1-cocycle which gives $\widetilde{L}$. Since $U_i \cap U_j$ has no
point in $\overline{S}$ for $i,j$ with $i \not= j$, 
$\rho ^{-1}(U_i) \cap \rho ^{-1}(U_j)$ is biholomorphic onto
$U_i \cap U_j$ by $\rho $. Then there exists a nowhere-vanishing
holomorphic function $g_{ij}$ on $U_i \cap U_j$ such that 
$\tilde{g}_{ij} = g_{ij}\circ \rho $.
Let $L$ be the line bundle on $X_{\mathfrak m}$ determined by
$(g_{ij}) \in Z^1({\mathfrak U},{\mathcal O}_{\mathfrak m}^{*})$.
Then we have $\widetilde{L} = \rho ^{*}L$.
Since $c_X(\widetilde{L}) = 0$, we have $c(L) = 0$.
Let $y$ be the isomorphic class in which $L$ is contained.
We set $\tau (x) := y$. It is easy to check that
$\tau : {\rm Pic}^0(X) \longrightarrow {\rm Pic}^0(X_{\mathfrak m})$
is well-defined 
for any $\widetilde{L}$ in any class of ${\rm Pic}^0(X)$ is
the pull-back of some line bundle $L$ on $X_{{\mathfrak m}}$.
\end{proof}

We also gave another proof of the above theorem in \cite{ref2}
(see Section 4 in \cite{ref2}).

\subsection{Albanese varieties}
We defined the Albanese variety ${\rm Alb}(X)$ of a compact Riemann surface
$X$ in Section 5.1. We know that ${\rm Pic}^0(X)$,
$\overline{{\rm Div}^0(X)}$ and ${\rm Alb}(X)$ are isomorphic one another.
\par
In this section we consider Albanese varieties for singular curves.
Let $X_{\mathfrak m}$ be a singular curve of genus $\pi = g + \delta $ as
in Section 2, where $g$ is the genus of $X$. Take a basis
$\{ \omega _1, \dots , \omega _{\pi}\}$ of $H^0(X_{\mathfrak m},\Omega _{\mathfrak m})$.
We fix a canonical homology basis $\{ \alpha _1, \beta _1, \dots , \alpha _g,
\beta _g \}$ of $X$. Let $S = \{ P_1, \dots , P_s\}$.
We denote by $\gamma _j$ a small circle centered at $P_j$ with anticlockwise
direction for $j = 1, \dots ,s$. Then the set $\{ \alpha _1, \beta _1, \dots ,
\alpha _g, \beta _g, \gamma _1, \dots , \gamma _{s-1}\}$ forms a basis of
$H_1(X \setminus S, {\mathbb Z}) = H_1(X_{\mathfrak m} \setminus \overline{S},
{\mathbb Z})$. Let $H^0(X_{\mathfrak m},\Omega _{\mathfrak m})^{*}$ be the dual space of
$H^0(X_{\mathfrak m},\Omega _{\mathfrak m})$. We set
$$A := H^0(X_{\mathfrak m},\Omega _{\mathfrak m})^{*}/
H_1(X_{\mathfrak m} \setminus \overline{S},{\mathbb Z}).$$
In \cite{ref15} it is stated that
$A$ has the unique algebraic structure in which the following exact
sequence is algebraic:
$$0 \longrightarrow H \longrightarrow A \longrightarrow {\rm Alb}(X)
\longrightarrow 0,$$
where $H$ is an affine algebraic group.
Rosenlicht \cite{ref15} called it the generalized Jacobi variety of
$X_{\mathfrak m}$. However, this is one of structures which can be induced to $A$.
In \cite{ref2}, we called $A$ with the above algebraic structure the algebraic
Albanese variety of $X_{\mathfrak m}$, and wrote it ${\rm Alb}^{alg}(X_{\mathfrak m})$.
\par
We are interested in their analytic structure. We begin by considering $A$ explicitly.
We may assume that $\{ \rho ^{*}\omega _1, \dots , \rho ^{*}\omega _g\}$ is a basis
of $H^0(X,\Omega )$. Consider $2g + s -1$ vectors
$$\left(\int _{\alpha _i}\rho ^{*}\omega _1, \dots
, \int _{\alpha _i}\rho ^{*}\omega _{\pi }\right),\quad i = 1, \dots , g,$$
$$\left(\int _{\beta _i}\rho ^{*}\omega _1, \dots
, \int _{\beta _i}\rho ^{*}\omega _{\pi }\right),\quad i = 1, \dots , g$$
and
$$\left(\int _{\gamma _j}\rho ^{*}\omega _1, \dots
, \int _{\gamma _j}\rho ^{*}\omega _{\pi }\right),\quad j = 1, \dots , s-1$$
in ${\mathbb C}^{\pi}$. Let $\Gamma $ be a subgroup of ${\mathbb C}^{\pi}$
generated by these vectors over ${\mathbb Z}$.
We have the following lemma as in the non-singular case.

\begin{lemma}
There exist $a_1, \dots , a_{\pi} \in X \setminus S$ such that if
$$\rho ^{*}\omega (a_1) = \cdots = \rho ^{*}\omega (a_{\pi }) = 0,$$
then $\omega = 0$ for $\omega \in H^0(X_{\mathfrak m},\Omega _{\mathfrak m})$.
\end{lemma}

\begin{proof}
For any $a \in X \setminus S$ we set
$$H_a := \{ \omega \in H^0(X_{{\mathfrak m}}, \Omega _{{\mathfrak m}})
; \rho ^{*}\omega (a) = 0 \}.$$
Then $H_a$ is a linear subspace of
$H^0(X_{{\mathfrak m}}, \Omega _{{\mathfrak m}})$ with
$\dim H_a = \pi - 1$. Since
$\cap _{a \in X \setminus S}H_a = \{ 0 \}$, we can take
$a _1 , \dots , a_{\pi } \in X \setminus S$ such that
$\cap _{i = 1}^{\pi } H_{a_i} = \{ 0 \}$.
\end{proof}

\begin{proposition}
$\Gamma $ is a discrete subgroup of ${\mathbb C}^{\pi }$.
\end{proposition}

\begin{proof}
We take points $a_1, \dots , a_{\pi } \in X \setminus S$ possessing the 
property in Lemma 5.10. Let $(U_j,z_j)$ be a coordinate neighbourhood of
$a_j$ with $z_j(a_j) = 0$. We may assume that $\{ U_j\}$ is pairwise disjoint
and $U_j$ is simply connected. We have a representation of $\rho ^{*}\omega _i$
on $U_j$ such as
$$\rho ^{*} \omega _i = \varphi _{ij}dz_j,\quad 
\varphi _{ij} \in H^0(U_j, {\mathcal O}_{X}).$$
Let $M:= \left(\varphi _{ij}(a_j)\right) _{1 \leq i,j \leq \pi }$.
Then $M$ 
is of ${\rm rank}\ \pi $. In fact, if ${\rm rank}\ M < \pi $, then there exist
$c_1, \dots , c_{\pi } \in {\mathbb C}$ with
$(c_1, \dots , c_{\pi }) \not= (0, \dots , 0)$ such that
$$\sum _{i=1}^{\pi } c_i \varphi _{ij}(a_j) = 0$$
for any $j = 1, \dots , \pi $.
We set $\omega _0 := \sum _{i=1}^{\pi }c_i\omega _i$.
Then we have
$$\rho ^{*}\omega _0 (a_1) = \cdots = \rho ^{*}\omega _0 (a_{\pi }) = 0.$$
Therefore $\omega _0 = 0$ by Lemma 5.10.
This contradicts that $\omega _1, \dots , \omega _{\pi }$ is a basis of
$H^0(X_{{\mathfrak m}}, \Omega _{{\mathfrak m}})$.\par
For any 
$x = (x_1, \dots , x_{\pi }) \in U_1 \times \cdots \times U_{\pi }$ we define
$$F_i(x) := \sum _{j=1}^{\pi } \int _{a_j}^{x_j} \rho ^{*}\omega _i,\quad
i = 1, \dots , \pi .$$
Then we have a holomorphic map
$$F : U_1 \times \cdots \times U_{\pi } \longrightarrow {\mathbb C}^{\pi },
\quad F(x) := (F_1(x), \dots , F_{\pi }(x)).$$
The jacobian matrix of $F$ is
$$J_F(x) := \left( \frac{\partial F_i}{\partial z_j}(x)\right) =
\left(\varphi _{ij}(x)\right).$$
Let $a = (a_1, \dots , a_{\pi }) \in U_1 \times \cdots \times U_{\pi }$.
Since $M = J_F(a)$ is invertible, we may assume that the map
$$F : U_1 \times \cdots \times U_{\pi } \longrightarrow
W := F(U_1 \times \cdots \times U_{\pi }) \subset {\mathbb C}^{\pi }$$
is biholomorphic, shrinking $U_1, \dots , U_{\pi }$ if necessarily.
\par
It suffices to show that $\Gamma \cap W = \{ 0 \}$.
Suppose that there exists $t \in \Gamma \cap (W \setminus \{ 0\})$.
Then there exists $x \in U_1 \times \cdots \times U_{\pi }$ with
$x \not= a$ such that $F(x) = t \in \Gamma $. Renumbering if necessarily, we
may assume that $x_j \not= a_j$ for $1 \leq j \leq k$ and $x_j = a_j$ for
$k < j \leq \pi $, where $1 \leq k \leq \pi $.
We define a divisor $D$ prime to $S$ by
$$ D := \sum _{i=1}^{k}(x_i - a_i).$$
Since $F(x) = t \in \Gamma $, there exists
$$\gamma = \sum n_i \alpha _i + \sum m_i \beta _i + \sum \ell _j \gamma _j,
\quad n_i, m_i, \ell _j \in {\mathbb Z}$$
such that
$$t = \left(
\int _{\gamma }\rho ^{*}\omega _1, \dots , \int _{\gamma }\rho ^{*}\omega _{\pi}
\right).$$
Let $c_i$ be a curve from $a_i$ to $x_i$ in $U_i$. We set
$$c := \sum _{i=1}^{k}c_i - \gamma .$$
Then we have $\partial c = D$. For any $j = 1, \dots , \pi $ we have
$$\int _c \rho ^{*}\omega _j = \sum _{i=1}^{k}\int _{c_i}\rho ^{*}\omega _j
- \int _{\gamma }\rho ^{*}\omega _j =
F_j(x) - \int _{\gamma }\rho ^{*}\omega _j = t_j - t_j = 0.$$
Hence we obtain
$$\int _c \rho ^{*}\omega = 0$$
for any $\omega \in H^0(X_{\mathfrak m},\Omega _{\mathfrak m})$.
By Theorem 4.2 (a generalized Abel's theorem) there exists
$f \in {\rm Mer}(X)$ with $f \equiv c(\overline{S})\ {\rm mod}\ {\mathfrak m}$
for some multiconstant $c(\overline{S})$ such that
$$(f) = D = \sum _{i=1}^{k}(x_i - a_i).$$
Let $d_i/z_i$ be the principal part of $f$ at $a_i$. Then
$d_i \not= 0$ for $1 \leq i \leq k$. By the residue theorem we obtain
\begin{equation*}
\begin{split}
0 &= \sum _{P \in X} {\rm Res}_P(f \rho ^{*}\omega _i)\\
&= \sum _{P \in X \setminus S}{\rm Res}_P(f \rho ^{*}\omega _i)
+ \sum _{Q \in \overline{S}} c_Q \sum _{P \in \rho ^{-1}(Q)}
{\rm Res}_P(\rho ^{*}\omega _i)\\
&= \sum _{j=1}^{k}d_j\varphi _{ij}(a_j)
\end{split}
\end{equation*}
for any $i = 1, \dots , \pi $. This contradicts that
$M = (\varphi _{ij}(a_j))$ is invertible.
\end{proof}

By the above proposition we have
$$A = H^{0}(X_{\mathfrak m},\Omega _{\mathfrak m})^{*}/
H_1(X \setminus S, {\mathbb Z}) \cong {\mathbb C}^{\pi}/\Gamma $$
as a complex Lie group. We call $A$ with the structure as a complex
Lie group the analytic Albanese variety of $X_{\mathfrak m}$, and
write it as ${\rm Alb}^{an}(X_{\mathfrak m})$. We shall study its
properties in detail in the following sections.

\subsection{Toroidal groups and quasi-abelian varieties}
A connected complex Lie group $G$ is said to be a toroidal group if
$H^0(G,{\mathcal O}_G) = {\mathbb C}$, where ${\mathcal O}_G$ is the
structure sheaf on $G$. It is well-known that a toroidal
group $G$ is commutative. Then we have $G \cong {\mathbb C}^n/\Lambda $
as a complex Lie group, where $n = \dim G$ and 
$\Lambda $ is a discrete subgroup of
${\mathbb C}^n$ with ${\rm rank}\ \Lambda = n+m$ $(1 \leq m \leq n)$.
Let $\lambda _1 = (\lambda _{11}, \dots , \lambda _{1n}), \dots ,
\lambda _{n+m} = ( \lambda _{n+m, 1}, \dots , \lambda _{n+m, n}) \in
{\mathbb C}^n$ be generators of $\Lambda .$
Then the matrix
$$P =
\begin{pmatrix}
\lambda _{11} & \dots & \lambda _{n+m, 1}\\
\lambda _{12} & \dots & \lambda _{n+m, 2}\\
\hdotsfor[2]{3}\\
\lambda _{1n} & \dots & \lambda _{n+m, n}\\
\end{pmatrix}$$
is called a period matrix of $G$. By a suitable change of
variables and generators we have the following
normal form of $P$
\begin{equation}
P = \left(
\begin{array}{ccc}
0 & I_m & T \\
I_{n-m} & R_1 & R_2
\end{array}
\right),
\end{equation}
where $I_k$ is the unit matrix of degree $k$, the matrix
$(I_m\ \  T)$ is a period matrix of an $m$-dimensional complex torus
and $(R_1\ \  R_2)$ is a real matrix.
We call the coordinates in (5.4) toroidal coordinates of $G$ and write
them as follows
$$z = (z', z'') = (z_1, \dots z_m; z_{m+1}, \dots z_n).$$
The projection $z \mapsto z'$ in these coordinates makes
$G$ a fibre bundle $\sigma : G \longrightarrow {\mathbb T}$ over
an $m$-dimensional complex torus ${\mathbb T}$ with fibres
$({\mathbb C}^{*})^{n-m}$.
\par
Let ${\mathbb R}_{\Lambda }^{n+m}$ be the real linear subspace of
${\mathbb C}^n$ spanned by $\Lambda $. We denote by 
${\mathbb C}_{\Lambda }^{m}$ the maximal complex linear subspace contained
in ${\mathbb R}_{\Lambda }^{n+m}$.
It is well-known that $\dim {\mathbb C}_{\Lambda }^{m} = m$.

\begin{definition}
A toroidal group $G = {\mathbb C}^n/\Lambda $ is a quasi-abelian variety
if there exists a hermitian form ${\mathcal H}$ on ${\mathbb C}^n$ such that\\
(1) ${\mathcal H}$
is positive definite on ${\mathbb C}^{m}_{\Lambda }$,\\
(2) the imaginary part ${\mathcal A} := {\rm Im}\ {\mathcal H}$ of ${\mathcal H}$ is
${\mathbb Z}$-valued on $\Lambda \times \Lambda $.\\
A hermitian form ${\mathcal H}$ satisfying the above conditions (1) and (2)
is said to be an ample Riemann form for $G$. We set
${\mathcal A}_{\Lambda } := {\mathcal A}|_{{\mathbb R}^{n+m}_{\Lambda } \times 
{\mathbb R}^{n+m}_{\Lambda }}$ for an ample Riemann form ${\mathcal H}$.
Since ${\mathcal A}_{\Lambda }$ is an alternating form, we have
$${\rm rank}\ {\mathcal A}_{\Lambda } = 2(m + k),\quad
0 \leq 2k \leq n-m.$$
In this case we say that an ample Riemann form ${\mathcal H}$ is of kind $k$.
\end{definition}

If a quasi-abelian variety $G$ has an ample Riemann form of kind $k$,
then it also has an ample Riemann form of kind $k'$ for any $k'$
with $2k \leq 2k' \leq n-m$ (\cite{ref6}).
Then we defined the kind of a quasi-abelian variety as follows in \cite{ref6}.

\begin{definition}
The kind of a quasi-abelian variety $G$ is the smallest integer $k$ with
$0 \leq 2k \leq n-m$ such that there exists an ample Riemann form of kind $k$ for $G$.
\end{definition}

If $G = {\mathbb C}^n /\Lambda $ is a quasi-abelian variety of kind $0$,
then the matrix $(I_m\ \ T)$ in (5.4) is a period matrix of an $m$-dimensional
abelian variety ${\mathbb A}$. Therefore $G$ has the structure of a fibre
bundle $\sigma : G \longrightarrow {\mathbb A}$ over ${\mathbb A}$ with
fibres $({\mathbb C}^{*})^{n-m}$.
Replacing fibres $({\mathbb C}^{*})^{n-m}$ with 
$({\mathbb P}^{1})^{n-m}$, we obtain the associated bundle
$\overline{\sigma } : \overline{G} \longrightarrow {\mathbb A}$ over
${\mathbb A}$ with fibres $({\mathbb P}^{1})^{n-m}$.
We say that $\overline{G}$ is the standard compactification of a
quasi-abelian variety $G$ of kind $0$.\par
Conversely, if the matrix $(I_m\ \ T)$ in (5.4) is a period matrix of
an $m$-dimensional abelian variety, then $G$ is a quasi-abelian variety of
kind $0$.\par
We refer to \cite{ref4, ref5} for further properties of toroidal groups.

\subsection{Canonical form of analytic Albanese varieties}
By Proposition 5.11 we have
$$A = H^0(X_{\mathfrak m},\Omega _{\mathfrak m})^{*}/
H_1(X \setminus S, {\mathbb Z}) \cong {\mathbb C}^{\pi }/\Gamma $$
as a complex Lie group. The theorem of Remmert-Morimoto says that
$$A \cong {\mathbb C}^{\pi }/\Gamma \cong
{\mathbb C}^p \times ({\mathbb C}^{*})^q \times G,$$
where $G$ is a toroidal group of dimension $r$ and $p + q + r = \pi $.

\begin{proposition}
$G$ is a quasi-abelian variety of kind $0$.
\end{proposition}

\begin{proof}
We use the same notations as in Section 5.4. Let
$\{ \omega _1, \dots , \omega _{\pi}\}$ be a basis of
$H^{0}(X_{\mathfrak m}, \Omega _{\mathfrak m})$ such that
$\{\rho ^{*}\omega _1, \dots , \rho ^{*}\omega _{g}\}$ is
a basis of $H^{0}(X, \Omega )$.
We may assume that
$$\left(
\begin{array}{cc}
\left( \int _{\alpha _i}\rho ^{*}\omega _j\right)&
\left( \int _{\beta _i}\rho ^{*}\omega _j\right)\\
\end{array}\right)_{1\leq i, j \leq g} =
\left(
\begin{array}{cc}
I_g & \tau \\
\end{array}\right)$$
is a period matrix of ${\rm Alb}(X)$, where $\tau $ is a matrix
in the Siegel upper half space ${\mathfrak S}_g$ of degree $g$.
By a well-known fact concerning meromorphic differentials on a
Riemann surface (for example, see Theorem II.5.1 in \cite{ref9}),
we can take $\omega _{g+1}, \dots , \omega _{\pi}$ such as
\begin{equation}
\int _{\gamma _i}\rho ^{*}\omega _{g+j} = 0,\ 1\ \text{or $-1$}.
\end{equation}
Moreover, we replace $\omega _{g+1}, \dots , \omega _{\pi }$ with
$\omega '_{g+1}, \dots , \omega '_{\pi }$ given by
$$
\begin{pmatrix}
\omega '_{g+1}\\
\vdots \\
\omega '_{\pi }
\end{pmatrix}
:= - 
\left(
\int _{\alpha _i} \rho ^{*}\omega _{g+j}
\right)
\begin{pmatrix}
\omega _1 \\
\vdots \\
\omega _g
\end{pmatrix}
+
\begin{pmatrix}
\omega _{g+1}\\
\vdots \\
\omega _{\pi }
\end{pmatrix}
.$$
Then a period matrix $P$ of $A$ is as follows
$$P = \left(
\begin{array}{ccc}
I_g & 0 & \tau \\
0 & \left( \int _{\gamma _i}\rho ^{*}\omega _{g+j}\right) &
\left( \int _{\beta _i}\rho ^{*}\omega _{g+j}\right)\\
\end{array}\right).$$
Here we write new forms $\omega '_{g+1} , \dots , \omega '_{\pi }$ 
as
$\omega _{g+1} , \dots , \omega _{\pi }$.
We note that the above change of $\omega _{g+1}, \dots , \omega _{\pi }$
preserves the property (5.5).
We set
$$M = \left(
\begin{array}{cc}
I_g & 0 \\
R_1 & I_{\pi - g}\\
\end{array}\right),$$
where $R_1 = - \left({\rm Im}\left(\int _{\beta _i}
\rho ^{*}\omega _{g+j}\right)\right)\left({\rm Im}\tau \right)^{-1}.$
Then we have
$$\widetilde{P} := MP =
\left(
\begin{array}{ccc}
I_g & 0 & \tau \\
R_1 & E &  R_2 \\
\end{array}\right),$$
where $R_1$ and $R_2$ are real matrices and entries of $E$ are
0, 1 or $-1$.
Since $(I_g\ \ \tau )$ is a period matrix of a $g$-dimensional abelian
variety, $G$ is a quasi-abelian variety of kind $0$.
\end{proof}

By the above proposition, $A$ has the following representation
as a complex Lie group
$$A \cong {\mathbb C}^p \times ({\mathbb C}^{*})^q
\times {\mathfrak Q},$$
where ${\mathfrak Q}$ is a quasi-abelian variety of kind $0$.
We call this representation the canonical form of the analytic Albanese variety
of $X_{\mathfrak m}$ and write
$${\rm Alb}^{an}(X_{\mathfrak m}) = {\mathbb C}^p \times ({\mathbb C}^{*})^q
\times {\mathfrak Q}.$$

\subsection{Analytic Albanese varieties}
Let $\omega _1, \dots , \omega _{\pi}$ be a basis of
$H^0(X_{\mathfrak m}, \Omega _{\mathfrak m})$.
We define a period map $\varphi $ with base point $P_0 \in X \setminus S$ by
$$\varphi : X \setminus S \longrightarrow
{\rm Alb}^{an}(X_{\mathfrak m}),\quad
P \longmapsto \left[ \left( \int _{P_0}^{P}\rho ^{*} \omega _1, \dots ,
\int _{P_0}^{P}\rho ^{*} \omega _{\pi}\right)\right].$$
Consider a commutative complex Lie group $G$.
Any holomorphic map $\psi : X \setminus S \longrightarrow G$ can be extended
to a homomorphism on ${\rm Div}(X_{\mathfrak m})$ by
$$\psi (D) := \sum _{P \in X \setminus S} D(P)\psi (P)$$
for $D = \sum _{P \in X \setminus S}D(P)P \in {\rm Div}(X_{\mathfrak m})$.
If $g$ is a meromorphic function on $X$ with
$g \equiv c(\overline{S})\ {\rm mod}\ {\mathfrak m}$ for some multiconstant
$c(\overline{S})$, then
$$\psi ((g)) := \sum _{P \in X \setminus S}{\rm ord}_P(g)\psi (P)$$
is well-defined.

\begin{definition}
We say that a holomorphic map $\psi : X \setminus S \longrightarrow G$
admits ${\mathfrak m}$ for a modulus if $\psi ((f)) = 0$ for any
meromorphic function $f$ on $X$ satisfying
$f \equiv c(\overline{S})\ {\rm mod}\ {\mathfrak m}$ with some multiconstant
$c(\overline{S})$.
\end{definition}

\begin{proposition}
The period map $\varphi : X \setminus S \longrightarrow {\rm Alb}^{an}(X_{\mathfrak m})$
defined above admits ${\mathfrak m}$ for a modulus. Furthermore, it is
a holomorphic embedding if $g \geq 1$.
\end{proposition}

\begin{proof}
Let $f$ be a meromorphic function on $X$ with
$f \equiv c(\overline{S})\ {\rm mod}\ {\mathfrak m}$ for some
multiconstant $c(\overline{S})$. By Theorem 4.2 (a generalized Abel's
theorem) there exists a 1-chain $c \in C_1(X \setminus S)$ with
$\partial c = (f)$ such that 
$$\int _c \rho ^{*}\omega = 0$$
for any $\omega \in H^0(X_{\mathfrak m}, \Omega _{\mathfrak m})$.
Then we have $\varphi ((f)) = 0$.
\par
Since $\rho ^{*}\omega _1, \dots , \rho ^{*}\omega _{\pi}$ have no
common zero, $d\varphi (P) \not= 0$ at any $P \in X \setminus S$.
Next we show that $\varphi : X \setminus S \longrightarrow
{\rm Alb}^{an}(X_{\mathfrak m})$ is injective.
Assume that there exist two distinct points $P$ and $P'$ in $X \setminus S$
such that $\varphi (P) = \varphi (P')$. Then we have
$\varphi ((P - P')) = 0$. Applying Theorem 4.2 again, we obtain a meromorphic
function $f$ on $X$ with $f \equiv c(\overline{S})\ {\rm mod}\ {\mathfrak m}$
for some multiconstant $c(\overline{S})$ such that $(f) = P - P'$.
Since the genus of $X$ is $g \geq 1$, there does not exist such a function.
Therefore $\varphi $ is injective.
\end{proof}

For a complex manifold $M$, we denote by $M^{(r)}$ its symmetric product of
degree $r$. We can extend $\varphi : X \setminus S \longrightarrow
{\rm Alb}^{an}(X_{\mathfrak m})$ to a holomorphic map
$\varphi : (X \setminus S)^{(r)} \longrightarrow {\rm Alb}^{an}(X_{\mathfrak m})$.

\begin{theorem}
The map $\varphi : (X \setminus S)^{(\pi )} \longrightarrow
{\rm Alb}^{an}(X_{\mathfrak m})$ is surjective.
\end{theorem}

\begin{proof}
Let $a_1, \dots , a_{\pi}$ be points in Lemma 5.10. We set
$$D_0 := a_1 + \cdots + a _{\pi } \in {\rm Div}(X_{\mathfrak m}).$$
We take a coordinate neighbourhood $(U_j, t_j)$ of $a_j$ with
$t_j(a_j) = 0$. We may assume that $U_j$ is simply connected and
$\{ U_j\}$ is pairwise disjoint. Let
$$K := \varphi (D_0) = \sum _{j=1}^{\pi }\varphi (a_j).$$
Each $\rho ^{*}\omega _k$ has a representation
$$\rho ^{*}\omega _k = \eta _{kj}(t_j)dt_j$$
with respect to $(U_j, t_j)$ for all $j = 1, \dots ,\pi $.
As shown in the proof of Proposition 5.11, we may assume that
$$\varphi : U_1 \times \cdots \times U_{\pi } \longrightarrow
W := \varphi (U_1 \times \cdots \times U_{\pi })$$
is biholomorphic and the image of $z = (z_1, \dots ,z_{\pi }) \in
U_1 \times \cdots \times U_{\pi }$ by $\varphi $ is
$K + (\varphi _1(z), \dots , \varphi _{\pi }(z))$, where
$$\varphi _{\ell}(z) = \sum _{j=1}^{\pi }\int _{0}^{z_j}
\eta _{\ell j}(t_j)dt_j\quad \mbox{for\quad $\ell = 1, \dots , \pi $}.$$
For any $c =(c_1, \dots , c_{\pi }) \in {\mathbb C}^{\pi }$ there exists
$N \in {\mathbb N}$ such that $K + c/N \in W$. Then there exists
$(P_1, \dots , P_{\pi }) \in U_1 \times \cdots \times U_{\pi }$ such that
$$\varphi (P_1 + \cdots + P_{\pi }) = K + \frac{c}{N}$$
or
$$N(\varphi (P_1 + \cdots + P_{\pi }) - K) = c.$$
It suffices to show that there exists a positive divisor
$D \in {\rm Div}(X_{\mathfrak m})$ with $\deg D = \pi $ such that
$$\varphi (D) = N(\varphi (P_1 + \cdots + P_{\pi }) - K).$$
\par
Put
$$\widetilde{D} := N \left( \sum _{j=1}^{\pi } a_j -
\sum _{j =1}^{\pi }P_j\right) - \pi P_0.$$
Then $\deg \widetilde{D} = - \pi $.
Let $\overline{S} = \{ Q_1, \dots , Q_s\}$.
For any $i = 1, \dots , s$ we define
$$
{\mathfrak m}_i(P) :=
\left\{
\begin{array}{ll}
{\mathfrak m}(P)&
\mbox{if $P \in \rho ^{-1}(Q_i)$}\\
0&
\mbox{otherwise}.
\end{array}\right.$$
Then ${\mathfrak m}_i$ is a modulus with support $\rho ^{-1}(Q_i)$.
We set
$${\mathfrak n}_i := {\mathfrak m}_1 + \cdots + {\mathfrak m}_i.$$
We denote by $X_{{\mathfrak n}_i}$ the singular curve constructed from
$X$ by ${\mathfrak n}_i$.
If $\pi _i$ is the genus of $X_{{\mathfrak n}_i}$, then we have
$$\pi _i = g + \deg {\mathfrak n}_i - i =
g + \sum _{\alpha = 1}^{i}\deg {\mathfrak m}_{\alpha } - i.$$
We note that $\widetilde{D} + \sum _{\alpha = i + 1}^{s}{\mathfrak m}_{\alpha }
\in {\rm Div}(X_{{\mathfrak n}_i})$ for any $i$.
Applying Theorem 3.13 on $X_{{\mathfrak n}_i}$, we obtain
\begin{eqnarray*}
\lefteqn{\dim H^0\left(X_{{\mathfrak n}_i},{\mathcal L}_{{\mathfrak n}_i}\left(-\widetilde{D} - 
\sum _{\alpha = i+1}^{s}{\mathfrak m}_{\alpha }\right)\right)}\\
&& = \dim H^0\left(X_{{\mathfrak n}_i},\Omega _{{\mathfrak n}_i}\left(\widetilde{D} + 
\sum _{\alpha = i+1}^{s}{\mathfrak m}_{\alpha }\right)\right)  +
\deg \left(- \widetilde{D} - \sum _{\alpha = i+1}^{s}{\mathfrak m}_{\alpha }\right) + 1 - \pi _i\\
&& =\dim H^0\left(X_{{\mathfrak n}_i}\;,\Omega _{{\mathfrak n}_i}\left(\widetilde{D}
+ \sum _{\alpha = i+1}^{s}{\mathfrak m}_{\alpha }\right)\right) + i -s +1.
\end{eqnarray*}
Let $\rho _i : X_{{\mathfrak n}_i} \longrightarrow X_{{\mathfrak n}_{i+1}}$
be the projection. It is obvious that
\begin{eqnarray*}
\lefteqn{
H^0\left(X_{{\mathfrak n}_i}, \Omega _{{\mathfrak n}_i}
\left(\widetilde{D} + \sum _{\alpha = i+1}^{s}{\mathfrak m}_{\alpha }\right) \right)}\\
&&= \left\{ \rho ^{*}_i\omega ; \omega \in
H^0\left(X_{{\mathfrak n}_{i+1}}, \Omega _{{\mathfrak n}_{i+1}}\left(
\widetilde{D} + \sum _{\alpha = i +2}^{s}{\mathfrak m}_{\alpha }\right) \right)\right\}.
\end{eqnarray*}
Then we obtain by the above equality that
\begin{eqnarray*}
\lefteqn{ \dim H^0\left( X_{{\mathfrak n}_i}, {\mathcal L}_{{\mathfrak n}_i}
\left( - \widetilde{D} - \sum _{\alpha = i+1}^{s}{\mathfrak m}_{\alpha }
\right)\right)}\\
&& = \dim H^0\left( X_{{\mathfrak n}_{i+1}}, \Omega _{{\mathfrak n}_{i+1}}
\left(\widetilde{D} + \sum _{\alpha = i +2}^{s}{\mathfrak m}_{\alpha }\right)\right)
+ i - s + 1.
\end{eqnarray*}
Applying again Theorem 3.13 on $X_{{\mathfrak n}_{i+1}}$, we have
\begin{eqnarray*}
\lefteqn{
\dim H^0\left( X_{{\mathfrak n}_{i+1}}, {\mathcal L}_{{\mathfrak n}_{i+1}}\left(
- \widetilde{D} - \sum _{\alpha = i+2}^{s}{\mathfrak m}_{\alpha }\right)\right)
- 1}\\
&& =
\dim H^0\left( X_{{\mathfrak n}_{i+1}}, \Omega _{{\mathfrak n}_{i+1}}\left(
\widetilde{D} + \sum _{\alpha = i+2}^{s}{\mathfrak m}_{\alpha }\right)\right)
+ i - s +1.
\end{eqnarray*}
Therefore we obtain
\begin{eqnarray*}
\lefteqn{
\dim H^0\left( X_{{\mathfrak n}_{i}}, {\mathcal L}_{{\mathfrak n}_{i}}\left(
- \widetilde{D} - \sum _{\alpha = i+1}^{s}{\mathfrak m}_{\alpha }\right)\right)
}\\
&& =
\dim H^0\left( X_{{\mathfrak n}_{i+1}}, {\mathcal L} _{{\mathfrak n}_{i+1}}\left(
\widetilde{D} + \sum _{\alpha = i+2}^{s}{\mathfrak m}_{\alpha }\right)\right)
-1.
\end{eqnarray*}
Considering ${\mathfrak n}_0= 0$, we see that the above equality holds for
$i=0$. From these equalities it follows that
$$\dim H^0(X, {\mathcal L}( - \widetilde{D} - {\mathfrak m}) )=
\dim H^0(X_{\mathfrak m},{\mathcal L}_{{\mathfrak m}}(-
\widetilde{D})) - s.$$
Let
$$H_i := \{ f \in H^0(X_{{\mathfrak m}},{\mathcal L}_{{\mathfrak m}}(
- \widetilde{D}) ) ; f(Q_i) = 0\}$$
for $i = 1, \dots ,s$. Then we have
$$H^0(X, {\mathcal L}(- \widetilde{D} - {\mathfrak m})) =
\bigcap _{i=1}^{s}\{ \rho ^{*} f ; f \in H_i \}.$$
From the above equality on the dimension it follows that
$$H^0(X_{{\mathfrak m}}, {\mathcal L}_{{\mathfrak m}}(
- \widetilde{D})) \setminus \left(\bigcup _{i=1}^{s} H_i\right) \not= \phi .$$
Then, there exists $f \in H^0(X_{\mathfrak m},{\mathcal L}_{\mathfrak m}(
- \widetilde{D}))$ with $(\rho ^{*}f) \in {\rm Div}(X_{\mathfrak m})$.
If we set
$$D := (\rho ^{*}f) - \widetilde{D},$$
then we have $D \geq 0$ and $D \in {\rm Div}(X_{\mathfrak m})$ for
$$(\rho ^{*}f) \geq - (- \widetilde{D}) = \widetilde{D}.$$
Since $D + \widetilde{D} = (\rho ^{*}f)$, we obtain
$\varphi (D + \widetilde{D}) = 0$ by Theorem 4.2.
Hence we have
\begin{align*}
0 & = \varphi (D) + \varphi \left(
N \left( \sum _{j=1}^{\pi }a_j - \sum _{j=1}^{\pi }P_j\right) - \pi P_0 \right)\\
& = \varphi (D) + N \left( K - \varphi \left(\sum_{j=1}^{\pi }P_j\right)\right)
\end{align*}
or
$$\varphi (D) = N \left( \varphi \left(
\sum _{j-1}^{\pi }P_j\right) - K \right).$$
This completes the proof.
\end{proof}

\begin{remark}
It follows from the above theorem that $\varphi (X \setminus S)$ generates
${\rm Alb}^{an}(X_{\mathfrak m})$.
\end{remark}

\begin{corollary}
We have an isomorphism $\overline{{\rm Div}^{0}(X_{\mathfrak m})}
\cong {\rm Alb}^{an}(X_{\mathfrak m})$ as groups.
\end{corollary}

\begin{proof}
We have a homomorphism $\varphi : {\rm Div}^0(X_{\mathfrak m}) \longrightarrow
{\rm Alb}^{an}(X_{\mathfrak m})$. By Theorem 5.17, for any
$a \in {\rm Alb}^{an}(X_{\mathfrak m})$ there exists
a positive divisor
$D \in {\rm Div}(X_{\mathfrak m})$ with $\deg D = \pi $ such that
$\varphi (D) = a$. Since $D - \pi P_0 \in {\rm Div}^0(X_{\mathfrak m})$
and $\varphi (D - \pi P_0) = \varphi (D) = a$, the map
$\varphi : {\rm Div}^0(X_{\mathfrak m}) \longrightarrow
{\rm Alb}^{an}(X_{\mathfrak m})$ is surjective.
It follows from Theorem 4.2 that
$${\rm Ker} \varphi = \{ (f) ; f \in {\rm Mer}(X)\;
\text{with $f \equiv c(\overline{S})\ {\rm mod}\ {\mathfrak m}$ 
for some multiconstant $c(\overline{S})$}\}.$$
Therefore we have $\overline{{\rm Div}^{0}(X_{\mathfrak m})}
\cong {\rm Alb}^{an}(X_{\mathfrak m})$.
\end{proof}

\begin{theorem}
The map $\varphi : (X \setminus S)^{(\pi )} \longrightarrow
{\rm Alb}^{an}(X_{{\mathfrak m}})$ is bimeromorphic.
\end{theorem}

\begin{proof}
Let $W$ be the set of all points at which the holomorphic
map $\varphi : (X \setminus S)^{(\pi )} \longrightarrow
{\rm Alb}^{an}(X_{{\mathfrak m}})$ is not degenerate.
Then it is obvious that $W$ is open and dense in
$(X \setminus S)^{(\pi )}$.
Take any $(P_1, \dots , P_{\pi }) \in W$.
We set $D := P_1 + \cdots + P_{\pi }$.
Then we have $H^0(X_{\mathfrak m}, \Omega _{\mathfrak m}(-D)) = \{ 0 \}$.
From Theorem 3.13 it follows that
$$\dim H^0(X_{\mathfrak m}, {\mathcal L}_{\mathfrak m}(D)) = 1.$$
\par
We show that $\varphi ^{-1}(\varphi (D)) = D$. Suppose that
there exists another $E \in {\rm Div}(X_{\mathfrak m})$ with
$\deg E = \pi $ and $E \geq 0$ such that $\varphi (D) = \varphi (E)$.
By Theorem 4.2 we have a 
meromorphic function $f$ on $X$ with
$f \equiv c(\overline{S})\ {\rm mod}\ {\mathfrak m}$ for some
multiconstant $c(\overline{S})$ such that $E - D = (f)$.
Since $D + (f) = E \geq 0$, we have $(f) \geq - D$.
Then $f \in H^0(X_{{\mathfrak m}}, {\mathcal L}_{{\mathfrak m}}(D))$.
However any function in $H^0(X_{{\mathfrak m}}, {\mathcal L}_{{\mathfrak m}}(D))$
must be constant for 
$\dim H^0(X_{{\mathfrak m}}, {\mathcal L}_{{\mathfrak m}}(D)) = 1$.
Therefore we have $E = D$. Thus we conclude that
$\varphi |W : W \longrightarrow \varphi (W)$ is biholomorphic.
The map $\varphi : (X \setminus S)^{(\pi )} \longrightarrow
{\rm Alb}^{an}(X_{{\mathfrak m}})$ is obviously proper.
Then $\varphi (W)$ is an open dense subset of
${\rm Alb}^{an }(X_{{\mathfrak m}})$.\par
Next we show that the map
$$\psi _W := (\varphi |W)^{-1} : \varphi (W) \longrightarrow
(X \setminus S)^{(\pi )}$$
is meromorphic through ${\rm Alb}^{an}(X_{{\mathfrak m}})$.
We denote by $G(\psi _W)$ and $G(\varphi )$ be graphs of
$\psi _W$ and $\varphi $ respectively. Then the closure
$\overline{G(\psi _W)}$ of $G(\psi _W)$ in
$(X \setminus S)^{(\pi )} \times {\rm Alb}^{an}(X_{{\mathfrak m}})$
coincides with $G(\varphi )$. Moreover, the projection
$p : \overline{G(\psi _W)} \longrightarrow
{\rm Alb}^{an}(X_{{\mathfrak m}})$ is proper for so is 
$\varphi : (X \setminus S)^{(\pi )} \longrightarrow
{\rm Alb}^{an}(X_{{\mathfrak m}})$.
Therefore $\varphi : (X \setminus S)^{(\pi )} \longrightarrow
{\rm Alb}^{an}(X_{{\mathfrak m}})$ is a bimeromorphic map.

\end{proof}

As in Section 5.6 we write
$${\rm Alb}^{an}(X_{\mathfrak m}) = {\mathbb C}^p \times ({\mathbb C}^{*})^q
\times {\mathfrak Q},$$
where ${\mathfrak Q}$ is an $r$-dimensional quasi-abelian variety of kind $0$ and
$p + q + r = \pi $.
Let ${\mathfrak Q} = {\mathbb C}^r/\Gamma _0 $, where $\Gamma _0$ is a 
discrete subgroup of ${\mathbb C}^r$ with ${\rm rank}\ \Gamma _0 = r + s$.
Then ${\mathfrak Q}$ is a fibre bundle over an $s$-dimensional abelian
variety $A_0$ with fibres $({\mathbb C}^{*})^{r-s}$.
Let $\overline{\mathfrak Q}$ be the standard compactification of ${\mathfrak Q}$.
Compactifying ${\mathbb C}^p \times ({\mathbb C}^{*})^q$ by
$({\mathbb P}^1)^{p+q}$, we obtain a compactification
$$\overline{{\rm Alb}^{an}(X_{\mathfrak m})} :=
({\mathbb P}^1)^{p+q} \times \overline{\mathfrak Q}$$
of ${\rm Alb}^{an}(X_{\mathfrak m})$. We call it the standard compacification
of ${\rm Alb}^{an}(X_{\mathfrak m})$.

\begin{remark}
The map $\varphi : X \setminus S \longrightarrow {\rm Alb}^{an}(X_{\mathfrak m})$
does not extend to a holomorphic map $\overline{\varphi } : X \longrightarrow
\overline{{\rm Alb}^{an}(X_{\mathfrak m})}$.
\end{remark}

\subsection{The universality of analytic Albanese varieties}
Let $G$ be a commutative complex Lie group. Then, by the theorem of
Remmert-Morimoto we have
$$G \cong {\mathbb C}^p \times ({\mathbb C}^{*})^q \times G_0,$$
where $G_0 = {\mathbb C}^r/\Gamma $ is a toroidal group with
${\rm rank}\, \Gamma = r+s$ $(1\leq s \leq r)$. The toroidal group
$G_0$ has the structure of principal $({\mathbb C}^{*})^{r-s}$-bundle
$\sigma : G_0 \longrightarrow {\mathbb T}$ over an $s$-dimensional
complex torus ${\mathbb T}$. Replacing fibres $({\mathbb C}^{*})^{r-s}$
with $({\mathbb P}^1)^{r-s}$, we obtain the associated
$({\mathbb P}^1)^{r-s}$-bundle $\overline{\sigma } : \overline{G_0}
\longrightarrow {\mathbb T}$. We call $\overline{G} = ({\mathbb P}^1)^{p+q}
\times \overline{G_0}$ the standard compactification of $G$.\par
The following theorem is considered as an analytic version of
Th\'eor\`eme 1 in \cite{ref16}.

\begin{theorem}
Let $X$ be a compact Riemann surface, and let $S$ be a finite
subset of $X$. We consider a commutative complex Lie group $G$
with the standard compactification $\overline{G}$. Let
$f : X \setminus S \longrightarrow G$ be a holomorphic map
which has the holomorphic extension
$\overline{f} : X \longrightarrow \overline{G}$. Then there
exists a modulus ${\mathfrak m}$ with support $S$ such that
$f$ admits ${\mathfrak m}$ for a modulus with respect to any
equivalence relation $R$ on $S$.
\end{theorem}

\begin{proof}
Let $n := \dim _{{\mathbb C}}G$. There exists a basis
$\{ \omega _1, \dots , \omega _n\}$ of the vector space of
holomorphic 1-forms on $G$ which are invariant by translation.
For any $i$ the pull-back $f^{*}\omega _i$ of $\omega _i$ is
a holomorphic 1-form on $X \setminus S$ which extends
meromorphically to $X$. Then, for any $P \in S$ we can take
a positive integer $n_P$ such that
$${\rm ord }_P(f^{*}\omega _i) \geq - n_P\quad
\text{for $1 \leq i \leq n$.}$$
We define a modulus ${\mathfrak m}$ with support $S$ by
${\mathfrak m} := \sum _{P \in S} n_P P$.\par
Consider any equivalence relation $R$ on $S$, and set
$\overline{S} := S/R$. It suffices to show that
$f((\varphi )) = 0$ for any $\varphi \in {\rm Mer}(X)$
with $\varphi \equiv c(\overline{S})\ {\rm mod }\ {\mathfrak m}$
for some multiconstant $c(\overline{S})$. The function
$\varphi $ gives a holomorphic map $\Phi : X \longrightarrow
{\mathbb P}^1$. Let ${\rm Trace }_{\Phi }(f^{*}\omega _i)$
be the trace of $f^{*}\omega _i$ by $\Phi $. Then
${\rm Trace }_{\Phi }(f^{*}\omega _i)$ is holomorphic on
${\mathbb P}^1 \setminus \varphi (S)$. By the same proof of
Assertion in Section 4.5, we can show that it is holomorphic
on ${\mathbb P}^1$, hence ${\rm Trace }_{\Phi }(f^{*}\omega _i) = 0$
$(i=1, \dots ,n)$.
On the other hand, we can define the trace
$\widetilde{f} := {\rm Trace}_{\Phi }(f) : {\mathbb P}^1 \setminus
\varphi (S) \longrightarrow G$ of $f$ by $\Phi $.
Since $\widetilde{f}^{*}\omega _i ={\rm Trace }_{\Phi }(f^{*}\omega _i)$
on ${\mathbb P}^1 \setminus \varphi (S)$, we have
$\widetilde{f}^{*}\omega _i = 0$ for $i=1, \dots ,n$. Therefore,
the tangent map of $\widetilde{f}$ vanishes. This means that
$\widetilde{f}$ is a constant map. We denote by
$(\varphi )_a$ the zeros of $\varphi - a$ for any
$a \in {\mathbb P}^1 \setminus \varphi (S)$. We note
$f((\varphi )_a) = \widetilde{f}(a)$. Since $\widetilde{f}$ is
a constant map, we have $\widetilde{f}(a) = \widetilde{f}(b)$
for any $a, b \in {\mathbb P}^1 \setminus \varphi (S)$.
Therefore we obtain
$$f((\varphi )) = f((\varphi )_0) - f((\varphi )_{\infty })
= \widetilde{f}(0) - \widetilde{f}(\infty ) = 0.$$
\end{proof}

\begin{theorem}[Universality Property]
Let $G$ be a commutative complex Lie group, and let $P_0$ be the base point
of the map $\varphi : X \setminus S \longrightarrow {\rm Alb}^{an}(X_{\mathfrak m})$.
Then, for any holomorphic map $\psi : X \setminus S \longrightarrow G$
which admits ${\mathfrak m}$ for a modulus there exists uniquely a
homomorphism $\Psi : {\rm Alb}^{an}(X_{\mathfrak m}) \longrightarrow G$
between complex Lie groups such that $\psi = \Psi \circ \varphi + g_0$,
where $g_0 = \psi (P_0)$.
\end{theorem}

\begin{proof}
We may assume $g_0 = 0$ without loss of generality. Any holomorphic
map $\psi : X \setminus S \longrightarrow G$ can be extended to a
group homomorphism $\psi : \overline{{\rm Div}^0(X_{\mathfrak m})}
\longrightarrow G$, for it admits ${\mathfrak m}$ for
a modulus. By Corollary 5.18 we have an isomorphism
$\overline{\varphi } : \overline{{\rm Div}^0(X_{\mathfrak m})}
\longrightarrow {\rm Alb}^{an}(X_{\mathfrak m})$.
Then there exists uniquely a group homomorphism $\Psi :
{\rm Alb}^{an}(X_{\mathfrak m}) \longrightarrow G$ such that
$\psi = \Psi \circ \overline{\varphi }$ on $\overline{{\rm Div}^0(X_{\mathfrak m})}$.
\par
It suffices to show that $\Psi $ is holomorphic.
We denote by $S\psi $ the extension of $\psi $ to
$(X \setminus S)^{(\pi )}$.
By Theorem 5.19 there exists an open dense subset $W$ of
$(X \setminus S)^{(\pi )}$ such that 
$\varphi |W : W \longrightarrow \varphi (W)$ is biholomorphic
and $\varphi (W)$ is an open dense subset of
${\rm Alb}^{an}(X_{{\mathfrak m}})$.
If we set $\tau := (\varphi |W)^{-1} : \varphi (W) \longrightarrow W$,
then $\Psi = S\psi \circ \tau $ on $\varphi (W)$. 
Then $\Psi $ is holomorphic on $\varphi (W)$. Since $\varphi (W)$ is
open and dense in ${\rm Alb}^{an}(X_{{\mathfrak m}})$, $\Psi $
is holomorphic on the whole of  ${\rm Alb}^{an}(X_{{\mathfrak m}})$.
\end{proof}

\section{Further Properties of Analytic Albanese Varieties}

\subsection{Curves with general singularities}
We extend the method developed in \cite{ref2} to the general
setting. Let $X$ be a compact Riemann surface of genus $g$,
and let $S = \{ P_1, \dots , P_s\}$ be a finite subset of
$X$. Considering an equivalence relation $R$ on $S$, we set
$\overline{S} = S/R$. Let ${\mathfrak m}$ be a modulus with
support $S$. Then we obtain a singular curve
$X_{{\mathfrak m} } = (X \setminus S) \cup \overline{S}$ with
the projection $\rho : X \longrightarrow X_{{\mathfrak m} }$.
The genus of $X_{{\mathfrak m} }$ is $\pi =
\dim H^0(X_{{\mathfrak m} }, \Omega _{{\mathfrak m} })$.
Let $\pi = g + \delta$. We take a homology basis
$\{ \alpha _1, \beta _1, \dots , \alpha _g, \beta _g\}$ of
$X$ and small circles $\{ \gamma _1, \dots \gamma _s\}$ as
in Fig.1.\par
\begin{figure}
\includegraphics[width=13cm,height=8cm, clip]{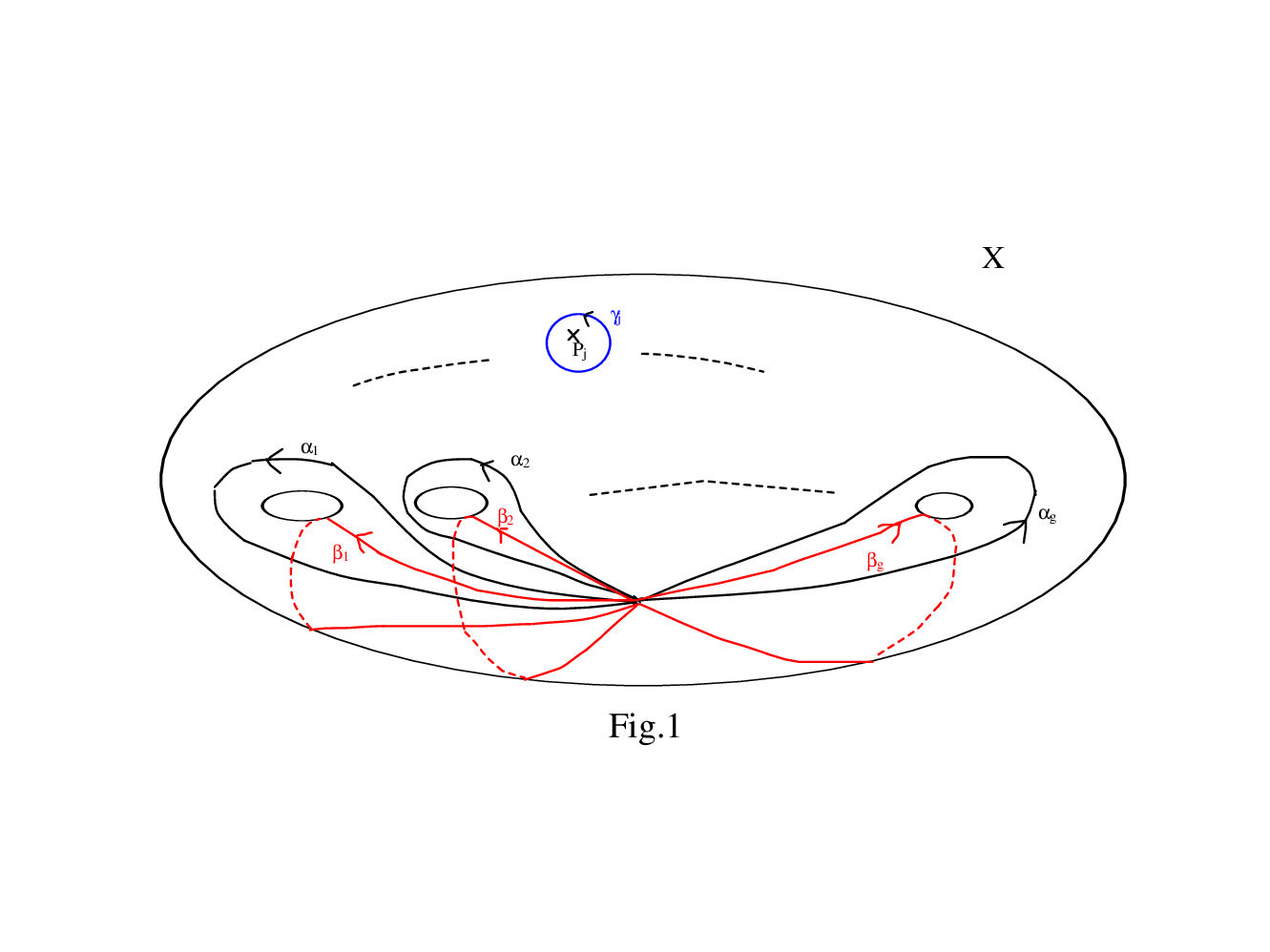}
\end{figure}
Then the set $\{ \alpha _1, \beta _1, \dots , \alpha _g, \beta _g,
\gamma _1,\dots ,\gamma _{s-1}\}$ forms a basis of
$H_1(X \setminus S, {\mathbb Z}) = H_1(X_{{\mathfrak m}}\setminus \overline{S},
{\mathbb Z})$.
Let $\{ \omega _1, \dots , \omega _{\pi }\}$ be a basis of
$H^0(X_{\mathfrak m}, \Omega _{\mathfrak m})$ such that
$\{ \rho ^{*}\omega _1, \dots , \rho ^{*}\omega _{g}\}$ is
a basis of $H^0(X, \Omega )$ and satisfies
$$\int _{\alpha _j}\rho ^{*}\omega _i = \delta _{ij},\quad
\tau = \left( \int _{\beta _j}\rho ^{*}\omega _i \right)_{1\leq i,j \leq g}
\in {\mathfrak S}_g.$$
Each $\rho ^{*}\omega _{g+i}$ is a meromorphic differential on
$X$ which is holomorphic on $X \setminus S$ and satisfies
$${\rm ord}_P(\rho ^{*}\omega _{g+i}) \geq - {\mathfrak m}(P)
\quad \mbox{for any $P \in S$}$$
and
$$\sum _{P \in \rho ^{-1}(Q)} {\rm Res}_P(\rho ^{*}(f\omega _{g+i}))
= 0$$
for any $f \in {\mathcal O}_{{\mathfrak m},Q}$ and any $Q \in \overline{S}$.
Renumbering $\{\rho ^{*}\omega _{g+i}\}$ if necessary, we can take
an integer $k$ with $0\leq k \leq \delta $ such that for
$i=1, \dots , k$ there exists $P \in S$ with
${\rm ord}_P(\rho ^{*}\omega _{g+i}) = -1$ and
for $k < i \leq \delta $ we have
${\rm Res}_P(\rho ^{*}\omega _{g+i})=0$ for any $P \in S$.
Moreover, we normalize as follows:\\
For $1 \leq i \leq k$ the form $\rho ^{*}\omega _{g+i}$ is
holomorphic except $\{ P_i, P_{k+i}\}$ and
${\rm ord}_{P_i}(\rho ^{*}\omega _{g+i}) =
{\rm ord}_{P_{k+i}}(\rho ^{*}\omega _{g+i}) = -1$ with
${\rm Res}_{P_i}(\rho ^{*}\omega _{g+i}) = 1/(2\pi \sqrt{-1})$. 
If $k<i\leq \delta $,
then ${\rm Res}_P(\rho ^{*}\omega _{g+i}) = 0$ for any
$P \in S$.
\begin{remark}
Let $\rho ^{*}\omega _{g+1}, \dots , \rho ^{*}\omega _{\pi }$ be
normalized as above. Take $i$ with $k < i \leq \delta $. Let $P \in S$.
We have the representation $\rho ^{*}\omega _{g+i} = f(t)dt$
by a local coordinate $t$ at $P$, where $f(t)$ is a meromorphic
function. Then the Laurent expansion of $f(t)$ at $t(P)=0$
does not contain a term $1/t$.
\end{remark}
Combining the normalization in the proof of Proposition 5.14 in
Section 5.6 with the above normalization,
we may assume that $\rho ^{*}\omega _1, \dots , \rho ^{*}\omega _{\pi}$
are further normalized as
$$\left(
\begin{array}{cc}
\left( \int _{\alpha _j}\rho ^{*}\omega _i\right)&
\left( \int _{\gamma _j}\rho ^{*}\omega _i\right)\\
\end{array}\right)
=
\left(
\begin{array}{cc}
I_g & 0 \\
0 & 
\left(
\begin{array}{ccc}
I_k & - I_k & 0\\
0 & 0 & 0\\
\end{array}\right)\\
\end{array}\right).$$
Let
$$B := \left(
\int _{\beta _j}\rho ^{*}\omega _{g+i}\right)_
{i=1, \dots , \delta ; j = 1, \dots , g}.$$
Then we have a period matrix of ${\rm Alb}^{an}(X_{\mathfrak m})$
as
$$\left(
\begin{array}{ccc}
I_g & \tau & 0\\
0 & B & \left(
\begin{array}{ccc}
I_k & - I_k & 0\\
 0 & 0 & 0\\
\end{array}\right)\\
\end{array}\right).$$
We consider a simply connected domain $D$ obtained by
cutting $X$ open along $\alpha _1, \beta _1, \dots , \alpha _g,
\beta _g$.
Let $D_0$ be the subdomain surrounded by $\partial D$ and
$\gamma _1, \dots , \gamma _s$ (see Fig.2).\par
\begin{figure}
\includegraphics[width=13cm,height=8cm, clip]{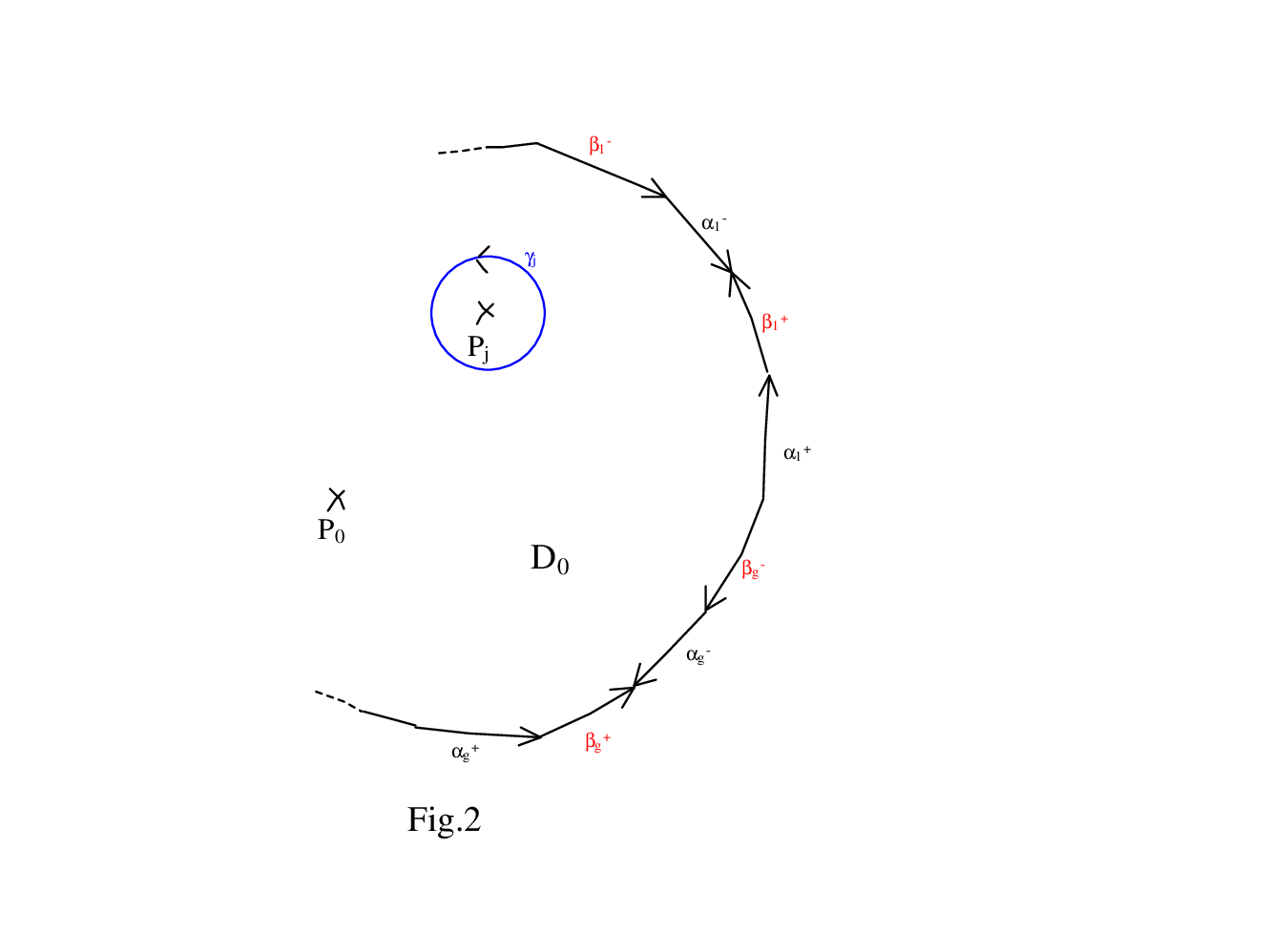}
\end{figure}
Fixing a point $P_0 \in D_0$, we define
$$h_j(z) := \int _{P_0}^{z}\rho ^{*}\omega _j$$
for any $z \in D$ and $j=1, \dots , g$.
Then $h_j$ is a holomorphic function on $D$ smoothly extended to
$\overline{D}$, which satisfies
$dh_j = \rho ^{*}\omega _j$.
Let $p_+ \in \alpha _i^{+}$.
We denote by $p_- \in \alpha _i^{-}$ the point corresponding to
$p_+$.
Then we have
$$h_j(p_-) = h_j(p_+) + \int _{\beta _i}\rho ^{*}\omega _j.$$
Similarly we have
$$h_j(q_-) = h_j(q_+) - \int _{\alpha _i}\rho ^{*}\omega _j$$
for $q_+ \in \beta _i^{+}$ and its correspondent point
$q_- \in \beta _i^{-}$.
Then we have for $j = 1, \dots , g$ and
$\ell = 1, \dots , \delta $
\begin{equation*}\label{xx}
\begin{split}
0 & = \iint _{D_0} \rho ^{*}\omega _j \wedge \rho ^{*}\omega _{g+\ell }\\
&= \iint _{D_0}d(h_j\rho ^{*}\omega _{g+\ell })\\
& = \int _{\partial D_0} h_j\rho ^{*}\omega _{g+\ell }\\
& = \sum _{i=1}^{g}\left( \int _{\alpha _i^{+}}h_j \rho ^{*}\omega _{g+\ell } +
\int _{\beta _i^{+}}h_j \rho ^{*}\omega _{g+\ell }\right)
- \sum _{i=1}^{g}\left( \int _{\alpha _i^{-}}h_j \rho ^{*}\omega _{g+\ell } +
\int _{\beta _i^{-}}h_j \rho ^{*}\omega _{g+\ell }\right)\\
& \quad - \sum _{\nu = 1}^{s} \int _{\gamma _{\nu }}h_j \rho ^{*}\omega _{g+\ell }\\
& = \sum _{i=1}^{g} \left( \int _{\alpha _i}h_j\rho ^{*}\omega _{g+\ell} +
\int _{\beta _i}h_j\rho ^{*}\omega _{g+\ell }\right)\\
& \quad - \sum _{i=1}^{g}\left[ \int _{\alpha _i}\left(
h_j + \int _{\beta _i}\rho ^{*}\omega _j\right) \rho ^{*}\omega _{g+\ell } +
\int _{\beta _i}\left( h_j - \int _{\alpha _i}\rho ^{*}\omega _j\right)
\rho ^{*}\omega _{g+\ell }\right]\\
& \quad - \sum _{\nu = 1}^{s} \int _{\gamma _{\nu }}h_j \rho ^{*}\omega _{g+\ell }\\
& = - \sum _{i=1}^{g}\left[ \left( \int _{\beta _i}\rho ^{*}\omega _j\right)
\left(\int _{\alpha _i}\rho ^{*}\omega _{g+\ell }\right) - 
\left( \int _{\alpha _i}\rho ^{*}\omega _j\right)
\left(\int _{\beta _i}\rho ^{*}\omega _{g+\ell }\right)\right]\\
& \quad - \sum _{\nu = 1}^{s} \int _{\gamma _{\nu }}h_j \rho ^{*}\omega _{g+\ell }\\
& =  \int _{\beta _j}\rho ^{*}\omega _{g+\ell } -
\sum _{\nu = 1}^{s} \int _{\gamma _{\nu }}h_j \rho ^{*}\omega _{g+\ell }.\\
\end{split}
\end{equation*}
Since
$$\int _{\gamma _{\nu }}h_j \rho ^{*}\omega _{g+\ell } =
2\pi \sqrt{-1}h_j(P_{\nu }) {\rm Res}_{P_{\nu }}(\rho ^{*}\omega _{g+\ell }),$$
we have
$$\sum _{\nu = 1}^{s}
\int _{\gamma _{\nu }}h_j \rho ^{*}\omega _{g+\ell } =
2\pi \sqrt{-1} \sum _{P \in S} h_j(P) {\rm Res}_{P}(\rho ^{*}\omega _{g+\ell }).$$
Therefore we obtain
$$\int _{\beta _j}\rho ^{*}\omega _{g+\ell } = 2\pi \sqrt{-1}
\sum _{P \in S}h_j(P) {\rm Res}_P(\rho ^{*}\omega _{g+\ell }).$$
By the normalization of $\rho ^{*}\omega _{g+1}, \dots , \rho ^{*}\omega _{\pi }$,
we obtain
$$\sum _{P \in S} h_j(P) {\rm Res}_P(\rho ^{*}\omega _{g+\ell }) = 
\frac{1}{2\pi \sqrt{-1}}(h_j(P_{\ell }) - h_j(P_{k+\ell }))$$
for $j=1, \dots , g$ if $1 \leq \ell \leq k$, and
$$\sum _{P \in S} h_j(P) {\rm Res}_P(\rho ^{*}\omega _{g+\ell }) = 0$$
for $j=1, \dots , g$ if $k+1 \leq \ell \leq \delta $.
Therefore we have
\begin{equation*}
\int _{\beta _j} \rho ^{*}\omega _{g + \ell } =
\begin{cases}
h_j(P_{\ell }) - h_j(P_{k+\ell }) &
\text{if $1 \leq \ell \leq k$},\\
0 & \text{if $k+1 \leq \ell \leq \delta $}
\end{cases}
\end{equation*}
for $j = 1, \dots , g$.
We set
$$H_k(S) :=\left(
\begin{array}{ccc}
h_1(P_1) - h_1(P_{k+1}) & \cdots &
h_g(P_1) - h_g(P_{k+1}) \\
\vdots &   &   \vdots \\
h_1(P_k) - h_1(P_{2k}) & \cdots &
h_g(P_k) - h_g(P_{2k}) \\
\end{array}
\right).$$
Then we have
$$B = \left(
\begin{array}{c}
H_k(S)\\
0
\end{array}\right).$$
Hence we obtain the following theorem.
\begin{theorem}
Let $X_{\mathfrak m}$ be as above. Then we have a period matrix 
$$\left(
\begin{array}{ccc}
I_g & \tau & 0 \\
0 &
H_k(S) & I_k\\
0 & 0 & 0
\end{array}\right)$$
of
${\rm Alb}^{an}(X_{\mathfrak m})$.
\end{theorem}

\subsection{Curves with nodes, the case $\pi = 2$}
Let ${\mathcal M}_g$ be the moduli space of non-singular curves of genus $g$.
The compactification $\widehat{{\mathcal M}_g}$ of ${\mathcal M}_g$ is obtained 
by adding stable curves of genus $g$ (\cite{ref7, ref8}).
An irreducible stable curve is a curve with only nodes as singularities.
We have a result for curves of general genus with nodes in \cite{ref2}.
We consider the case $\pi = 2$ in this section.
We show that an analytic Albanese variety may have many algebraic
structures. \par
Let $X$ be a compact Riemann surface of genus 1. Taking distinct points
$P_1$ and $P_2$ in $X$, we set $S = \{ P_1, P_2\}$. Let ${\mathfrak m}$ be a modulus
with support $S$ defined by 
${\mathfrak m}(P_1) = {\mathfrak m}(P_2) = 1$.
We identify $P_1$ with $P_2$ and put $\overline{S} = \{ Q \}$.
Then we obtain a singular curve
$X_{{\mathfrak m} } = (X \setminus S) \cup \overline{S}$
with the only node $Q$.
Let $\{ \alpha , \beta \}$ and $\{ \gamma _1, \gamma _2 \}$ be a homology
basis of $X$ and small circles respectively as in Fig.1 in the preceding
section. We take a basis $\{ \omega , \eta \}$ of 
$H^0(X_{{\mathfrak m} }, \Omega _{{\mathfrak m} })$ such that
$\rho ^{*} \omega $ generates $H^0(X, \Omega )$, where
$\rho : X \longrightarrow X_{{\mathfrak m} }$ is the projection.
Let ${\mathcal H}$ be the upper half plane. We may assume that $\omega $ is
normalized as 
$$\int _{\alpha }\rho ^{*}\omega = 1 \quad\mbox{and\quad
$\int _{\beta }\rho ^{*}\omega = \tau \in {\mathcal H}$}.$$
Furthermore we can take $\eta $ satisfying
$$\int _{\alpha }\rho ^{*}\eta = 0 \quad\mbox{and\quad
$\displaystyle{{\rm Res}_{P_1}(\rho ^{*}\eta ) = 
\frac{1}{2\pi \sqrt{-1}}}$}.$$
Then we have necessarily
${\rm Res}_{P_2}(\rho ^{*}\eta )= -1/(2\pi \sqrt{-1})$.
In this case we have
$$\left(
\begin{array}{cc}
\int _{\alpha }\rho ^{*}\omega & \int _{\gamma _1 }\rho ^{*}\omega \\
\int _{\alpha }\rho ^{*}\eta & \int _{\gamma _1 }\rho ^{*}\eta \\
\end{array}\right) =
\left(
\begin{array}{cc}
1 & 0 \\
0 & 1 \\
\end{array}\right) .$$
Let $D$ and $D_0$ be as in the preceding section.
We fix $P_0 \in D_0$.
Define
$$h(z) := \int _{P_0}^{z} \rho ^{*}\omega $$
for any $z \in D$. We set
$$H(S) := h(P_1) - h(P_2).$$
By Theorem 6.1, a period matrix of ${\rm Alb}^{an}(X_{{\mathfrak m} })$ is
$$\left(
\begin{array}{ccc}
0 & 1 & \tau \\
1 & 0 & H(S) \\
\end{array}\right).$$
This means that the structure of ${\rm Alb}^{an}(X_{{\mathfrak m} })$ is
determined by the relation of $P_1$ and $P_2$. The analytic Albanese variety
${\rm Alb}^{an}(X_{{\mathfrak m} })$ is one of the following cases:\\
(i) ${\rm Alb}^{an}(X_{{\mathfrak m} }) = {\mathbb T} \times {\mathbb C}^{*}$,
where ${\mathbb T}$ is a complex torus of dimension 1;\\
(ii) ${\rm Alb}^{an}(X_{{\mathfrak m} })$ is a 2-dimensional non-compact quasi-abelian
variety.
\begin{remark}
In the case (i) ${\mathbb T}$ depends on $X_{{\mathfrak m}}$ and is not isomorphic to $J(X)$ in general.
\end{remark}
\par
We consider the isomorphic classes of $\{{\rm Alb}^{an}(X_{{\mathfrak m}})\}$.
We identify $X$ with ${\mathbb T}_{\tau } := {\mathbb C}/\Gamma _{\tau}$,
where $\Gamma _{\tau } = {\mathbb Z} + \tau {\mathbb Z}$.
Let $z$ be a coordinate of ${\mathbb C}$. 
We denote by $[z]$ the equivalence class of $z$ modulo $\Gamma _{\tau }$.
From this identification it follows that
$\rho ^{*} \omega = dz$.
Let ${\mathcal F} = \{ a = s + t\tau ; 0 \leq s, t < 1 \}$ be the
fundamental parallelogram of $X = {\mathbb T}_{\tau }$. 
We can take $z_1, z_2 \in {\mathcal F}$ such that $P_1 = [z_1]$
and $P_2 = [z_2]$. Then we have
$$H(S) = z_1 - z_2.$$
For any $a \in {\mathcal F} \setminus \{ 0\}$ 
there exist $P_1=[z_1]$ and $P_2=[z_2]$ with $P_1 \not= P_2$
such that $z_1 - z_2 = a$. Then we can take a period
matrix of each ${\rm Alb}^{an}(X_{{\mathfrak m} })$ as follows
$$P_a = \left(
\begin{array}{ccc}
0 & 1 &\tau \\
1 & 0 & a \\
\end{array}\right),\quad
a \in {\mathcal F} \setminus \{ 0 \}.$$
We denote by $\Gamma _a$ a discrete subgroup of ${\mathbb C}^2$ with period matrix $P_a$.
Let $G_a := {\mathbb C}^2/\Gamma _a$.
\begin{lemma}
The quotient group $G_a$ is not toroidal if and only if
$$a = r + q\tau,\quad r, q \in {\mathbb Q}\quad
\mbox{and\quad $0\leq r, q <1$}.$$
\end{lemma}
\begin{proof}
It is well-known (cf. \cite{ref4, ref5}) that $G_a$ is toroidal if and only if there exists no
$\sigma = { }^t(\sigma _1, \sigma _2) \in {\mathbb Z}^2 \setminus \{ 0\}$ so that
$$\sigma _1 \tau + \sigma _2 a \in {\mathbb Z}.$$
The assertion is immediate from this condition.
\end{proof}
Next we consider the biholomorphic equivalence on $\{ X_{{\mathfrak m} } \}$.
Let $S' = \{ P'_1, P'_2 \}$ be another set of distinct two points in $X$.
Using $S'$,
we construct a singular curve $X_{{\mathfrak m} '}$ of genus 2 as above. Let
$\rho ' : X \longrightarrow X_{{\mathfrak m} '}$ be the projection.
\begin{lemma}
A map $f : X_{{\mathfrak m} } \longrightarrow X_{{\mathfrak m} '}$ is biholomorphic if
and only if there exists an automorphism
$\widetilde{f} : X \longrightarrow X$ with $\widetilde{f}(S) = S'$ such that
$f \circ \rho = \rho ' \circ \widetilde{f}$.
\end{lemma}
\begin{proof}
Suppose that a map $f : X_{{\mathfrak m} } \longrightarrow X_{{\mathfrak m} '}$
is biholomorphic.
Let $g := f|(X\setminus S)$. Then $g : X \setminus S \longrightarrow X \setminus S'$
is biholomorphic. By Riemann removability theorem, $g$ has the holomorphic
extension $\widetilde{f}$ to $X$ with $\widetilde{f}(S) = S'$.\par
The converse is obvious.
\end{proof}
The following proposition is well-known (for example, see Chapter III,
Section 1, Proposition 1.12 in \cite{ref13}).
\begin{proposition}
Any automorphism $\varphi : X \longrightarrow X$ is induced from
a linear function $\Phi (z) = \gamma z + z_0$, where
$z_0 \in {\mathbb C}$ and
\begin{equation}
\gamma =\begin{cases}
\text{a $4^{th}$ root of unity}& \text{if $\Gamma _{\tau }$ is square,}\\
\text{a $6^{th}$ root of unity}& \text{if $\Gamma _{\tau }$ is hexagonal,}\\
\pm 1 & \text{otherwise.}\\
\end{cases}
\tag{$*$}
\end{equation}
\end{proposition}
\begin{proposition}
Let $X = {\mathbb T}_{\tau }$. Consider two singular curves 
$X_{{\mathfrak m} }$ and $X_{{\mathfrak m} '}$ of genus 2 with a node
constructed from $X$.
We assume that the supports of ${\mathfrak m}$ and ${\mathfrak m}'$
are $S = \{P_1 = [z_1], P_2 = [z_2] \}$ and 
$S' = \{P'_1 = [z'_1], P'_2 = [z'_2]\}$ respectively.
Then, $X_{{\mathfrak m} }$ and $X_{{\mathfrak m} '}$ are
biholomorphic if and only if
$$z'_1 - z'_2 \equiv \gamma (z_1 - z_2)\quad {\rm mod}\quad \Gamma _{\tau },$$
where $\gamma $ is a complex number possessing the property $(*)$ in Proposition 6.4.
\end{proposition}
\begin{proof}
Let $f : X_{{\mathfrak m} } \longrightarrow X_{{\mathfrak m} '}$ be a biholomorphic
map. By Lemma 6.3 we have an automorphism
$\widetilde{f} : X \longrightarrow X$ with $\widetilde{f}(S) = S'$ such that
$f \circ \rho  = \rho ' \circ \widetilde{f}$.
It follows from Proposition 6.4 that $\widetilde{f}$ is induced from
a linear function $F(z) = \gamma z + z_0$, where $\gamma $ possesses the
property $(*)$ in Proposition 6.4.
Since $\widetilde{f}(S) = S'$, $\widetilde{f}$ satisfies one of the following two cases:\\
(1) $\widetilde{f}(P_1) = P'_1$ and $\widetilde{f}(P_2) = P'_2$;\\
(2) $\widetilde{f}(P_1) = P'_2$ and $\widetilde{f}(P_2) = P'_1$.\\
In the case (1), we have
$$[z'_1 - z'_2] = P'_1 - P'_2 = \widetilde{f}(P_1) - \widetilde{f}(P_2)
= [\gamma (z_1 - z_2)].$$
Then we obtain
$$z'_1 - z'_2 \equiv \gamma (z_1 - z_2)\quad {\rm mod}\quad \Gamma _{\tau }.$$
Similarly we obtain
$$z'_1 - z'_2 \equiv - \gamma (z_1 - z_2)\quad {\rm mod}\quad \Gamma _{\tau }$$
in the case (2). 
\par
Conversely we assume that 
$$z'_1 - z'_2 \equiv \gamma (z_1 - z_2)\quad {\rm mod}\quad \Gamma _{\tau }.$$
Letting
$z_0 := z'_2 - \gamma z_2$, we define $G(z) := \gamma z + z_0$.
Then $G$ gives an automorphism
$\widetilde{g} : X \longrightarrow X$.
By the assumption we obtain
$$G(z_1) = \gamma z_1 + z_0 = \gamma (z_1 - z_2) + z'_2 \equiv z'_1\quad {\rm mod}\quad
\Gamma _{\tau }.$$
Then we have $\widetilde{g}(P_1) = P'_1$. 
Since $G(z_2) = z'_2$, we have
$\widetilde{g}(P_2) = P'_2$. Therefore $\widetilde{g}$ gives
a biholomorphic map $g : X_{{\mathfrak m} } \longrightarrow 
X_{{\mathfrak m} '}$.
\end{proof}
We give examples $X_{{\mathfrak m} }$ and $X_{{\mathfrak m} '}$ such that
${\rm Alb}^{an}(X_{{\mathfrak m} }) \cong {\rm Alb}^{an}(X_{{\mathfrak m} '})$ is a
2-dimensional non-compact quasi-abelian variety even though $X_{{\mathfrak m} }$
and $X_{{\mathfrak m} '}$ are not biholomorphic.
To do this, we have to study the equivalence of period matrices of toroidal groups in detail.
\par
Let $G = {\mathbb C}^n/\Gamma $ and $G' = {\mathbb C}^n/\Gamma '$ be toroidal groups
of ${\rm rank}\ \Gamma = {\rm rank}\ \Gamma ' = n+m$, whose period matrices are
$P$ and $P'$ respectively. Then $G$ and $G'$ are isomorphic if and only if
there exist $M \in GL(n, {\mathbb C})$ and $A \in GL(n+m, {\mathbb Z})$ such that
$P = MP'A$.
In this case we write $P \sim P'$.\par
Every 2-dimensional non-compact toroidal group $G = {\mathbb C}^2/\Gamma $ has a
period matrix in toroidal coordinates as follows
$$\left(
\begin{array}{ccc}
0 & 1 & \tau \\
1 & s & t \\
\end{array}\right),$$
where $\tau \in {\mathcal H}$ and $ s, t \in {\mathbb R}$ (see p.10 in \cite{ref5}).
Let
$$\left(
\begin{array}{ccc}
0 & 1 & \tau '\\
1 & s' & t' \\
\end{array}\right)$$
be a period matrix of another toroidal group $G'= {\mathbb C}^2/\Gamma '$.
If
$$\left(
\begin{array}{ccc}
0 & 1 & \tau \\
1 & s & t \\
\end{array}\right) \sim
\left(
\begin{array}{ccc}
0 & 1 & \tau '\\
1 & s' & t' \\
\end{array}\right),$$
then there exist $M \in GL(2, {\mathbb C})$ and
$$A = \left(
\begin{array}{ccc}
a_{11} & a_{12} & a_{13} \\
a_{21} & a_{22} & a_{23} \\
a_{31} & a_{32} & a_{33} \\
\end{array}\right)
\in GL(3, {\mathbb Z})$$
such that
$$\left(
\begin{array}{ccc}
0 & 1 & \tau \\
1 & s & t \\
\end{array}\right) = M
\left(
\begin{array}{ccc}
0 & 1 & \tau '\\
1 & s' & t' \\
\end{array}\right) A.$$
We set
$$\begin{cases}
A_{11} = a_{21} + a_{31}\tau ',\\
A_{12} = a_{22} + a_{32}\tau ',\\
A_{13} = a_{23} + a_{33}\tau ',\\
A_{21} = a_{11} + a_{21}s' + a_{31}t',\\
A_{22} = a_{12} + a_{22}s' + a_{32}t',\\
A_{23} = a_{13} + a_{23}s' + a_{33}t'.\\
\end{cases}
$$
By a straight calculation we see that $M$ is written by $\{ A_{ij}\} $ and
$$\begin{cases}
s = \frac{A_{22}}{A_{21}},\\
t = \frac{A_{23}}{A_{21}},\\
\tau = \frac{A_{11} A_{23} - A_{13} A_{21}}
{A_{11} A_{22} - A_{12} A_{21}}.\\
\end{cases}$$
\par
Now we consider the following period matrices
$$P_a = \left(
\begin{array}{ccc}
0 & 1 & \tau \\
1 & 0 & a \\
\end{array}\right),\quad
P_b = \left(
\begin{array}{ccc}
0 & 1 & \tau \\
1 & 0 & b \\
\end{array}\right),\quad
a, b \in {\mathbb R} \setminus {\mathbb Q}.$$
Without loss of generality, we may assume
$0 < a, b < 1$.
Applying the above result to this case, we obtain that
$P_a \sim P_b$ if and only if there exists
$A \in GL(3, {\mathbb Z})$ such that
$$\begin{cases}
A_{22} = 0,\\
a = \frac{A_{23}}{A_{21}},\\
\tau = \frac{A_{11} A_{23} - A_{13} A_{21}}{- A_{12} A_{21}}.\\
\end{cases}$$
We see that all of such $A$'s are
$$\pm \left(
\begin{array}{ccc}
1 & 0 & 0 \\
0 & 1 & 0 \\
0 & 0 & 1\\
\end{array}\right)
\quad
\text{and\quad $\pm
\left(
\begin{array}{ccc}
1 & 0 & 1 \\
0 & -1 & 0 \\
0 & 0 & -1\\
\end{array}\right)$},$$
by an elementary argument.
If we take
$$A = \pm \left(
\begin{array}{ccc}
1 & 0 & 1 \\
0 & -1 & 0 \\
0 & 0 & -1\\
\end{array}\right),$$
then we obtain
$$\left(
\begin{array}{ccc}
0 & 1 & \tau \\
1 & 0 & 1-b\\
\end{array}\right) \sim
\left(
\begin{array}{ccc}
0 & 1 & \tau \\
1 & 0 & b\\
\end{array}\right) .$$
\par
\begin{example}
Suppose that $\Gamma _{\tau }$ is neither square nor hexagonal.
Let $S = \{ P_1, P_2 \}$ be a subset of $X = {\mathbb T}_{\tau }$
with $P_1 = [z_1]$ and $P_2 = [z_2]$.
We assume that
$$z_1 - z_2 \in {\mathbb R} \setminus {\mathbb Q},\quad
0 < z_1 - z_2 < 1, \quad z_1 - z_2\not= \frac{1}{2}.$$
We take $S' = \{ P'_1= [z'_1], P'_2=[z'_2]\}$ such as
$$z'_1 - z'_2 = 1 - (z_1 - z_2).$$
We construct $X_{{\mathfrak m} }$ and $X_{{\mathfrak m} '}$ from
$S$ and $S'$ respectively. Since
$z_1 - z_2 \not\equiv \pm (z'_1 - z'_2)\ {\rm mod}\ \Gamma _{\tau }$,
$X_{{\mathfrak m} }$ and
$X_{{\mathfrak m} '}$ are not biholomorphic by Proposition 6.5. However
${\rm Alb}^{an}(X_{{\mathfrak m} }) \cong {\rm Alb}^{an}(X_{{\mathfrak m} '})$
and it is a non-compact quasi-abelian variety,
for
$$\left(
\begin{array}{ccc}
0 & 1 & \tau \\
1 & 0 & z'_1 - z'_2\\
\end{array}\right) \sim
\left(
\begin{array}{ccc}
0 & 1 & \tau \\
1 & 0 & z_1 - z_2\\
\end{array}\right) $$
by the choice of $S$ and $S'$.
\end{example}
\subsection{Curves with cusps}
Let $X$ be a compact Riemann surface of genus $g$.
We set $S = \{ P \}$, $P \in X$.
We define a modulus ${\mathfrak m}$ with support $S$ by
${\mathfrak m}(P) = 2$.
Let $\overline{S} = \{ Q \}$, $Q = P$.
Then we obtain a singular curve $X_{{\mathfrak m} }$ with
the only cusp $Q$. The genus of $X_{{\mathfrak m} }$ is 
$\pi = g + 1$.
Let $\{ \alpha _1, \beta _1, \dots , \alpha _g, \beta _g\}$
be a canonical homology basis of $X$. We take a small circle $\gamma $
centered at $P$.
We denote by $\rho : X \longrightarrow X_{{\mathfrak m} }$
the projection. We can take a basis
$\{ \omega _1, \dots , \omega _g, \eta \}$ of
$H^0(X_{{\mathfrak m} }, \Omega _{{\mathfrak m} })$ such that
$\{ \rho ^{*}\omega _1, \dots , \rho ^{*}\omega _g\}$ is a
basis of $H^0(X, \Omega )$ and it is normalized as
$$\int _{\alpha _j}\rho ^{*}\omega _i = \delta _{ij},\quad
\tau = \left(\int _{\beta _j}\rho ^{*}\omega _i\right) \in
{\mathfrak S}_g$$
and 
$$\int _{\alpha _j}\rho ^{*}\eta = 0,\quad
j = 1, \dots , g.$$
It is obvious that $\rho ^{*}\eta $ is a meromorphic
differential with ${\rm Res}_P(\rho ^{*}\eta ) = 0$.
Then we see that $k=0$, $C=0$ and 
$$\int _{\beta _j}\rho ^{*}\eta = 2\pi \sqrt{-1}
h_j(P){\rm Res}_P(\rho ^{*}\eta ) = 0$$
in Section 6.1.
Therefore a period matrix of ${\rm Alb}^{an}(X_{{\mathfrak m} })$ is
$$\left(
\begin{array}{cc}
I_g & \tau \\
0 & 0 
\end{array}\right).$$
Thus we obtain the following theorem.
\begin{theorem}
Let $X_{{\mathfrak m} }$ be a singular curve whose
singularity is the only cusp.
If $X$ is the normalization of $X_{{\mathfrak m} }$, then we have
$${\rm Alb}^{an}(X_{{\mathfrak m} }) \cong
J(X) \times {\mathbb C}.$$
\end{theorem}
\begin{remark}
If the genus of $X$ is 1, then for any two points $P$ and $P'$ in $X$
there exists an automorphism $f : X \longrightarrow X$ with
$f(P) = P'$. 
Let $X_{{\mathfrak m}}$ and $X_{{\mathfrak m}'}$ be
singular curves with the only cusp constructed from $P$ and $P'$
respectively.
Then $X_{{\mathfrak m} } \cong X_{{\mathfrak m}'}$.
\par
Next we assume that $X$ is a compact Riemann surface of
genus $g \geq 2$.
Then the number of automorphisms of $X$ is at most
$84(g - 1)$ by Hurwitz' theorem.
Fix a singular curve $X_{{\mathfrak m}}$ with the only cusp
constructed from $P \in X$. Then,
there exist infinitely many
$X_{{\mathfrak m}'}$ whose singularity is the only cusp $Q'$ such that 
$X_{{\mathfrak m} } \not\cong X_{{\mathfrak m}'}$.
However we have ${\rm Alb}^{an}(X_{{\mathfrak m}}) \cong
{\rm Alb}^{an}(X_{{\mathfrak m}'})$ by the above theorem.
\end{remark}

\section{Meromorphic Function Fields}
\subsection{Non-singular case}
Let $X$ be a compact Riemann surface of genus $g$.
We stated some properties of the Jacobi variety $J(X)$ of $X$ in Section 4.1.
In this section we state another viewpoint of $J(X)$ considering the connection between
the two meromorphic function fields ${\rm Mer}(X)$ and
${\rm Mer}(J(X))$ on $X$ and on $J(X)$ respectively.\par
We follow the restatement of the classical argument (cf.\cite{ref11}) by \cite{ref14}.
The $n$-dimensional complex projective space ${\mathbb P}^n$ is
identified with the $n-$symmetric product $({\mathbb P}^1)^{(n)}$ of
${\mathbb P}^1$ by the map induced from the following
rational map
$$\tau : ({\mathbb P}^1)^n \ni (x_1, \dots , x_n) \longmapsto
(a_0 : a_1 :\cdots :a_n) \in {\mathbb P}^n,$$
$$\frac{a_1}{a_0} = - \sum _{i=1}^{n}x_i,\ 
\frac{a_2}{a_0} = \sum _{i<j}x_ix_j,\ \dots,\ 
\frac{a_n}{a_0} = (-1)^n \prod_{i=1}^{n}x_i,$$
where $(x_1, \dots , x_n)$ is the inhomogeneous coordinates of
$({\mathbb P}^1)^n$ and $(a_0: a_1: \cdots : a_n)$ is the homogeneous
coordinates of ${\mathbb P}^n$.
We have $x, y \in {\rm Mer}(X)$ such that 
${\rm Mer}(X) = {\mathbb C}(x,y)$. Then
we can take a non-zero
irreducible polynomial $f$ such that $f(x,y) = 0$.
Let $C$ be the closure of $\{ f =0\}$ in
${\mathbb P}^1 \times {\mathbb P}^1$. Then, by the normalization
$\mu _0 : X \longrightarrow C$ we can define a holomorphic map
$$\mu : X^g \longrightarrow C^g \subset ({\mathbb P}^1 \times {\mathbb P}^1)^g \cong
({\mathbb P}^1)^g \times ({\mathbb P}^1)^g.$$
Let $\sigma _0 : X^g \longrightarrow X^{(g)}$ be the canonical projection.
Since $J(X) \cong {\mathbb C}^g / \Gamma $, we have the projection
$\sigma : {\mathbb C}^g \longrightarrow J(X).$
A  period map $\varphi : X \longrightarrow J(X)$ is extended to a
birational map $\varphi : X^{(g)} \longrightarrow J(X).$
Then we have the following commutative diagram:
\def\spmapright#1{\smash{%
 \mathop{\hbox to 1.3cm{\rightarrowfill}}
  \limits^{#1}}}
\def\splongmapright#1{\smash{%
 \mathop{\hbox to 3.8cm{\rightarrowfill}}
  \limits^{#1}}}
\def\lmapdown#1{\Big\downarrow %
 \llap{$\vcenter{\hbox{$\scriptstyle#1 \,$ }}$}}
\def\rmapdown#1{\Big\downarrow %
 \rlap{$\vcenter{\hbox{$\scriptstyle #1$}} $ }}
\def\lmapup#1{\Big\uparrow %
 \llap{$\vcenter{\hbox{$\scriptstyle #1\,$}}$ }}
\def\lswmap#1{\swarrow %
 \llap{$\vcenter{\hbox{$\scriptstyle #1\,$}}$ }}
\def\rsemap#1{\searrow %
 \rlap{$\vcenter{\hbox{$\scriptstyle #1$}}$ }}
$$\begin{array}{ccc}
X^g & \spmapright{\mu}  C^g \subset ({\mathbb P}^1 \times {\mathbb P}^1)^g&
\cong ({\mathbb P}^1)^g \times ({\mathbb P}^1)^g\\
\lmapdown{\sigma _0} &      & \rmapdown{\tau \times \tau }\\
X^{(g)}& \splongmapright{\hat{\mu}}& {\mathbb P}^g \times {\mathbb P}^g \\
\lmapdown{\varphi} &   &  \lmapdown {\sigma _1}\quad  \rmapdown {\sigma _2}\\
J(X) &   &   {\mathbb P}^g\quad {\mathbb P}^g\\
\lmapup{\sigma } &   &   \\
{\mathbb C}^g &   &   \\
\end{array},
$$
where $\hat{\mu} : X^{(g)} \longrightarrow {\mathbb P}^g \times {\mathbb P}^g$
is the holomorphic mapping induced from $\mu $ and
$\sigma _i : {\mathbb P}^g \times {\mathbb P}^g \longrightarrow {\mathbb P}^g$
is the projection onto the $i$-th component for $i = 1,2$.
We note that $\hat{\mu}(X^{(g)})$ is projective algebraic and
$\hat{\mu} : X^{(g)} \longrightarrow \hat{\mu}(X^{(g)})$ is a birational map.
Thus we obtain rational maps 
$\wp ^x : {\mathbb C}^g \longrightarrow {\mathbb P}^g$ and
$\wp ^y : {\mathbb C}^g \longrightarrow {\mathbb P}^g$ defined by
$$\wp ^x := \sigma _1 \circ \hat{\mu}\circ \varphi ^{-1} \circ \sigma,\quad
\wp ^y := \sigma _2 \circ \hat{\mu}\circ \varphi ^{-1} \circ \sigma .$$
Using homogeneous coordinates of ${\mathbb P}^g$, we write
$$\wp ^x(z) = (1: \xi _1(z): \cdots : \xi _g(z)),\quad
\wp ^y(z) = (1: \eta _1(z): \cdots : \eta _g(z)).$$
Then we have 
${\rm Mer}(J(X)) = {\mathbb C}(\xi _1(z), \dots , \xi _g(z), \eta _1(z), \dots , \eta _g(z)).$
The relations of $\xi _1(z), \dots , \xi _g(z), \eta _1(z), \dots , \eta _g(z)$
are obtained by the elimination of $x_1, \dots , x_g, y_1, \dots , y_g$ from
the following equations
$$f(x_1, y_1) = 0, \dots , f(x_g, y_g) = 0,$$
$$\xi _1 = - \sum _{i=1}^{g} x_i,\  \xi _2 = \sum _{i<j}x_ix_j, \dots ,\ 
\xi _g = (-1)^g\prod _{i=1}^{g}x_i,$$
$$\eta _1 = - \sum _{i=1}^{g} y_i,\  \eta _2 = \sum _{i<j}y_iy_j, \dots ,\ 
\eta _g = (-1)^g\prod _{i=1}^{g}y_i.$$
These functions $\xi _1(z), \dots , \xi _g(z), \eta _1(z), \dots , \eta _g(z)$ are
the fundamental abelian functions belonging to ${\rm Mer}(X)$ in \cite{ref11}.
It is easy to check that the meromorphic function field
${\mathbb C}(\xi _1(z), \dots , \xi _g(z), \eta _1(z), \dots , \eta _g(z))$ admits
an algebraic addition theorem.
This is an abelian function field.

\subsection{Singular case}
In this section we develop the theory in the preceding section to
a singular curve $X_{{\mathfrak m} }$ defined by a modulus ${\mathfrak m}$ with
support $S$.\par
Let $\rho : X \longrightarrow X_{{\mathfrak m} }$ be the projection.
We denote by $\pi = g + \delta $ the genus of $X_{{\mathfrak m} },$
where $g$ is the genus of $X$.
We set $K := \rho ^{*}{\rm Mer}(X_{{\mathfrak m} }) \subset {\rm Mer}(X).$
Then there exist $x,y \in K$ such that $K = {\mathbb C}(x,y).$
Since $x$ and $y$ are algebraically dependent, we have a non-zero irreducible
polynomial $f$ such that $f(x,y) = 0.$ Let $C$ be the closure of
$\{ f = 0 \}$ in ${\mathbb P}^1 \times {\mathbb P}^1.$ We obtain the following
commutative diagram in the same manner as in the preceding section:
$$\begin{array}{ccc}
X^{\pi } & \spmapright{\mu}  C^{\pi } \subset ({\mathbb P}^1 \times {\mathbb P}^1)^{\pi }&
\cong ({\mathbb P}^1)^{\pi } \times ({\mathbb P}^1)^{\pi }\\
\lmapdown{\sigma _0} &      & \rmapdown{\tau \times \tau }\\
X^{(\pi )}& \splongmapright{\hat{\mu}}& {\mathbb P}^{\pi } \times {\mathbb P}^{\pi } \\
\cup &      &
  \lmapdown {\sigma _1}\quad  \rmapdown {\sigma _2}\\
(X \setminus S)^{(\pi )} &    &  {\mathbb P}^{\pi }\quad {\mathbb P}^{\pi }\\
\lmapdown{\varphi} &   &        \\
{\rm Alb}^{an}(X_{{\mathfrak m} }) &   &   \\
\lmapup{\sigma } &   &   \\
{\mathbb C}^{\pi } &   &   \\
\end{array}.
$$
By Theorem 5.19 the map $\varphi : (X \setminus S)^{(\pi )} \longrightarrow
{\rm Alb}^{an}(X_{{\mathfrak m} })$ is bimeromorphic. Therefore we can define
meromorphic maps
$\wp ^x : {\mathbb C}^{\pi } \longrightarrow {\mathbb P}^{\pi }$ and
$\wp ^y : {\mathbb C}^{\pi } \longrightarrow {\mathbb P}^{\pi }$ by
$$\wp ^x := \sigma _1 \circ \hat{\mu}\circ \varphi ^{-1} \circ \sigma \quad
\text{and}\quad
\wp ^y := \sigma _2 \circ \hat{\mu}\circ \varphi ^{-1} \circ \sigma .$$
If we represent $\wp ^x$ and $\wp ^y$ in homogeneous coordinates of
${\mathbb P}^{\pi }$ as
$$\wp ^x(z) = (1: \xi _1(z): \cdots : \xi _{\pi}(z)) \quad
\text{and}\quad
\wp ^y(z) = (1: \eta _1(z): \cdots : \eta _{\pi }(z)),$$
then $\xi _1(z), \dots , \xi _{\pi }(z), \eta _1(z), \dots , \eta _{\pi }(z) \in
{\rm Mer}({\rm Alb}^{an}(X_{{\mathfrak m} })).$
For the simplicity we write $A = {\rm Alb}^{an}(X_{{\mathfrak m} })$.
Let $[z]$ be a generic point of $A$, where $z \in {\mathbb C}^{\pi }$.
Then there exists uniquely $(P_1, \dots , P_{\pi }) \in
(X \setminus S)^{(\pi )}$ with $\varphi ((P_1, \dots , P_{\pi })) = [z]$.
In this case we have
$$\xi _1(z) = - \sum _{i=1}^{\pi }x(P_i),\ 
\xi _2(z) = \sum _{i<j}x(P_i)x(P_j), \dots ,\ 
\xi _{\pi }(z) = (-1)^{\pi }\prod _{i=1}^{\pi }x(P_i),$$
$$\eta _1(z) = - \sum _{i=1}^{\pi } y(P_i),\ 
\eta _2 (z) = \sum _{i<j}y(P_i)y(P_j), \dots ,\  
\eta _{\pi }(z) = (-1)^{\pi }\prod _{i=1}^{\pi }y(P_i).$$
We denote
$\kappa := {\mathbb C}(\xi _1(z), \dots , \xi _{\pi }(z),
\eta _1(z), \dots , \eta _{\pi }(z))\subset {\rm Mer}(A)$.
From the above representation it follows that any function
in $\varphi ^{*}\kappa $ extends meromorphically to $S$ and
$\varphi ^{*}\kappa = {\rm Mer}((X_{{\mathfrak m}} )^{(\pi )})
\subset {\rm Mer}(X^{(\pi )})$, where
$(X_{{\mathfrak m}} )^{(\pi )}$ is the symmetric product of
$X_{{\mathfrak m}} $ of degree $\pi $.
Therefore $\varphi ^{*}\kappa $ is finitely generated over
${\mathbb C}$ and ${\rm Trans}_{{\mathbb C}}\varphi ^{*}\kappa = \pi $.
Since $\varphi : (X \setminus S)^{(\pi )} \longrightarrow A$
is a bimeromorphic map, $\kappa $ is also finitely generated
over ${\mathbb C}$ and ${\rm Trans}_{{\mathbb C}}\kappa = \pi $.\par
We have the canonical representation
$$A = {\mathcal Q} \times ({\mathbb C}^{*})^q \times 
{\mathbb C}^p,$$
where ${\mathcal Q}$ is an $r$-dimensional quasi-abelian variety
of kind 0 and $p+q+r=\pi $ (see Section 5.6). Let
$$P =
\begin{pmatrix}
P_0 & 0 \\
0 & I_q \\
0 & 0
\end{pmatrix}$$
be a period matrix of $A$, where
$$P_0 =
\begin{pmatrix}
0 & I_g & \tau \\
I_{r-g} & R_1 & R_2
\end{pmatrix}$$
is a period matrix of ${\mathcal Q} = {\mathbb C}^r/\Gamma _0$ in 
the second normal form. We may assume that basis
$\omega _1, \dots ,\omega _{\pi }$ of
$H^0(X_{{\mathfrak m}} , \Omega _{{\mathfrak m}} )$ are
normalized as in Section 6.1. We take toroidal coordinates
$$z = (z',z'') = (z_1, \dots , z_g; z_{g+1}, \dots , z_r)$$
of ${\mathbb C}^r$. If $\zeta = (\zeta _1, \dots ,\zeta _q)$ and
$w = (w_1, \dots ,w_p)$ are coordinates of ${\mathbb C}^q$ and
${\mathbb C}^p$ respectively, then $(z,\zeta ,w)$ are coordinates
of ${\mathbb C}^{\pi }$. Any $\gamma \in \Gamma _0$ has the
representation $\gamma = (\gamma ', \gamma '')$ in toroidal
coordinates. We set
$$
\begin{pmatrix}
I_g & \tau \\
R_1 & R_2
\end{pmatrix}
= (\gamma _1, \dots ,\gamma _{2g}).$$
We denote by $\Lambda _{\tau }$ the lattice in ${\mathbb C}^g$
generated by $(\gamma _1', \dots , \gamma _{2g}') =
(I_g\ \tau )$. There exists an isomorphism from $\Lambda _{\tau }$
into $\Gamma _0$ defined by
$$\iota (\lambda ) :=
\sum _{j=1}^{2g}k_j \gamma _j$$
for any $\lambda = \sum _{j=1}^{2g}k_j \gamma _j'$.
We write ${\bf e}(t) = \exp (2\pi \sqrt{-1}t)$ for $t \in {\mathbb C}$.
Furthermore, we set ${\bf e}(z'') := {\bf e}(z_{g+1}) \cdots
{\bf e}(z_r)$ for $z'' = (z_{g+1}, \dots ,z_r)$. We define a
homomorphism
$$\psi : \Lambda _{\tau } \longrightarrow
{\mathbb C}_1^{\times } = \{ t \in {\mathbb C} ; |t| = 1 \}$$
by $\psi (\lambda ) := {\bf e}(\iota (\lambda )'')$
for $\lambda \in \Lambda _{\tau }$. We set $m := r - g$.
Let $\alpha \in {\mathbb Z}^m$. If we define
$\psi ^{-\alpha }(\lambda ) := \psi (\lambda )^{-\alpha }$,
then $\psi ^{-\alpha } : \Lambda _{\tau } \longrightarrow
{\mathbb C}_1^{\times }$ is also a homomorphism.\par
Let $\overline{A} = \overline{{\mathcal Q}} \times ({\mathbb P}^1)^q
\times ({\mathbb P}^1)^p$ be the standard compactification
of $A$. We set
$${\rm Mer}(\overline{A})|_A := \{ f|_A ;
f \in {\rm Mer}(\overline{A})\},$$
where $f|_A$ is the restriction of $f$ onto $A$.

\begin{proposition}
Any $F \in \varphi ^{*}({\rm Mer}(\overline{A})|_A)$ extends
meromorphically to $X^{(\pi )}$. Then we can consider
$\varphi ^{*}({\rm Mer}(\overline{A})|_A)$ as a subfield
of ${\rm Mer}(X^{(\pi )})$.
\end{proposition}

\begin{proof}
We use the above notations. First we consider
$g \in {\rm Mer}(\overline{{\mathcal Q}})|_{{\mathcal Q}}$.
There exists a theta factor $\varrho _0$ of an abelian
variety ${\mathbb C}^g/\Lambda _{\tau }$ such that if we
denote by $\Theta (\psi ^{-\alpha })$ the space of all theta
functions for $\psi ^{-\alpha }\varrho _0$, then we have the
representation of $g$ as
$$g(z) = \frac{\displaystyle{\sum _{\text{finite sum }}D_{\beta }(z')
{\bf e}(z'')^{\beta }}}{\displaystyle{\sum _{\text{finite sum }}
C_{\alpha }(z'){\bf e}(z'')^{\alpha }}},\quad
C_{\alpha } \in \Theta (\psi ^{-\alpha }),\ 
D_{\beta }\in \Theta (\psi ^{-\beta })$$
(\cite{ref3}).\par
Any $f \in {\rm Mer}(\overline{A})|_A$ is a rational function
of functions in ${\rm Mer}(\overline{{\mathcal Q}})|_{{\mathcal Q}}$,
${\bf e}(\zeta _1), \dots , {\bf e}(\zeta _q)$ and
$w_1, \dots ,w_p$
(\cite{ref3}). For any $(P_1, \dots , P_{\pi }) \in
(X \setminus S)^{(\pi )}$ we have
$$
\varphi ((P_1, \dots ,P_{\pi })) = \left[
\left( \sum _{i=1}^{\pi }\int _{P_0}^{P_i}\rho ^{*}\omega _1,
\dots , \sum _{i=1}^{\pi }\int _{P_0}^{P_i}
\rho ^{*}\omega _{\pi }\right) \right].$$
Since $\rho ^{*}\omega _1, \dots , \rho ^{*}\omega _g$ are
holomorphic forms on $X$,
$\displaystyle{\int _{P_0}^{P}\rho ^{*}\omega _i}$ is holomorphic 
on $X$ for $i = 1, \dots ,g$. From Remark in Section 6.1 it
follows that $\displaystyle{\int _{P_0}^{P}\rho ^{*}\omega _{g+i}}$
extends meromorphically to $S$ for $i$ with
$m+q < i \leq \delta $.\par
For $i$ with $1 \leq i \leq m+q$ we show that 
$\displaystyle{{\bf e}\left( \int _{P_0}^{P} \rho ^{*}
\omega _{g+i}\right)}$
extends meromorphically to $S$. By the normalization of
$\{ \rho ^{*}\omega _{g+i}\}$ we see that $\rho ^{*}\omega _{g+i}$
is holomorphic except two points $P, P' \in S$ with
$${\rm ord}_P(\rho ^{*}\omega _{g+i}) =
{\rm ord}_{P'}(\rho ^{*}\omega _{g+i}) = -1.$$
Let $P \in S$ be such a point. We assume
$${\rm Res}_P(\rho ^{*}\omega _{g+i}) =
\frac{1}{2\pi \sqrt{-1}}.$$
We can take a local coordinate $t$ on a neighborhood $U$ of $P$
with $t(P) = 0$ such that
$$\rho ^{*}\omega _{g+i} = \left( \frac{1}{2\pi \sqrt{-1}}
\frac{1}{t} + h(t)\right) dt$$
on $U$, where $h(t)$ is a holomorphic function on $U$.
Then we obtain
$${\bf e}\left( \int _{P_0}^{P} \rho ^{*}\omega _{g+i}\right)
= t {\bf e}(h_1(t))\quad
\text{on}\ U,$$
where $h_1(t)$ is a holomorphic function.
If
$${\rm Res}_P(\rho ^{*}\omega _{g+i}) = -
\frac{1}{2\pi \sqrt{-1}},$$
then we obtain
$${\bf e}\left( \int _{P_0}^{P} \rho ^{*}\omega _{g+i}\right)
= \frac{1}{t} {\bf e}(h_2(t))$$
in the same manner.\par
Thus we conclude that any function $F = \varphi ^{*}f$ with
$f \in {\rm Mer}(\overline{A})|_A$ extends meromorphically
to $X^{(\pi )}$.
\end{proof}

We think that $\kappa $ is a generalization of abelian function fields.
By the definition of $\kappa $ and Proposition 7.1, both $\kappa $ and
${\rm Mer}(\overline{A})|_A$ are subfields of $(\varphi ^{-1})^{*}{\rm Mer}
(X^{(\pi )})$. Unfortunately, we do not know the relation of $\kappa $
and ${\rm Mer}(\overline{A})|_A$ yet.



\end{document}